\documentclass[11pt,a4paper]{article}
\usepackage{caption}

\usepackage{epsf,epsfig,amsfonts,amsgen,amsmath,amstext,amsbsy,amsopn,amsthm,mathrsfs}
\usepackage{amsmath}
\usepackage[noadjust]{cite}
\usepackage{amsfonts,amsthm,amssymb,bm}
\usepackage{amsfonts}
\usepackage{graphics}
\usepackage{latexsym,bm}
\usepackage{amsfonts,amsthm,amssymb,bbding}
\usepackage{indentfirst}
\usepackage{graphicx}
\usepackage{color}

\newcommand{\Pset}{\mathcal{L}}
\usepackage{float}
\usepackage{tikz,enumerate}
\usetikzlibrary{patterns.meta, calc}

\usepackage{epsf,epsfig,amsfonts,amsgen,amsmath,amstext,amsbsy,amsopn,amsthm}

\usepackage{ebezier,eepic}
\usepackage{color}
\usepackage{multirow}
\usepackage{mathrsfs}
\usepackage{tikz}
\usepackage{enumerate}
\usepackage{cite}
\usepackage{amsthm}
\usepackage{thm-restate}
\usepackage{geometry}
\usepackage{comment}
\usepackage{setspace}
\usepackage{xcolor}
\usepackage{caption}
\usepackage{subcaption}
\usepackage{enumitem}
\usepackage{authblk}

\usepackage{geometry}
\newtheorem{dfn}{Definition}[section]
\newtheorem{thm}[dfn]{Theorem}
\newtheorem{obs}[dfn]{Observation}
\newtheorem{lem}[dfn]{Lemma}
\newtheorem{prop}[dfn]{Proposition}

\newtheorem{corollary}[dfn]{Corollary}
\newtheorem{conjecture}[dfn]{Conjecture}

\newcommand{\1}{{\uppercase\expandafter{\romannumeral1}}}
\newcommand{\2}{{\uppercase\expandafter{\romannumeral2}}}
\newcommand{\3}{{\uppercase\expandafter{\romannumeral3}}}
\newcommand{\4}{{\uppercase\expandafter{\romannumeral4}}}
\setlist[itemize]{itemsep=6pt, parsep=0pt}
\let\OLDthebibliography\thebibliography
\renewcommand\thebibliography[1]{
  \OLDthebibliography{#1}
  \setlength{\parskip}{0pt}
  \setlength{\itemsep}{7pt plus 0.3ex}
}

\def\AP{\mathrm{AP}}
\usepackage[colorlinks=true,anchorcolor=blue,filecolor=blue,linkcolor=red,urlcolor=blue,citecolor=blue,hypertexnames=false]{hyperref}
\usepackage{cleveref}

\begin{document}
\title{Dean's conjecture and cycles modulo $k$}

\author{
Yufan Luo$^{1}$
~~~~ Jie Ma$^{1,2}$
~~~~ Ziyuan Zhao$^{1}$
}

\date{}

\maketitle

\footnotetext[1]{School of Mathematical Sciences, University of Science and Technology of China, Hefei 230026, China.}
\footnotetext[2]{Yau Mathematical Sciences Center, Tsinghua University, Beijing 100084, China.}

\begin{abstract}
Dean conjectured three decades ago that every graph with minimum degree at least $k\ge 3$ contains a cycle whose length is divisible by $k$.
While the conjecture has been verified for $k\in \{3,4\}$, it remains open for $k\ge 5$.
A weaker version, also proposed by Dean, asserting that every $k$-connected graph contains a cycle of length divisible by $k$, was resolved by Gao, Huo, Liu, and Ma~\cite{gao2022unified} using the notion of admissible cycles.

In this paper, we resolve Dean’s conjecture for all $k\ge 6$.
In fact, we prove a stronger result by showing that every graph with minimum degree at least $k$ contains cycles of length $r \pmod k$ for every even integer $r$, unless every end-block belongs to a specific family of exceptional graphs, which fail only to contain cycles of length $2 \pmod k$.
We also establish a strengthened result on the existence of admissible cycles.
Our proof introduces two sparse graph families, called \emph{trigonal graphs} and \emph{tetragonal graphs}, which provide a flexible framework for studying path and cycle lengths and may be of independent interest.
\end{abstract}

\section{Introduction}\label{sec: intro}
\noindent The study of cycle lengths in graphs is a central and classical theme in graph theory; see \cite{bondy1996basic,verstraete2016extremal} for comprehensive treatments.
A particularly interesting problem in this area concerns the existence of cycles whose lengths are divisible by a given integer~$k$; see, for example, \cite{alon1989cycles,dean1988graphs,thomassen1988presence,thomassen1992even}.
The present work is motivated by the following beautiful conjecture of Dean (see Conjecture~7.4 in \cite{bondy1996basic}), which has remained open for three decades.

\vskip -1mm

\begin{conjecture}[Dean's conjecture]\label{conj:Dean}
    For every integer $k \geq 3$, every graph with minimum degree at least $k$ contains a cycle of length divisible by $k$.
\end{conjecture}

\vskip -1mm

The minimum degree condition in Conjecture~\ref{conj:Dean} is best possible, as complete bipartite graphs $K_{k-1,n}$ for odd $k$ and $n\geq k-1$ show that it cannot be weakened to $k-1$.
The conjecture has been verified for $k=3$ and $k=4$ by Chen and Saito \cite{chen1994graphs} and by Dean, Lesniak, and Saito \cite{dean1993cycles}, respectively, but remains open for all $k \ge 5$.
A weaker version, also proposed by Dean \cite{dean1988graphs}, asserting that every $k$-connected graph contains a cycle of length divisible by $k$, was resolved by Gao, Huo, Liu, and Ma \cite{gao2022unified} via a unified approach to related cycle problems.

In this paper, we prove Conjecture~\ref{conj:Dean} for all $k \geq 6$. 
To state our main result, let $\mathcal{H}_k$ denote the family of all graphs $H_{k,n;t}$, where $2 \le t \le k < n$, obtained from the complete bipartite graph $K_{k,n}$ by deleting $k-t$ edges incident to a single vertex in the part of size~$n$.

\vskip -1mm

\begin{thm}[Main Theorem]\label{main 1}
	For every integer $k\ge 6$, let $G$ be a graph with minimum degree at least $k$.
    Then exactly one of the following holds:
    \begin{itemize}
        \item[(1)] $G$ contains a cycle of length $r \pmod k$ for every even integer $r$;
        \item[(2)] $k$ is odd, every end-block of $G$ is isomorphic to a graph in $\{K_{k+1}, K_{k,k}\}\cup \mathcal{H}_{k}$, and every non-end-block contains no cycle of length $2 \pmod k$;
        \item[(3)] $k$ is even, every end-block of $G$ is isomorphic to $K_{k+1}$, and every non-end-block contains no cycle of length $2 \pmod k$.
    \end{itemize}
\end{thm}

Observe that every graph in $\{K_{k+1}, K_{k,k}\} \cup \mathcal{H}_k$ contains cycles of all even lengths modulo~$k$, with the sole exception of length $2 \pmod k$.
Consequently, we obtain the following immediate corollary, 
which resolves Conjecture~\ref{conj:Dean} affirmatively for all $k \ge 6$.\footnote{The case $k=5$ requires separate arguments and we will address this special case in forthcoming work.}

\begin{corollary}
    Let $k \ge 6$ be an integer. Then for every even integer $r \not\equiv 2 \pmod k$, every graph with minimum degree at least $k$ contains a cycle of length $r \pmod k$.
\end{corollary}

The general study of cycle lengths modulo~$k$ dates back to the 1970s, initiated by the work of Burr and Erd\H{o}s \cite{Erdos1976}.
For integers $\ell$ and $k$ with even integers in the residue class $\ell \pmod k$, let $c_{\ell,k}$ denote the smallest constant $c$ such that every $n$-vertex graph with at least $cn$ edges contains a cycle of length $\ell \pmod k$.
Erd\H{o}s \cite{Erdos1976} conjectured that $c_{\ell,k}$ exists for all $\ell$ when $k$ is odd,
and this was later confirmed by Bollob\'as \cite{bollobas1977cycles}.
Thomassen \cite{thomassen1983graph} further improved the bound to $c_{\ell, k} \le 4k(k+1)$ for all integers $k$ and even $\ell$.
In a subsequent work \cite{thomassen1988presence}, Thomassen provided a polynomial-time algorithm for finding a cycle of length divisible by $k$.
Resolving a conjecture of Thomassen \cite{thomassen1983graph}, Gao, Huo, Liu, and Ma \cite{gao2022unified} showed that for any $k \ge 3$, every graph with minimum degree at least $k+1$ contains cycles of all even lengths modulo~$k$ (the case of even $k$ was previously proved in \cite{liu2018cycle}).
From an extremal perspective, Sudakov and Verstra\"ete \cite{sudakov2017extremal} established a striking relation that for all $3\leq \ell<k$, the constant $c_{\ell,k}$ is upper bounded by the $k$-vertex Tur\'an number of $C_\ell$. 
To date, exact values of $c_{\ell, k}$ are known for very few pairs $(\ell, k)$; 
we refer to \cite{chen1994graphs,bai2025graphs, Dean1991Cycles,GYORI20267, gao2024two}. 
Diwan \cite{diwan2024cycles} extended this study to weighted graphs.

An effective approach to obtaining cycles with prescribed residues is to find a collection of {\it admissible} cycles, where the lengths of the cycles (or paths) form an arithmetic progression with common difference~$1$ or~$2$.
This notion was first introduced in \cite{gao2022unified} and has since been applied in a number of works on cycle length problems; see, for example,
\cite{lin2025strengthening,gao2021strengthening,gao2024two,chiba2023minimum,li2025cycles}.
A central result in this line of research, conjectured by Liu and Ma \cite{liu2018cycle} and proved by Gao et al. \cite{gao2022unified}, asserts that every graph with minimum degree at least $k+1$ contains $k$ admissible cycles.
We generalize this result by relaxing the minimum degree condition for all integers $k \ge 7$.

\begin{thm}\label{thm:exist admis}
    Let $k\ge 7$ and let $G$ be a graph with minimum degree at least $k$.
    Then  $G$ contains $k$ admissible cycles, unless every end-block of $G$ is isomorphic to a graph in $\{K_{k+1}, K_{k,k}\}\cup \mathcal{H}_{k}$.
\end{thm}

Our proofs employ several novel techniques, distinct from those in \cite{liu2018cycle,gao2022unified}, as discussed below.

\vskip 2mm

{\noindent \bf Proof Overview and New Tools.}
Our proof builds upon the approach of Gao et al.~\cite{gao2022unified}, which refine earlier results of Fan~\cite{fan2002distribution} and Liu–Ma~\cite{liu2018cycle}.
Specifically, the idea is to find a collection of $k$ admissible paths between any two given vertices $x,y$ in a suitable 2-connected graph $G$.
The key step is to identify a specified subgraph $H$ (called the {\it core} subgraph) such that, for some positive integer $t$, the following typically hold:
\begin{itemize}
\item[(a)] the core subgraph $H$ satisfies a robust spreading property, namely, $H$ contains $t$ admissible paths between many pairs of vertices; and
\item[(b)] the graph obtained from $G$ by contracting or deleting $V(H)$ satisfies suitable minimum degree conditions and yields $k+1-t$ admissible paths between prescribed pairs of vertices.
\end{itemize}
By concatenating these paths appropriately, one obtains the desired $k$ admissible paths in $G$.

In all previous works, such as Gao et al.~\cite{gao2022unified}, the core subgraphs were typically dense structures:
either complete graphs, complete bipartite graphs, or graphs closely resembling them.
The main technical contribution of this work is to go beyond this paradigm by introducing two new families of core graphs that can be very sparse (with average degree at most four; they can even be outer-planar) while still satisfying properties (a) and (b).
We call these families \textit{trigonal graphs} and \textit{tetragonal graphs}, which handle the non-bipartite and bipartite cases, respectively; a detailed discussion is provided in Section~3.
We would like to point out that the high-connectivity condition was essential in the proof of the weaker version of Dean's conjecture in \cite{gao2022unified}.
Owing to their sparsity, these core subgraphs can be found in graphs with lower connectivity, providing the crucial ingredient for our proof.
We believe that these two families are of independent interest and may also find applications in related problems.

\vskip 2mm

{\noindent \bf A Related Result.}
During a conference in Xi'an in June 2025, we learned that Bai, Grzesik, Li, and Prorok~\cite{bai2025cycle} independently obtained results related to our Theorems~\ref{main 1} and~\ref{thm:exist admis}.
More precisely, among other results, they~\cite{bai2025cycle} proved corresponding versions of Theorems~\ref{main 1} and~\ref{thm:exist admis} under the additional assumption that the graph is 2-connected, for every integer $k \ge 4$.
Their proofs are mainly based on the approach of Gao et al.~\cite{gao2022unified}.
Beyond the difference in proofs, we would like to emphasize that, in order to derive Conjecture~\ref{conj:Dean}, it is crucial to establish results of the type given in our main theorem that remove any connectivity assumption.

\vskip 2mm

{\noindent \bf Paper Organization.}
The remainder of this paper is organized as follows.  
In Section~2, we present the necessary definitions and preliminary results.  
Section~3 introduces two key families of graphs, namely \textit{trigonal graphs} and \textit{tetragonal graphs}, and establishes properties of paths within them.  
In Section~4, we introduce \textit{$k$-weak} graphs, reduce the proof of Theorem~\ref{main 1} to this class (see Theorem~\ref{thm:weak graph}), and establish several structural lemmas.  
The proof of Theorem~\ref{thm:weak graph} is then split between Sections~5 and~6, according to whether the host graph is non-bipartite or bipartite.  
At the end of Section~6, we also prove Theorem~\ref{thm:exist admis}.

\section{Preliminaries}
\subsection{Notations}
\noindent All graphs in this paper are finite, undirected, and simple. We use standard graph theory notation and terminology; see \cite{diestel2024graph}. 
The set of neighbors of a vertex $v$ in $G$ is denoted by $N_{G}(v)$, and the degree of $v$ is denoted by $\deg_{G}(v)=|N_{G}(v)|$. 
For $A,B\subseteq V(G)$ and $v\in V(G)$, we denote by $N_{A}(v)=N_G(v)\cap A$, $N_A(B)=\bigcup_{b\in B} N_A(b)$ and $\deg_{A}(v)=|N_{A}(v)|$, where we omit the subscript $G$, as the host graph $G$ will be clear in the context.
For a subgraph $H\subseteq G$, we simplify $N_{V(H)}(v)$ to $N_H(v)$ and $\deg_{V(H)}(v)$ to $\deg_H(v)$.
For a positive integer $k$, we write $\delta_k(G)$ for the {\it $k$-th minimum degree} of $G$, and abbreviate the minimum degree $\delta_1(G)$ to $\delta(G)$.
Recall the collection of graphs $\mathcal{H}_{k}=\{H_{k,n;t}:2\leq t\leq k<n\}$, we have $\delta(H_{k,n;t})=t$, $\delta_2(H_{k,n;t})=k$, and $H_{k,n;k}\simeq K_{k,n}$.
For $U\subseteq V(G)$, $G[U]$ denotes the subgraph of $G$ induced by $U$, and let $G-U:=G[V(G)\setminus U]$.
If $U=\{u\}$, we write $G-u$ for $G-\{u\}$.
We say that a graph $G'$ is obtained from $G$ by contracting $U$ into a vertex $u$ if $V(G')=(V(G)\setminus U)\cup \{u\}$ and $E(G')=E(G-U)\cup \{uv:v\in N_{G}(U)\}$.

We say that a vertex $v\in V(G)$ is a {\it cut-vertex} of $G$ if $G-v$ contains more components than $G$. 
A {\it block} in $G$ is a maximal connected subgraph that has no cut-vertex of its own (i.e., it is a maximal 2-connected subgraph, a bridge, or an isolated vertex).
An {\it end-block} of $G$ is a block containing at most one cut-vertex of $G$.
Note that if a connected graph $G$ of order at least three is not 2-connected, then $G$ contains at least two end-blocks.

For two positive integers $k\leq \ell$, we define $[k,\ell]=\{k,k+1,\dots, \ell\}$ and $[k]=[1,k]$.
For two integer sets $X$ and $Y$, we denote their set addition by $X+Y:=\{x+y:x\in X,y\in Y\}$.
We also write $k+X:=\{k+x:x\in X\}$.
Given two vertex subsets $A,B\subseteq V(G)$, a path $P=x_1\cdots x_t$ in $G$ is called an \textit{$(A,B)$-path} if $V(P)\cap A=\{x_1\}$ and $V(P)\cap B=\{x_t\}$.
If an $(A,B)$-path consists of a single edge, we call it an $(A,B)$-edge; we write $E(A,B)$ for the set of all $(A,B)$-edges.
We abbreviate $(\{a\},B)$-paths to $(a,B)$-paths and $(V(H),B)$-paths to $(H,B)$-paths if $H$ is a subgraph of $G$. 
For a subgraph $H\subseteq G$ and distinct vertices $u,v\in V(G)$, we write $\mathcal{P}_{u,v}^{H}$ for the set of all $(u,v)$-paths whose internal vertices belong to $H$, and $\Pset^H_{u,v}$ for the set of lengths of these paths. 
For vertices $x,y\in V(H)$, let $\mathrm{dist}_H(x,y)$ denote the length of a shortest $(x,y)$-path in $H$.

For distinct vertices $u,v\in V(G)$, let $G+uv$ (resp. $G-uv$) denote the graph with vertex set $V(G)$ and edge set $E(G)\cup \{uv\}$ (resp. $E(G)\setminus \{uv\}$). 
We refer to the triple $(G,u,v)$ as a {\it rooted graph} to implicitly fix two distinct vertices $u,v\in V(G)$.
The {\it minimum degree} of a rooted graph $(G,u,v)$, denoted by $\delta(G,u,v)$, is the minimum degree in $G$ of vertices in $V(G)\setminus\{u,v\}$. 
The {\it second minimum degree}, $\delta_2(G,u,v)$, is defined analogously.
We say that the rooted graph $(G,u,v)$ is 2-connected if $G+uv$ is 2-connected. 
For a connected graph $M$ and a block $B$ of $M$, we write $\mathrm{Cut}(B)$ for the set of cut-vertices of $M$ contained in $B$; we typically omit the notation for $M$ when it is clear from the context.
For brevity, we write a \textit{$k$-$\AP$} for an arithmetic progression with common difference~$k$.

We will frequently use the notions of consecutive and admissible sets for integers, paths, and cycles, which we now define formally.
\begin{dfn}
A set of integers is called \textit{consecutive} (resp. \textit{admissible}) if its elements form a $1$-$\AP$ (resp. $1$-$\AP$ or $2$-$\AP$).
A family of paths or cycles in a graph $G$ is called \textit{consecutive} (resp. \textit{admissible}) if the set of their lengths is consecutive (resp. admissible).
\end{dfn}

Throughout, we use the following basic fact to estimate the size of an admissible or consecutive cycle family obtained by concatenating paths from two path families.
\begin{obs}\label{obs:addition}
Let $X$ and $Y$ be admissible integer sets of size $s$ and $t$, respectively. Then:
\begin{itemize}
\item[(1)] $X + Y$ is admissible and has size at least $s + t - 1$.
\item[(2)] If either $X$ or $Y$ is consecutive, then so is $X + Y$.
\item[(3)] If $X$ is a 2-$\AP$, $Y$ is consecutive, and $t \geq 2$, then $X + Y$ is of size at least $2s + t - 2$.
\end{itemize}
\end{obs}

\subsection{Some useful lemmas}
\noindent We use the following result by Chiba and Yamashita~\cite{chiba2023minimum}, which provides a sufficient condition for the existence of $k$ admissible paths in a 2-connected rooted graph.

\begin{lem}\textup{(\cite{chiba2023minimum})}\label{lem:admis path}
	Let $k$ be a positive integer. If $(G,x,y)$ is a $2$-connected rooted graph with $|G|\geq 4$ and $\delta_{2}(G,x,y)\ge k+1$, then there exist $k$ admissible $(x,y)$-paths in $G$.
\end{lem}

Using this lemma, Chiba and Yamashita proved the following theorem on admissible cycles.

\begin{thm}\textup{(\cite{chiba2023minimum})}\label{Thm:32}
    For any integer $k\ge 2$, every graph $G$ on at least three vertices, having at most two vertices of degree less than $k+1$, contains $k$ admissible cycles.
\end{thm}

We then present a technical lemma regarding the distance between neighbors of vertices on a cycle. This result operates independently in general graphs but is essential for linking paths to the core subgraphs in the subsequent sections.

\begin{lem}\label{lem:short dist}
	Let $C$ be a cycle of length $s\ge 3$ in graph $G$, and let $u_1, u_2\in V(G-C)$ be two distinct vertices such that $\deg_C(u_1) > 0$ and $\deg_C(u_2) > 0$. 
	If neither $N_{C}(u_1)$ nor $N_{C}(u_2)$ contains two consecutive vertices of $C$, then the following statements hold.
	\begin{itemize}
		\item[(1)] If $N_C(u_1)\cap N_C(u_2)=\emptyset$, then there exist $v_1\in N_C(u_1)$ and $v_2\in N_C(u_2)$ such that $1\leq \mathrm{dist}_{C}(v_1,v_2)\leq \max\{1,\lfloor s/2\rfloor+2-\deg_C(u_1)-\deg_C(u_2)\}$;
		\item[(2)] If $\deg_{C}(u_{1})\ge 2$, then there exist $v_1\in N_C(u_1)$ and $v_2\in N_C(u_2)$ such that $1\leq\mathrm{dist}_{C}(v_1,v_2)\leq \max\{3,s/\deg_C(u_1),\lfloor s/2\rfloor+3-\deg_C(u_1)-\deg_C(u_2)\}$.
	\end{itemize}
\end{lem}

\begin{proof}
    Let $\ell:=\min\{\mathrm{dist}_{C}(v_1,v_2) : v_1\in N_C(u_1), v_2\in N_C(u_2), v_1\neq v_2\}$.
	Since $\mathrm{dist}_{C}(v_1,v_2)\leq s/2$ for any pair $v_1,v_2$, it suffices to prove the result assuming $\deg_C(u_1)+\deg_C(u_2)\geq 3$ and $\ell\geq 2$.
	Color the vertices in $N_C(u_1)$ red and those in $N_C(u_2)$ blue; vertices in $N_C(u_1) \cap N_C(u_2)$ receive both colors.
	These colored vertices divide $C$ into $t:=|N_C(u_1)\cup N_C(u_2)|$ subpaths $P_1,\dots, P_t$.
    Note that the endpoints of each path are colored, while their internal vertices are uncolored.
    Since neither $N_{C}(u_1)$ nor $N_{C}(u_2)$ contains consecutive vertices, $e(P_i) \geq 2$ if the endpoints share a color; otherwise, $e(P_i) \geq \ell$ by the definition of $\ell$.
	
	Firstly, consider (1). Suppose $N_C(u_1)\cap N_C(u_2)=\emptyset$; then $t=\deg_C(u_1)+\deg_C(u_2)\geq 2$.
	Since there are at least two $P_i$ whose endpoints have different colors, it follows that $s=|C|=\sum_{i\in [t]}e(P_i)\geq (t-2)\cdot 2+2\cdot \ell$ (using $\ell\geq 2$).
	Hence $\ell\leq s/2+2-t$, as desired.

    It remains to consider (2). 
    We may assume that $N_C(u_1) \cap N_C(u_2) \neq \emptyset$ and $\ell \ge 4$. 
    Let $r := |N_C(u_1) \cap N_C(u_2)| \ge 1$, i.e., there are $r$ vertices colored both red and blue. 
    Then $t + r = \deg_C(u_1) + \deg_C(u_2)$. 
    Let $p$ be the number of subpaths $P_i$ that have at least one endpoint in $N_C(u_1) \cap N_C(u_2)$.
    Note that each of such paths is of length at least $\ell$. 
    Since each of the $r$ common neighbors is an endpoint of exactly two such subpaths, and each of such path $P_i$ contains at most two endpoints in $N_C(u_1) \cap N_C(u_2)$, it follows that $p \ge r$. Equality holds if and only if $N_C(u_1) = N_C(u_2)$.

    If $N_C(u_1) = N_C(u_2)$, then $\deg_C(u_1)=\deg_C(u_2)=r$ and $s=|C|\geq r\cdot\ell$, implying $\ell\leq s/r=s/\deg_C(u_1)$.
	Otherwise, we have $p\geq r+1$.
    Since $C$ is split into subpaths $P_1,\dots, P_t$, where $t\geq \deg_C(u_1)\geq 2$.
    Note that at least $p$ of such paths have lengths no less than $\ell$. Hence, 
	\begin{align*}
		s&=|C|\geq p\cdot\ell+(t-p)\cdot 2\\
        &= (\ell-2)p+2(\deg_C(u_1)+\deg_C(u_2)-r)\\
		&\geq (\ell-2)(r+1)+2(\deg_C(u_1)+\deg_C(u_2)-r)\\
		&\ge (\ell-4)r+(\ell-2)+2(\deg_C(u_1)+\deg_C(u_2))\\
		&\geq(\ell-4)+(\ell-2)+2(\deg_C(u_1)+\deg_C(u_2)).
	\end{align*}
	Therefore, $\ell\leq s/2+3-\deg_C(u_1)-\deg_C(u_2)$, completing the proof of (2).
\end{proof}

\section{The core subgraphs: trigonal and tetragonal graphs}\label{sec:core graphs}
\noindent In this section, we introduce two graph families: \textit{trigonal graphs} and \textit{tetragonal graphs}. These structures provide the framework for constructing the core subgraphs in our proof.
Each graph in these families is equipped with a specific Hamiltonian cycle, referred to as its {\it boundary cycle}. 
The defining characteristic of these graphs is the rich structure of path lengths between their vertices. 
Specifically, we establish that the set of path lengths between any pair of vertices forms a long arithmetic progression with common difference 1 or 2 (i.e., an admissible set). 
This flexibility is the cornerstone of our main proof, enabling us to construct cycles of desired residues modulo $k$ when combined with paths from the remainder of the graph.

\subsection{Trigonal graphs}
\begin{dfn}[Trigonal graph]\label{trigonal}
    A trigonal graph $T$ is a non-bipartite outer-planar graph equipped with a Hamiltonian cycle $\partial T$, defined as the final graph $T_{n}$ (with $\partial T = \partial T_{n}$) of a finite sequence of trigonal graphs $T_{3},T_4,\ldots,T_{n}$ that satisfies the following properties:
	
	\begin{itemize}
		\item $T_{3}\simeq K_{3}$ and $\partial T_{3}=T_{3}$.
		
		\item For every $3\leq i \leq n-1$, $T_{i+1}$ is obtained from $T_{i}$ by adding a new vertex $x_i$ and a path $P_{i}:=a_{i}x_{i}b_{i}$ where $a_{i}b_{i}\in E(\partial T_{i})$; $\partial T_{i+1}$ is obtained from $\partial T_{i}$ by adding the path $P_{i}$ and deleting the edge $a_{i}b_{i}$.
	\end{itemize}
\end{dfn}

\begin{figure}
    \centering
\tikzset{every picture/.style={line width=0.75pt}} 

\begin{tikzpicture}[x=0.75pt,y=0.75pt,yscale=-1,xscale=1,scale=0.78]

\draw [line width=0.75]    (91.42,48.65) -- (57.14,103.27) ;
\draw [line width=0.75]    (91.42,48.65) -- (123.48,103.34) ;
\draw [line width=0.75]    (57.14,103.27) -- (123.48,103.34) ;
\draw  [fill={rgb, 255:red, 0; green, 0; blue, 0 }  ,fill opacity=1 ][line width=0.75]  (88.8,48.68) .. controls (88.79,47.23) and (89.95,46.05) .. (91.39,46.04) .. controls (92.84,46.03) and (94.02,47.19) .. (94.03,48.63) .. controls (94.04,50.07) and (92.88,51.25) .. (91.44,51.27) .. controls (90,51.28) and (88.82,50.12) .. (88.8,48.68) -- cycle ;
\draw  [fill={rgb, 255:red, 0; green, 0; blue, 0 }  ,fill opacity=1 ][line width=0.75]  (54.53,103.3) .. controls (54.51,101.85) and (55.67,100.67) .. (57.12,100.66) .. controls (58.56,100.65) and (59.74,101.81) .. (59.76,103.25) .. controls (59.77,104.69) and (58.61,105.88) .. (57.16,105.89) .. controls (55.72,105.9) and (54.54,104.74) .. (54.53,103.3) -- cycle ;
\draw  [fill={rgb, 255:red, 0; green, 0; blue, 0 }  ,fill opacity=1 ][line width=0.75]  (120.86,103.36) .. controls (120.85,101.92) and (122.01,100.74) .. (123.45,100.72) .. controls (124.9,100.71) and (126.08,101.87) .. (126.09,103.32) .. controls (126.1,104.76) and (124.94,105.94) .. (123.5,105.95) .. controls (122.06,105.97) and (120.88,104.81) .. (120.86,103.36) -- cycle ;
\draw [line width=0.75]    (141.35,85.5) -- (172.12,85.74)(141.32,88.5) -- (172.09,88.74) ;
\draw [shift={(180.11,87.31)}, rotate = 180.45] [color={rgb, 255:red, 0; green, 0; blue, 0 }  ][line width=0.75]    (10.93,-3.29) .. controls (6.95,-1.4) and (3.31,-0.3) .. (0,0) .. controls (3.31,0.3) and (6.95,1.4) .. (10.93,3.29)   ;
\draw [line width=0.75]    (226.08,48.32) -- (191.81,102.94) ;
\draw [color={rgb, 255:red, 74; green, 144; blue, 226 }  ,draw opacity=1 ][line width=0.75]    (226.08,48.32) -- (258.14,103.01) ;
\draw [line width=0.75]    (191.81,102.94) -- (258.14,103.01) ;
\draw  [fill={rgb, 255:red, 0; green, 0; blue, 0 }  ,fill opacity=1 ][line width=0.75]  (223.47,48.34) .. controls (223.46,46.9) and (224.62,45.72) .. (226.06,45.71) .. controls (227.5,45.69) and (228.69,46.85) .. (228.7,48.3) .. controls (228.71,49.74) and (227.55,50.92) .. (226.11,50.93) .. controls (224.66,50.95) and (223.48,49.79) .. (223.47,48.34) -- cycle ;
\draw  [fill={rgb, 255:red, 0; green, 0; blue, 0 }  ,fill opacity=1 ][line width=0.75]  (189.19,102.96) .. controls (189.18,101.52) and (190.34,100.34) .. (191.78,100.33) .. controls (193.23,100.31) and (194.41,101.47) .. (194.42,102.92) .. controls (194.43,104.36) and (193.27,105.54) .. (191.83,105.56) .. controls (190.39,105.57) and (189.21,104.41) .. (189.19,102.96) -- cycle ;
\draw  [fill={rgb, 255:red, 0; green, 0; blue, 0 }  ,fill opacity=1 ][line width=0.75]  (255.53,103.03) .. controls (255.52,101.59) and (256.68,100.4) .. (258.12,100.39) .. controls (259.56,100.38) and (260.75,101.54) .. (260.76,102.98) .. controls (260.77,104.43) and (259.61,105.61) .. (258.17,105.62) .. controls (256.72,105.63) and (255.54,104.47) .. (255.53,103.03) -- cycle ;
\draw [line width=0.75]    (226.08,48.32) -- (287.62,48.2) ;
\draw [line width=0.75]    (287.62,48.2) -- (258.14,103.01) ;
\draw  [fill={rgb, 255:red, 0; green, 0; blue, 0 }  ,fill opacity=1 ][line width=0.75]  (285.01,48.22) .. controls (285,46.78) and (286.16,45.6) .. (287.6,45.58) .. controls (289.04,45.57) and (290.23,46.73) .. (290.24,48.17) .. controls (290.25,49.62) and (289.09,50.8) .. (287.65,50.81) .. controls (286.2,50.82) and (285.02,49.66) .. (285.01,48.22) -- cycle ;
\draw [line width=0.75]    (296.35,86.17) -- (327.12,86.41)(296.32,89.17) -- (327.09,89.41) ;
\draw [shift={(335.11,87.97)}, rotate = 180.45] [color={rgb, 255:red, 0; green, 0; blue, 0 }  ][line width=0.75]    (10.93,-3.29) .. controls (6.95,-1.4) and (3.31,-0.3) .. (0,0) .. controls (3.31,0.3) and (6.95,1.4) .. (10.93,3.29)   ;
\draw [line width=0.75]    (383.08,48.59) -- (348.81,103.21) ;
\draw [color={rgb, 255:red, 74; green, 144; blue, 226 }  ,draw opacity=1 ][line width=0.75]    (383.08,48.59) -- (415.14,103.27) ;
\draw [color={rgb, 255:red, 74; green, 144; blue, 226 }  ,draw opacity=1 ][line width=0.75]    (348.81,103.21) -- (415.14,103.27) ;
\draw  [fill={rgb, 255:red, 0; green, 0; blue, 0 }  ,fill opacity=1 ][line width=0.75]  (380.47,48.61) .. controls (380.46,47.17) and (381.62,45.99) .. (383.06,45.97) .. controls (384.5,45.96) and (385.69,47.12) .. (385.7,48.56) .. controls (385.71,50.01) and (384.55,51.19) .. (383.11,51.2) .. controls (381.66,51.21) and (380.48,50.05) .. (380.47,48.61) -- cycle ;
\draw  [fill={rgb, 255:red, 0; green, 0; blue, 0 }  ,fill opacity=1 ][line width=0.75]  (346.19,103.23) .. controls (346.18,101.79) and (347.34,100.61) .. (348.78,100.59) .. controls (350.23,100.58) and (351.41,101.74) .. (351.42,103.18) .. controls (351.43,104.63) and (350.27,105.81) .. (348.83,105.82) .. controls (347.39,105.84) and (346.21,104.68) .. (346.19,103.23) -- cycle ;
\draw  [fill={rgb, 255:red, 0; green, 0; blue, 0 }  ,fill opacity=1 ][line width=0.75]  (412.53,103.3) .. controls (412.52,101.85) and (413.68,100.67) .. (415.12,100.66) .. controls (416.56,100.64) and (417.75,101.8) .. (417.76,103.25) .. controls (417.77,104.69) and (416.61,105.87) .. (415.17,105.89) .. controls (413.72,105.9) and (412.54,104.74) .. (412.53,103.3) -- cycle ;
\draw [line width=0.75]    (383.08,48.59) -- (444.62,48.46) ;
\draw [line width=0.75]    (444.62,48.46) -- (415.14,103.27) ;
\draw  [fill={rgb, 255:red, 0; green, 0; blue, 0 }  ,fill opacity=1 ][line width=0.75]  (442.01,48.49) .. controls (442,47.04) and (443.16,45.86) .. (444.6,45.85) .. controls (446.04,45.84) and (447.23,47) .. (447.24,48.44) .. controls (447.25,49.88) and (446.09,51.06) .. (444.65,51.08) .. controls (443.2,51.09) and (442.02,49.93) .. (442.01,48.49) -- cycle ;
\draw [line width=0.75]    (348.81,103.21) -- (383.13,155.5) ;
\draw [line width=0.75]    (415.14,103.27) -- (383.13,155.5) ;
\draw  [fill={rgb, 255:red, 0; green, 0; blue, 0 }  ,fill opacity=1 ] (380.52,155.52) .. controls (380.5,154.08) and (381.66,152.9) .. (383.11,152.89) .. controls (384.55,152.87) and (385.73,154.03) .. (385.74,155.48) .. controls (385.76,156.92) and (384.6,158.1) .. (383.15,158.11) .. controls (381.71,158.13) and (380.53,156.97) .. (380.52,155.52) -- cycle ;
\draw [line width=0.75]    (447.35,89.5) -- (478.12,89.74)(447.32,92.5) -- (478.09,92.74) ;
\draw [shift={(486.11,91.31)}, rotate = 180.45] [color={rgb, 255:red, 0; green, 0; blue, 0 }  ][line width=0.75]    (10.93,-3.29) .. controls (6.95,-1.4) and (3.31,-0.3) .. (0,0) .. controls (3.31,0.3) and (6.95,1.4) .. (10.93,3.29)   ;
\draw [color={rgb, 255:red, 74; green, 144; blue, 226 }  ,draw opacity=1 ][line width=0.75]    (564.95,48.99) -- (530.67,103.61) ;
\draw [color={rgb, 255:red, 74; green, 144; blue, 226 }  ,draw opacity=1 ][line width=0.75]    (564.95,48.99) -- (597.01,103.67) ;
\draw [color={rgb, 255:red, 74; green, 144; blue, 226 }  ,draw opacity=1 ][line width=0.75]    (530.67,103.61) -- (597.01,103.67) ;
\draw  [fill={rgb, 255:red, 0; green, 0; blue, 0 }  ,fill opacity=1 ][line width=0.75]  (562.34,49.01) .. controls (562.32,47.57) and (563.48,46.39) .. (564.93,46.37) .. controls (566.37,46.36) and (567.55,47.52) .. (567.56,48.96) .. controls (567.58,50.41) and (566.42,51.59) .. (564.97,51.6) .. controls (563.53,51.61) and (562.35,50.45) .. (562.34,49.01) -- cycle ;
\draw  [fill={rgb, 255:red, 0; green, 0; blue, 0 }  ,fill opacity=1 ][line width=0.75]  (528.06,103.63) .. controls (528.05,102.19) and (529.21,101.01) .. (530.65,100.99) .. controls (532.09,100.98) and (533.28,102.14) .. (533.29,103.58) .. controls (533.3,105.03) and (532.14,106.21) .. (530.7,106.22) .. controls (529.25,106.24) and (528.07,105.08) .. (528.06,103.63) -- cycle ;
\draw  [fill={rgb, 255:red, 0; green, 0; blue, 0 }  ,fill opacity=1 ][line width=0.75]  (594.4,103.7) .. controls (594.38,102.25) and (595.54,101.07) .. (596.99,101.06) .. controls (598.43,101.04) and (599.61,102.2) .. (599.63,103.65) .. controls (599.64,105.09) and (598.48,106.27) .. (597.03,106.29) .. controls (595.59,106.3) and (594.41,105.14) .. (594.4,103.7) -- cycle ;
\draw [line width=0.75]    (564.95,48.99) -- (626.49,48.86) ;
\draw [line width=0.75]    (626.49,48.86) -- (597.01,103.67) ;
\draw  [fill={rgb, 255:red, 0; green, 0; blue, 0 }  ,fill opacity=1 ] (623.88,48.89) .. controls (623.86,47.44) and (625.02,46.26) .. (626.47,46.25) .. controls (627.91,46.24) and (629.09,47.4) .. (629.11,48.84) .. controls (629.12,50.28) and (627.96,51.46) .. (626.51,51.48) .. controls (625.07,51.49) and (623.89,50.33) .. (623.88,48.89) -- cycle ;
\draw [line width=0.75]    (530.67,103.61) -- (565,155.9) ;
\draw [line width=0.75]    (597.01,103.67) -- (565,155.9) ;
\draw  [fill={rgb, 255:red, 0; green, 0; blue, 0 }  ,fill opacity=1 ] (562.38,155.92) .. controls (562.37,154.48) and (563.53,153.3) .. (564.97,153.29) .. controls (566.42,153.27) and (567.6,154.43) .. (567.61,155.88) .. controls (567.62,157.32) and (566.46,158.5) .. (565.02,158.51) .. controls (563.58,158.53) and (562.4,157.37) .. (562.38,155.92) -- cycle ;
\draw [line width=0.75]    (564.95,48.99) -- (501.97,49.9) ;
\draw [line width=0.75]    (501.97,49.9) -- (530.67,103.61) ;
\draw  [fill={rgb, 255:red, 0; green, 0; blue, 0 }  ,fill opacity=1 ][line width=0.75]  (499.36,49.92) .. controls (499.34,48.48) and (500.5,47.3) .. (501.95,47.29) .. controls (503.39,47.27) and (504.57,48.43) .. (504.59,49.88) .. controls (504.6,51.32) and (503.44,52.5) .. (502,52.51) .. controls (500.55,52.53) and (499.37,51.37) .. (499.36,49.92) -- cycle ;

\draw (80,168.5) node [anchor=north west][inner sep=0.75pt]   [align=left] {$T_3$};
\draw (222,168.5) node [anchor=north west][inner sep=0.75pt]   [align=left] {$T_4$};
\draw (375.67,168.5) node [anchor=north west][inner sep=0.75pt]   [align=left] {$T_5$};
\draw (554.2,168.5) node [anchor=north west][inner sep=0.75pt]   [align=left] {$T_6$};
\draw (215.55,35.33) node [anchor=north west][inner sep=0.75pt]  [font=\scriptsize] [align=left] {$a_3$};
\draw (262.76,105.98) node [anchor=north west][inner sep=0.75pt]  [font=\scriptsize] [align=left] {$b_3$};
\draw (290.81,36.47) node [anchor=north west][inner sep=0.75pt]  [font=\scriptsize] [align=left] {$x_3$};
\draw (334.55,105.93) node [anchor=north west][inner sep=0.75pt]  [font=\scriptsize] [align=left] {$a_4$};
\draw (419.88,102.93) node [anchor=north west][inner sep=0.75pt]  [font=\scriptsize] [align=left] {$b_4$};
\draw (390.21,153.27) node [anchor=north west][inner sep=0.75pt]  [font=\scriptsize] [align=left] {$x_4$};
\draw (562.41,34.33) node [anchor=north west][inner sep=0.75pt]  [font=\scriptsize] [align=left] {$a_5$};
\draw (518.08,104.33) node [anchor=north west][inner sep=0.75pt]  [font=\scriptsize] [align=left] {$b_5$};
\draw (491.08,35.33) node [anchor=north west][inner sep=0.75pt]  [font=\scriptsize] [align=left] {$x_5$};
\end{tikzpicture}
    \caption{Trigonal graphs}
    \label{fig:trigonal}
\end{figure}
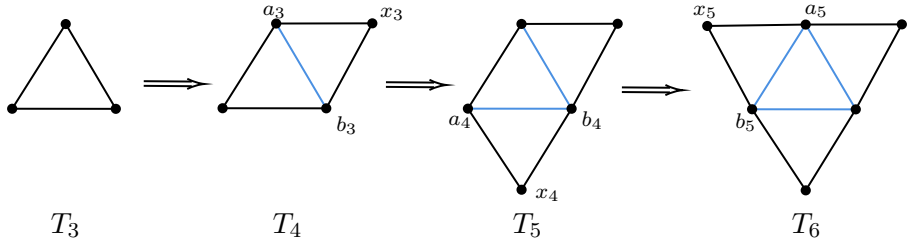

By definition, a trigonal graph is exactly an outer-planar graph in which every inner face is a triangle.
We refer to Figures~\ref{fig:trigonal} and \ref{fig:small T-graph} for examples of trigonal graphs. The black lines represent the boundary cycle. In each step, $T_{i+1}$ is obtained from $T_{i}$ by replacing the boundary edge $a_ib_i$ of $\partial T_{i}$ with the path $a_ix_ib_i$. 
The following proposition ensures the existence of consecutive path lengths between two vertices at a given distance on $\partial T$. 
\begin{prop}\label{prop: T-graph}
	Let $T$ be a trigonal graph with $|T|=t$, and let $u,v$ be two distinct vertices of $V(T)$ with $\mathrm{dist}_{\partial T}(u,v)=d$, then $[d,t-d]\subseteq \Pset_{u,v}^{T}$. In particular, if $uv\in E(\partial T)$, then $[1,t-1]\subseteq \Pset_{u,v}^{T}$.
\end{prop}

\begin{proof}
We proceed by induction on $t$. 
The base case $t=3$ is trivial.
Assume that the proposition holds for every trigonal graph of order at most $k$ for some $k \ge 3$. 
Suppose that $|T| = t = k + 1\geq 4$. 
It is clear that $\{d, k + 1 - d\} \subseteq \Pset^T_{u,v}$.

If $\deg_T(u) \ge 3$ or $\deg_T(v) \ge 3$, then there exists a vertex $w \in V(T) \setminus \{u, v\}$ such that $\deg_T(w) = 2$. Let $T' := T - w$. Then $T'$ is a trigonal graph of order $k$ and $\mathrm{dist}_{\partial T'}(u, v) \le d$. By the induction hypothesis, 
\[
\Pset^T_{u,v} \supseteq \{d, k + 1 - d\} \cup \Pset^{T'}_{u,v} \supseteq \{d, k + 1 - d\} \cup [d, k - d] = [d, k + 1 - d].
\]

Otherwise, we must have $\deg_T(u) = \deg_T(v) = 2$. 
By the recursive construction of trigonal graphs, every edge on $\partial T$ belongs to a triangle in $T$.
Hence, $u$ and $v$ cannot be adjacent on $\partial T$, as otherwise $T \simeq K_3$, contradicting that $t = k + 1 \ge 4$.
Thus, $d=\mathrm{dist}_{\partial T}(u, v) \ge 2$. 
Let $u' \in V(T)$ be a vertex satisfying $\mathrm{dist}_{\partial T}(u, u') = 1$ and $\mathrm{dist}_{\partial T}(v, u') = d - 1$. Note that $T' = T - \{u\}$ is a trigonal graph on $k$ vertices. The induction hypothesis yields $[d - 1, k - (d-1)] \subseteq \Pset^{T'}_{v, u'}$. Hence,
\[
\Pset^T_{u,v} \supseteq (\Pset^{T'}_{v, u'} + 1) \supseteq [d, k + 1 - d].
\]
This completes the induction proof.
\end{proof}
In particular, applying Proposition~\ref{prop: T-graph} to $uv\in E(\partial T)$ in a trigonal graph $T$ of order $t$, it follows that $T$ has cycles (containing the edge $uv$) of all lengths in $[3,t]$.

The following simple observation yields an improved bound for the cases $t \le 6$ in Proposition~\ref{prop: T-graph}. Its proof follows directly from Figure~\ref{fig:small T-graph}, which we omit.

\begin{obs}\label{obs: small T-graph}
    Let $T$ be a trigonal graph with $|T|=t\le 6$, and let $u,v$ be two distinct vertices of $V(T)$ with $\mathrm{dist}_{\partial T}(u,v)=d$. Then either $[d,t-d+1]\subseteq \Pset_{u,v}^{T}$ or $[d-1,t-d]\subseteq \Pset_{u,v}^{T}$ (i.e., $\mathcal{P}_{u,v}^{T}$ contains $t-2d+2$ consecutive paths), unless $T$ is of Case III with $\{u,v\}=\{w_{3},w_{5}\}$.
\end{obs}

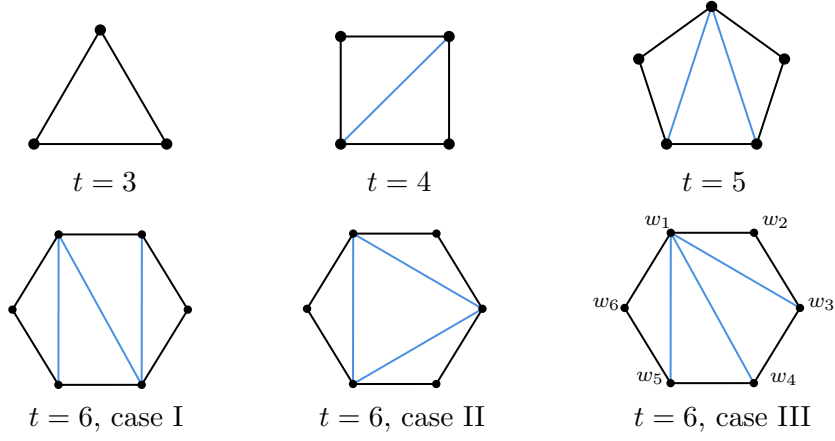
\begin{figure}
    \centering
\tikzset{every picture/.style={line width=0.75pt}}      
\begin{tikzpicture}[x=0.75pt,y=0.75pt,yscale=-1,xscale=1, scale=0.85]

\def\radius{45}     
\def\baseY{125}    

\def\cXA{126} 
\def\cXB{299} 
\def\cXC{485} 

\coordinate (c1) at (\cXA, {\baseY - \radius*0.5});

\coordinate (t3_v1) at ($(c1) + (270:\radius)$); 
\coordinate (t3_v2) at ($(c1) + (30:\radius)$);  
\coordinate (t3_v3) at ($(c1) + (150:\radius)$); 

\draw (t3_v1) -- (t3_v2) -- (t3_v3) -- cycle;
\foreach \p in {t3_v1, t3_v2, t3_v3} \fill [black] (\p) circle (2.5pt);

\coordinate (c2) at (\cXB, {\baseY - \radius*0.707});

\coordinate (t4_v1) at ($(c2) + (225:\radius)$);
\coordinate (t4_v2) at ($(c2) + (315:\radius)$);
\coordinate (t4_v3) at ($(c2) + (45:\radius)$);  
\coordinate (t4_v4) at ($(c2) + (135:\radius)$); 

\draw (t4_v1) -- (t4_v2) -- (t4_v3) -- (t4_v4) -- cycle;

\draw [color={rgb, 255:red, 74; green, 144; blue, 226 }, draw opacity=1] (t4_v2) -- (t4_v4);
\foreach \p in {t4_v1, t4_v2, t4_v3, t4_v4} \fill [black] (\p) circle (2.5pt);

\coordinate (c3) at (\cXC, {\baseY - \radius*0.809});

\coordinate (t5_v1) at ($(c3) + (270:\radius)$); 
\coordinate (t5_v2) at ($(c3) + (342:\radius)$); 
\coordinate (t5_v3) at ($(c3) + (54:\radius)$);  
\coordinate (t5_v4) at ($(c3) + (126:\radius)$); 
\coordinate (t5_v5) at ($(c3) + (198:\radius)$); 

\draw (t5_v1) -- (t5_v2) -- (t5_v3) -- (t5_v4) -- (t5_v5) -- cycle;

\draw [color={rgb, 255:red, 74; green, 144; blue, 226 }, draw opacity=1] (t5_v1) -- (t5_v3);
\draw [color={rgb, 255:red, 74; green, 144; blue, 226 }, draw opacity=1] (t5_v1) -- (t5_v4);
\foreach \p in {t5_v1, t5_v2, t5_v3, t5_v4, t5_v5} \fill [black] (\p) circle (2.5pt);

\draw    (101.5,178.5) -- (150.39,178.5) ;
\draw    (74.39,223) -- (101.5,178.5) ;
\draw    (150.39,178.5) -- (177.39,223) ;
\draw    (150.29,267.49) -- (177.39,223) ;
\draw    (74.39,223) -- (101.39,267.5) ;
\draw    (101.39,267.5) -- (150.29,267.49) ;
\draw    (274.5,178) -- (323.39,178) ;
\draw    (247.39,222.5) -- (274.5,178) ;
\draw    (323.39,178) -- (350.39,222.5) ;
\draw    (323.29,266.99) -- (350.39,222.5) ;
\draw    (247.39,222.5) -- (274.39,267) ;
\draw    (274.39,267) -- (323.29,266.99) ;
\draw    (461,177.5) -- (509.89,177.5) ;
\draw    (433.89,222) -- (461,177.5) ;
\draw    (509.89,177.5) -- (536.89,222) ;
\draw    (509.79,266.49) -- (536.89,222) ;
\draw    (433.89,222) -- (460.89,266.5) ;
\draw    (460.89,266.5) -- (509.79,266.49) ;
\draw [color={rgb, 255:red, 74; green, 144; blue, 226 }  ,draw opacity=1 ]   (101.5,178.5) -- (101.39,267.5) ;
\draw [color={rgb, 255:red, 74; green, 144; blue, 226 }  ,draw opacity=1 ]   (101.5,178.5) -- (150.29,267.49) ;
\draw [color={rgb, 255:red, 74; green, 144; blue, 226 }  ,draw opacity=1 ]   (150.39,178.5) -- (150.29,267.49) ;
\draw [color={rgb, 255:red, 74; green, 144; blue, 226 }  ,draw opacity=1 ]   (274.5,178) -- (274.39,267) ;
\draw [color={rgb, 255:red, 74; green, 144; blue, 226 }  ,draw opacity=1 ]   (274.5,178) -- (350.39,222.5) ;
\draw [color={rgb, 255:red, 74; green, 144; blue, 226 }  ,draw opacity=1 ]   (350.39,222.5) -- (274.39,267) ;
\draw [color={rgb, 255:red, 74; green, 144; blue, 226 }  ,draw opacity=1 ]   (461,177.5) -- (460.89,266.5) ;
\draw [color={rgb, 255:red, 74; green, 144; blue, 226 }  ,draw opacity=1 ]   (461,177.5) -- (509.79,266.49) ;
\draw [color={rgb, 255:red, 74; green, 144; blue, 226 }  ,draw opacity=1 ]   (461,177.5) -- (536.89,222) ;

\draw  [fill={rgb, 255:red, 0; green, 0; blue, 0 }  ,fill opacity=1 ] (99.64,178.5) .. controls (99.64,177.47) and (100.47,176.64) .. (101.5,176.64) .. controls (102.53,176.64) and (103.36,177.47) .. (103.36,178.5) .. controls (103.36,179.53) and (102.53,180.36) .. (101.5,180.36) .. controls (100.47,180.36) and (99.64,179.53) .. (99.64,178.5) -- cycle ;
\draw  [fill={rgb, 255:red, 0; green, 0; blue, 0 }  ,fill opacity=1 ] (148.53,178.5) .. controls (148.53,177.47) and (149.36,176.63) .. (150.39,176.63) .. controls (151.42,176.63) and (152.26,177.47) .. (152.26,178.5) .. controls (152.26,179.52) and (151.42,180.36) .. (150.39,180.36) .. controls (149.36,180.36) and (148.53,179.52) .. (148.53,178.5) -- cycle ;
\draw  [fill={rgb, 255:red, 0; green, 0; blue, 0 }  ,fill opacity=1 ] (72.53,223) .. controls (72.53,221.97) and (73.36,221.13) .. (74.39,221.13) .. controls (75.42,221.13) and (76.26,221.97) .. (76.26,223) .. controls (76.26,224.02) and (75.42,224.86) .. (74.39,224.86) .. controls (73.36,224.86) and (72.53,224.02) .. (72.53,223) -- cycle ;
\draw  [fill={rgb, 255:red, 0; green, 0; blue, 0 }  ,fill opacity=1 ] (99.53,267.5) .. controls (99.53,266.47) and (100.36,265.63) .. (101.39,265.63) .. controls (102.42,265.63) and (103.26,266.47) .. (103.26,267.5) .. controls (103.26,268.52) and (102.42,269.36) .. (101.39,269.36) .. controls (100.36,269.36) and (99.53,268.52) .. (99.53,267.5) -- cycle ;
\draw  [fill={rgb, 255:red, 0; green, 0; blue, 0 }  ,fill opacity=1 ] (148.42,267.49) .. controls (148.42,266.46) and (149.26,265.63) .. (150.29,265.63) .. controls (151.32,265.63) and (152.15,266.46) .. (152.15,267.49) .. controls (152.15,268.52) and (151.32,269.35) .. (150.29,269.35) .. controls (149.26,269.35) and (148.42,268.52) .. (148.42,267.49) -- cycle ;
\draw  [fill={rgb, 255:red, 0; green, 0; blue, 0 }  ,fill opacity=1 ] (175.53,223) .. controls (175.53,221.97) and (176.36,221.13) .. (177.39,221.13) .. controls (178.42,221.13) and (179.26,221.97) .. (179.26,223) .. controls (179.26,224.02) and (178.42,224.86) .. (177.39,224.86) .. controls (176.36,224.86) and (175.53,224.02) .. (175.53,223) -- cycle ;
\draw  [fill={rgb, 255:red, 0; green, 0; blue, 0 }  ,fill opacity=1 ] (272.64,178) .. controls (272.64,176.97) and (273.47,176.14) .. (274.5,176.14) .. controls (275.53,176.14) and (276.36,176.97) .. (276.36,178) .. controls (276.36,179.03) and (275.53,179.86) .. (274.5,179.86) .. controls (273.47,179.86) and (272.64,179.03) .. (272.64,178) -- cycle ;
\draw  [fill={rgb, 255:red, 0; green, 0; blue, 0 }  ,fill opacity=1 ] (245.53,222.5) .. controls (245.53,221.47) and (246.36,220.63) .. (247.39,220.63) .. controls (248.42,220.63) and (249.26,221.47) .. (249.26,222.5) .. controls (249.26,223.52) and (248.42,224.36) .. (247.39,224.36) .. controls (246.36,224.36) and (245.53,223.52) .. (245.53,222.5) -- cycle ;
\draw  [fill={rgb, 255:red, 0; green, 0; blue, 0 }  ,fill opacity=1 ] (272.53,267) .. controls (272.53,265.97) and (273.36,265.13) .. (274.39,265.13) .. controls (275.42,265.13) and (276.26,265.97) .. (276.26,267) .. controls (276.26,268.02) and (275.42,268.86) .. (274.39,268.86) .. controls (273.36,268.86) and (272.53,268.02) .. (272.53,267) -- cycle ;
\draw  [fill={rgb, 255:red, 0; green, 0; blue, 0 }  ,fill opacity=1 ] (348.53,222.5) .. controls (348.53,221.47) and (349.36,220.63) .. (350.39,220.63) .. controls (351.42,220.63) and (352.26,221.47) .. (352.26,222.5) .. controls (352.26,223.52) and (351.42,224.36) .. (350.39,224.36) .. controls (349.36,224.36) and (348.53,223.52) .. (348.53,222.5) -- cycle ;
\draw  [fill={rgb, 255:red, 0; green, 0; blue, 0 }  ,fill opacity=1 ] (321.53,178) .. controls (321.53,176.97) and (322.36,176.13) .. (323.39,176.13) .. controls (324.42,176.13) and (325.26,176.97) .. (325.26,178) .. controls (325.26,179.02) and (324.42,179.86) .. (323.39,179.86) .. controls (322.36,179.86) and (321.53,179.02) .. (321.53,178) -- cycle ;
\draw  [fill={rgb, 255:red, 0; green, 0; blue, 0 }  ,fill opacity=1 ] (321.42,266.99) .. controls (321.42,265.96) and (322.26,265.13) .. (323.29,265.13) .. controls (324.32,265.13) and (325.15,265.96) .. (325.15,266.99) .. controls (325.15,268.02) and (324.32,268.85) .. (323.29,268.85) .. controls (322.26,268.85) and (321.42,268.02) .. (321.42,266.99) -- cycle ;
\draw  [fill={rgb, 255:red, 0; green, 0; blue, 0 }  ,fill opacity=1 ] (459.14,177.5) .. controls (459.14,176.47) and (459.97,175.64) .. (461,175.64) .. controls (462.03,175.64) and (462.86,176.47) .. (462.86,177.5) .. controls (462.86,178.53) and (462.03,179.36) .. (461,179.36) .. controls (459.97,179.36) and (459.14,178.53) .. (459.14,177.5) -- cycle ;
\draw  [fill={rgb, 255:red, 0; green, 0; blue, 0 }  ,fill opacity=1 ] (432.03,222) .. controls (432.03,220.97) and (432.86,220.13) .. (433.89,220.13) .. controls (434.92,220.13) and (435.76,220.97) .. (435.76,222) .. controls (435.76,223.02) and (434.92,223.86) .. (433.89,223.86) .. controls (432.86,223.86) and (432.03,223.02) .. (432.03,222) -- cycle ;
\draw  [fill={rgb, 255:red, 0; green, 0; blue, 0 }  ,fill opacity=1 ] (459.03,266.5) .. controls (459.03,265.47) and (459.86,264.63) .. (460.89,264.63) .. controls (461.92,264.63) and (462.76,265.47) .. (462.76,266.5) .. controls (462.76,267.52) and (461.92,268.36) .. (460.89,268.36) .. controls (459.86,268.36) and (459.03,267.52) .. (459.03,266.5) -- cycle ;
\draw  [fill={rgb, 255:red, 0; green, 0; blue, 0 }  ,fill opacity=1 ] (508.03,177.5) .. controls (508.03,176.47) and (508.86,175.63) .. (509.89,175.63) .. controls (510.92,175.63) and (511.76,176.47) .. (511.76,177.5) .. controls (511.76,178.52) and (510.92,179.36) .. (509.89,179.36) .. controls (508.86,179.36) and (508.03,178.52) .. (508.03,177.5) -- cycle ;
\draw  [fill={rgb, 255:red, 0; green, 0; blue, 0 }  ,fill opacity=1 ] (535.03,222) .. controls (535.03,220.97) and (535.86,220.13) .. (536.89,220.13) .. controls (537.92,220.13) and (538.76,220.97) .. (538.76,222) .. controls (538.76,223.02) and (537.92,223.86) .. (536.89,223.86) .. controls (535.86,223.86) and (535.03,223.02) .. (535.03,222) -- cycle ;
\draw  [fill={rgb, 255:red, 0; green, 0; blue, 0 }  ,fill opacity=1 ] (507.92,266.49) .. controls (507.92,265.46) and (508.76,264.63) .. (509.79,264.63) .. controls (510.82,264.63) and (511.65,265.46) .. (511.65,266.49) .. controls (511.65,267.52) and (510.82,268.35) .. (509.79,268.35) .. controls (508.76,268.35) and (507.92,267.52) .. (507.92,266.49) -- cycle ;

\draw (108,140) node [anchor=north west][inner sep=0.75pt]   [align=left] {$t=3$};
\draw (281,140) node [anchor=north west][inner sep=0.75pt]   [align=left] {$t=4$};
\draw (466,140) node [anchor=north west][inner sep=0.75pt]   [align=left] {$t=5$};
\draw (82.33,278.5) node [anchor=north west][inner sep=0.75pt]   [align=left] {$t=6$, case I};
\draw (252.83,278.5) node [anchor=north west][inner sep=0.75pt]   [align=left] {$t=6$, case II};
\draw (437.5,278.5) node [anchor=north west][inner sep=0.75pt]   [align=left] {$t=6$, case III};
\draw (443.83,165.59) node [anchor=north west][inner sep=0.75pt]   [align=left] {{\scriptsize $w_1$}};
\draw (512.83,165.76) node [anchor=north west][inner sep=0.75pt]   [align=left] {{\scriptsize $w_2$}};
\draw (540,214.09) node [anchor=north west][inner sep=0.75pt]   [align=left] {{\scriptsize $w_3$}};
\draw (517.67,258.59) node [anchor=north west][inner sep=0.75pt]   [align=left] {{\scriptsize $w_4$}};
\draw (439.17,258.26) node [anchor=north west][inner sep=0.75pt]   [align=left] {{\scriptsize $w_5$}};
\draw (413.5,214.09) node [anchor=north west][inner sep=0.75pt]   [align=left] {{\scriptsize $w_6$}};
\end{tikzpicture}
    \caption{All trigonal graphs on at most six vertices}
    \label{fig:small T-graph}
\end{figure}

\subsection{Tetragonal graphs}
\begin{dfn}[Tetragonal graph]\label{dfn:Q graph}
	A tetragonal graph $T$ is a bipartite outer-planar graph equipped with a Hamiltonian cycle $\partial T$, defined as the final graph $T_{n}$ (with $\partial T = \partial T_{n}$) of a finite sequence of tetragonal graphs $T_{2},T_3,\cdots,T_{n}$ satisfying the following properties:
	\begin{itemize}
		\item $T_{2}\simeq C_{4}$, $\partial T_{2}=T_{2}$.
		
		\item For every $2\leq i \leq n-1$, $T_{i+1}$ is obtained from $T_{i}$ by adding two new vertices $x_i,y_i$ and
        a path $P_{i}:=a_{i}x_{i}y_{i}b_{i}$ with $a_{i}b_{i}\in E(\partial T_{i})$; $\partial T_{i+1}$ is obtained from $\partial T_{i}$ by adding the path $P_{i}$ and deleting the edge $a_{i}b_{i}$.
	\end{itemize}
\end{dfn}

By definition, a tetragonal graph is exactly an outer-planar graph in which every inner face is a 4-cycle.
See Figure~\ref{fig:tetragonal} for a sequence of tetragonal graphs on $4,6,8,10$, and $12$ vertices.
The black lines represent their boundary cycles, and the vertices $a_i,b_i,x_i,y_i$ are marked to indicate that $T_{i+1}$ is obtained from $T_{i}$ by replacing the edge $a_ib_i$ on $\partial T_{i}$ with the path $a_ix_iy_ib_i$. 

The following proposition ensures the existence of admissible paths between two vertices at a given distance on the boundary $\partial T$.
\begin{prop}\label{prop:Q-graph}
    Let $T$ be a tetragonal graph with $|T|=2m$, and let $u, v$ be two distinct vertices of $T$ with $\mathrm{dist}_{\partial T}(u,v) = d$. Then $\{d,d+2,\dots,2m-d\}\subseteq \Pset_{u,v}^{T}$.
\end{prop}

\begin{proof}
    We proceed by induction on $m$. The base case $m=2$ (where $T\simeq C_4$) is trivial.
	Assume that the proposition holds for every tetragonal graph with fewer than $2m$ vertices.
	Let $a,x,y,b$ be four consecutive vertices on $\partial T$ such that $T$ is obtained from a smaller tetragonal graph $T'$ by adding the new path $axyb$.
    Then $V(T')=V(T)\setminus \{x,y\}$ and $|T'|=2m-2$.

    If $\{x,y\}=\{u,v\}$, then $1\in \Pset_{u,v}^{T}$.
	Applying the induction hypothesis to $T'$ yields $\Pset^{T'}_{a,b}\supseteq\{1,3,\dots,2m-3\}$.
	Consequently, $\Pset_{u,v}^{T}\supseteq\{1\}\cup (2+\Pset^{T'}_{a,b})\supseteq\{1,3,\dots,2m-1\}$.

    If $|\{x,y\}\cap \{u,v\}|=1$, we may assume without loss of generality that $u=x$ and $v\notin \{x,y\}$.
	Then $\mathrm{dist}_{\partial T'}(a,v)=d-1$ and $\mathrm{dist}_{\partial T'}(b,v)\in \{d-2,d\}$.
	Since $v \notin \{x,y\}$, we must have $v \neq a$ or $v \neq b$.
	If $v\neq a$, the induction hypothesis on $T'$ implies $\Pset^{T'}_{a,v}\supseteq\{d-1,d+1,\dots,2m-d-1\}$.
	Thus, $\Pset^{T}_{u,v}\supseteq 1+\Pset^{T'}_{a,v}\supseteq\{d,d+2,\dots,2m-d\}$.
	If $v\neq b$, the induction hypothesis implies $\Pset^{T'}_{b,v}\supseteq\{d,d+2,\dots,2m-d-2\}$.
	\mbox{Thus, $\Pset^{T}_{u,v}\supseteq \{d\}\cup (2+\Pset^{T'}_{b,v})\supseteq\{d,d+2,\dots,2m-d\}$, as desired.}

    Finally, if $\{x,y\}\cap \{u,v\}=\emptyset$, then $\mathrm{dist}_{\partial T'}(u,v)\in \{d-2,d\}$.
	The induction hypothesis implies $\Pset^{T'}_{u,v}\supseteq\{d,d+2,\dots,2m-d-2\}$.
	Hence, $\Pset^{T}_{u,v}\supseteq \{2m-d\}\cup \Pset^{T'}_{u,v}\supseteq\{d,d+2,\dots,2m-d\}$.
	In all cases, we have $\Pset^{T}_{u,v}\supseteq\{d,d+2,\dots,2m-d\}$, and the result follows by induction.
\end{proof}

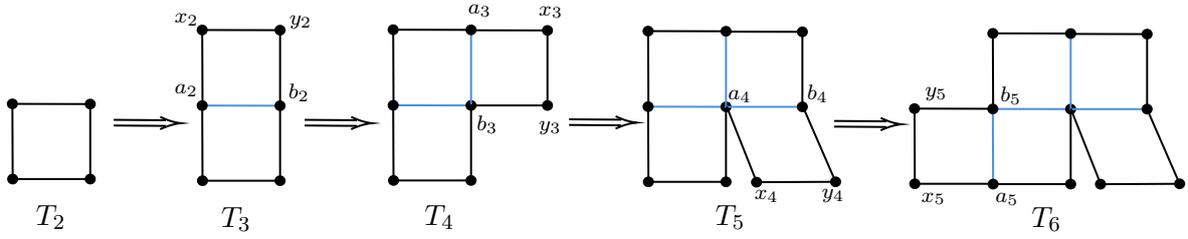
\begin{figure}
    \centering
\tikzset{every picture/.style={line width=0.75pt}} 

\tikzset{every picture/.style={line width=0.75pt}} 

\begin{tikzpicture}[x=0.75pt,y=0.75pt,yscale=-1,xscale=1, scale=0.95]

\draw [color={rgb, 255:red, 0; green, 0; blue, 0 }  ,draw opacity=1 ][line width=0.75]    (13.1,58.04) -- (13.27,97.85) ;
\draw [color={rgb, 255:red, 0; green, 0; blue, 0 }  ,draw opacity=1 ][line width=0.75]    (13.1,58.04) -- (54.02,58.14) ;
\draw [color={rgb, 255:red, 0; green, 0; blue, 0 }  ,draw opacity=1 ][line width=0.75]    (54.02,58.14) -- (54.02,97.85) ;
\draw [color={rgb, 255:red, 0; green, 0; blue, 0 }  ,draw opacity=1 ][line width=0.75]    (13.27,97.85) -- (54.02,97.85) ;
\draw  [color={rgb, 255:red, 0; green, 0; blue, 0 }  ,draw opacity=1 ][fill={rgb, 255:red, 0; green, 0; blue, 0 }  ,fill opacity=1 ][line width=0.75]  (10.72,58.06) .. controls (10.71,56.79) and (11.77,55.75) .. (13.08,55.73) .. controls (14.4,55.72) and (15.48,56.75) .. (15.49,58.02) .. controls (15.5,59.3) and (14.44,60.34) .. (13.13,60.35) .. controls (11.81,60.36) and (10.73,59.34) .. (10.72,58.06) -- cycle ;
\draw  [color={rgb, 255:red, 0; green, 0; blue, 0 }  ,draw opacity=1 ][fill={rgb, 255:red, 0; green, 0; blue, 0 }  ,fill opacity=1 ][line width=0.75]  (51.63,58.16) .. controls (51.62,56.88) and (52.68,55.84) .. (53.99,55.83) .. controls (55.31,55.82) and (56.39,56.84) .. (56.4,58.12) .. controls (56.41,59.39) and (55.36,60.43) .. (54.04,60.44) .. controls (52.72,60.46) and (51.64,59.43) .. (51.63,58.16) -- cycle ;
\draw  [color={rgb, 255:red, 0; green, 0; blue, 0 }  ,draw opacity=1 ][fill={rgb, 255:red, 0; green, 0; blue, 0 }  ,fill opacity=1 ][line width=0.75]  (10.88,97.87) .. controls (10.87,96.6) and (11.93,95.56) .. (13.25,95.54) .. controls (14.56,95.53) and (15.64,96.56) .. (15.65,97.83) .. controls (15.67,99.1) and (14.61,100.15) .. (13.29,100.16) .. controls (11.97,100.17) and (10.89,99.15) .. (10.88,97.87) -- cycle ;
\draw  [color={rgb, 255:red, 0; green, 0; blue, 0 }  ,draw opacity=1 ][fill={rgb, 255:red, 0; green, 0; blue, 0 }  ,fill opacity=1 ][line width=0.75]  (51.63,97.87) .. controls (51.62,96.6) and (52.68,95.56) .. (53.99,95.54) .. controls (55.31,95.53) and (56.39,96.56) .. (56.4,97.83) .. controls (56.41,99.1) and (55.36,100.15) .. (54.04,100.16) .. controls (52.72,100.17) and (51.64,99.15) .. (51.63,97.87) -- cycle ;
\draw [color={rgb, 255:red, 0; green, 0; blue, 0 }  ,draw opacity=1 ][line width=0.75]    (66.33,66.49) -- (93.7,66.7)(66.31,69.49) -- (93.68,69.7) ;
\draw [shift={(101.69,68.27)}, rotate = 180.44] [color={rgb, 255:red, 0; green, 0; blue, 0 }  ,draw opacity=1 ][line width=0.75]    (10.93,-3.29) .. controls (6.95,-1.4) and (3.31,-0.3) .. (0,0) .. controls (3.31,0.3) and (6.95,1.4) .. (10.93,3.29)   ;
\draw [color={rgb, 255:red, 0; green, 0; blue, 0 }  ,draw opacity=1 ][line width=0.75]    (112.85,58.92) -- (113.01,98.73) ;
\draw [color={rgb, 255:red, 74; green, 144; blue, 226 }  ,draw opacity=1 ][line width=0.75]    (112.85,58.92) -- (153.76,59.02) ;
\draw [color={rgb, 255:red, 0; green, 0; blue, 0 }  ,draw opacity=1 ][line width=0.75]    (153.76,59.02) -- (153.76,98.73) ;
\draw [color={rgb, 255:red, 0; green, 0; blue, 0 }  ,draw opacity=1 ][line width=0.75]    (113.01,98.73) -- (153.76,98.73) ;
\draw  [color={rgb, 255:red, 0; green, 0; blue, 0 }  ,draw opacity=1 ][fill={rgb, 255:red, 0; green, 0; blue, 0 }  ,fill opacity=1 ][line width=0.75]  (110.46,58.95) .. controls (110.45,57.67) and (111.51,56.63) .. (112.82,56.62) .. controls (114.14,56.61) and (115.22,57.63) .. (115.23,58.9) .. controls (115.24,60.18) and (114.19,61.22) .. (112.87,61.23) .. controls (111.55,61.24) and (110.47,60.22) .. (110.46,58.95) -- cycle ;
\draw  [color={rgb, 255:red, 0; green, 0; blue, 0 }  ,draw opacity=1 ][fill={rgb, 255:red, 0; green, 0; blue, 0 }  ,fill opacity=1 ][line width=0.75]  (151.37,59.04) .. controls (151.36,57.77) and (152.42,56.72) .. (153.74,56.71) .. controls (155.05,56.7) and (156.13,57.72) .. (156.14,59) .. controls (156.16,60.27) and (155.1,61.31) .. (153.78,61.33) .. controls (152.46,61.34) and (151.39,60.31) .. (151.37,59.04) -- cycle ;
\draw  [color={rgb, 255:red, 0; green, 0; blue, 0 }  ,draw opacity=1 ][fill={rgb, 255:red, 0; green, 0; blue, 0 }  ,fill opacity=1 ][line width=0.75]  (110.63,98.75) .. controls (110.61,97.48) and (111.67,96.44) .. (112.99,96.43) .. controls (114.31,96.41) and (115.38,97.44) .. (115.4,98.71) .. controls (115.41,99.99) and (114.35,101.03) .. (113.03,101.04) .. controls (111.71,101.05) and (110.64,100.03) .. (110.63,98.75) -- cycle ;
\draw  [color={rgb, 255:red, 0; green, 0; blue, 0 }  ,draw opacity=1 ][fill={rgb, 255:red, 0; green, 0; blue, 0 }  ,fill opacity=1 ][line width=0.75]  (151.37,98.75) .. controls (151.36,97.48) and (152.42,96.44) .. (153.74,96.43) .. controls (155.05,96.41) and (156.13,97.44) .. (156.14,98.71) .. controls (156.16,99.99) and (155.1,101.03) .. (153.78,101.04) .. controls (152.46,101.05) and (151.39,100.03) .. (151.37,98.75) -- cycle ;
\draw [color={rgb, 255:red, 0; green, 0; blue, 0 }  ,draw opacity=1 ][line width=0.75]    (112.85,18.92) -- (113.01,58.72) ;
\draw [color={rgb, 255:red, 0; green, 0; blue, 0 }  ,draw opacity=1 ][line width=0.75]    (112.85,18.92) -- (153.76,19.01) ;
\draw [color={rgb, 255:red, 0; green, 0; blue, 0 }  ,draw opacity=1 ][line width=0.75]    (153.76,19.01) -- (153.76,58.72) ;
\draw  [color={rgb, 255:red, 0; green, 0; blue, 0 }  ,draw opacity=1 ][fill={rgb, 255:red, 0; green, 0; blue, 0 }  ,fill opacity=1 ][line width=0.75]  (110.46,18.94) .. controls (110.45,17.66) and (111.51,16.62) .. (112.82,16.61) .. controls (114.14,16.6) and (115.22,17.62) .. (115.23,18.9) .. controls (115.24,20.17) and (114.19,21.21) .. (112.87,21.22) .. controls (111.55,21.23) and (110.47,20.21) .. (110.46,18.94) -- cycle ;
\draw  [color={rgb, 255:red, 0; green, 0; blue, 0 }  ,draw opacity=1 ][fill={rgb, 255:red, 0; green, 0; blue, 0 }  ,fill opacity=1 ][line width=0.75]  (151.37,19.03) .. controls (151.36,17.76) and (152.42,16.71) .. (153.74,16.7) .. controls (155.05,16.69) and (156.13,17.72) .. (156.14,18.99) .. controls (156.16,20.26) and (155.1,21.31) .. (153.78,21.32) .. controls (152.46,21.33) and (151.39,20.31) .. (151.37,19.03) -- cycle ;
\draw [color={rgb, 255:red, 0; green, 0; blue, 0 }  ,draw opacity=1 ][line width=0.75]    (166.68,66.49) -- (194.05,66.7)(166.66,69.49) -- (194.03,69.7) ;
\draw [shift={(202.04,68.27)}, rotate = 180.44] [color={rgb, 255:red, 0; green, 0; blue, 0 }  ,draw opacity=1 ][line width=0.75]    (10.93,-3.29) .. controls (6.95,-1.4) and (3.31,-0.3) .. (0,0) .. controls (3.31,0.3) and (6.95,1.4) .. (10.93,3.29)   ;
\draw [color={rgb, 255:red, 0; green, 0; blue, 0 }  ,draw opacity=1 ][line width=0.75]    (213.2,58.78) -- (213.36,98.59) ;
\draw [color={rgb, 255:red, 74; green, 144; blue, 226 }  ,draw opacity=1 ][line width=0.75]    (213.2,58.78) -- (254.11,58.87) ;
\draw [color={rgb, 255:red, 0; green, 0; blue, 0 }  ,draw opacity=1 ][line width=0.75]    (254.11,58.87) -- (254.11,98.59) ;
\draw [color={rgb, 255:red, 0; green, 0; blue, 0 }  ,draw opacity=1 ][line width=0.75]    (213.36,98.59) -- (254.11,98.59) ;
\draw  [color={rgb, 255:red, 0; green, 0; blue, 0 }  ,draw opacity=1 ][fill={rgb, 255:red, 0; green, 0; blue, 0 }  ,fill opacity=1 ][line width=0.75]  (210.81,58.8) .. controls (210.8,57.52) and (211.86,56.48) .. (213.18,56.47) .. controls (214.49,56.46) and (215.57,57.48) .. (215.58,58.76) .. controls (215.59,60.03) and (214.54,61.07) .. (213.22,61.08) .. controls (211.9,61.1) and (210.82,60.07) .. (210.81,58.8) -- cycle ;
\draw  [color={rgb, 255:red, 0; green, 0; blue, 0 }  ,draw opacity=1 ][fill={rgb, 255:red, 0; green, 0; blue, 0 }  ,fill opacity=1 ][line width=0.75]  (251.72,58.89) .. controls (251.71,57.62) and (252.77,56.58) .. (254.09,56.56) .. controls (255.4,56.55) and (256.48,57.58) .. (256.49,58.85) .. controls (256.51,60.13) and (255.45,61.17) .. (254.13,61.18) .. controls (252.81,61.19) and (251.74,60.17) .. (251.72,58.89) -- cycle ;
\draw  [color={rgb, 255:red, 0; green, 0; blue, 0 }  ,draw opacity=1 ][fill={rgb, 255:red, 0; green, 0; blue, 0 }  ,fill opacity=1 ][line width=0.75]  (210.98,98.61) .. controls (210.96,97.33) and (212.02,96.29) .. (213.34,96.28) .. controls (214.66,96.27) and (215.73,97.29) .. (215.75,98.57) .. controls (215.76,99.84) and (214.7,100.88) .. (213.38,100.89) .. controls (212.07,100.91) and (210.99,99.88) .. (210.98,98.61) -- cycle ;
\draw  [color={rgb, 255:red, 0; green, 0; blue, 0 }  ,draw opacity=1 ][fill={rgb, 255:red, 0; green, 0; blue, 0 }  ,fill opacity=1 ][line width=0.75]  (251.72,98.61) .. controls (251.71,97.33) and (252.77,96.29) .. (254.09,96.28) .. controls (255.4,96.27) and (256.48,97.29) .. (256.49,98.57) .. controls (256.51,99.84) and (255.45,100.88) .. (254.13,100.89) .. controls (252.81,100.91) and (251.74,99.88) .. (251.72,98.61) -- cycle ;
\draw [color={rgb, 255:red, 0; green, 0; blue, 0 }  ,draw opacity=1 ][line width=0.75]    (213.2,18.77) -- (213.36,58.58) ;
\draw [color={rgb, 255:red, 0; green, 0; blue, 0 }  ,draw opacity=1 ][line width=0.75]    (213.2,18.77) -- (254.11,18.86) ;
\draw [color={rgb, 255:red, 74; green, 144; blue, 226 }  ,draw opacity=1 ][line width=0.75]    (254.11,18.86) -- (254.11,58.58) ;
\draw  [color={rgb, 255:red, 0; green, 0; blue, 0 }  ,draw opacity=1 ][fill={rgb, 255:red, 0; green, 0; blue, 0 }  ,fill opacity=1 ][line width=0.75]  (210.81,18.79) .. controls (210.8,17.52) and (211.86,16.47) .. (213.18,16.46) .. controls (214.49,16.45) and (215.57,17.47) .. (215.58,18.75) .. controls (215.59,20.02) and (214.54,21.06) .. (213.22,21.08) .. controls (211.9,21.09) and (210.82,20.06) .. (210.81,18.79) -- cycle ;
\draw  [color={rgb, 255:red, 0; green, 0; blue, 0 }  ,draw opacity=1 ][fill={rgb, 255:red, 0; green, 0; blue, 0 }  ,fill opacity=1 ][line width=0.75]  (251.72,18.88) .. controls (251.71,17.61) and (252.77,16.57) .. (254.09,16.56) .. controls (255.4,16.54) and (256.48,17.57) .. (256.49,18.84) .. controls (256.51,20.12) and (255.45,21.16) .. (254.13,21.17) .. controls (252.81,21.18) and (251.74,20.16) .. (251.72,18.88) -- cycle ;
\draw [color={rgb, 255:red, 0; green, 0; blue, 0 }  ,draw opacity=1 ][line width=0.75]    (253.34,19.06) -- (294.25,19.16) ;
\draw [color={rgb, 255:red, 0; green, 0; blue, 0 }  ,draw opacity=1 ][line width=0.75]    (294.25,19.16) -- (294.25,58.87) ;
\draw [color={rgb, 255:red, 0; green, 0; blue, 0 }  ,draw opacity=1 ][line width=0.75]    (253.5,58.87) -- (294.25,58.87) ;
\draw  [color={rgb, 255:red, 0; green, 0; blue, 0 }  ,draw opacity=1 ][fill={rgb, 255:red, 0; green, 0; blue, 0 }  ,fill opacity=1 ][line width=0.75]  (291.86,19.18) .. controls (291.85,17.9) and (292.91,16.86) .. (294.23,16.85) .. controls (295.55,16.84) and (296.62,17.86) .. (296.63,19.14) .. controls (296.65,20.41) and (295.59,21.45) .. (294.27,21.46) .. controls (292.95,21.48) and (291.88,20.45) .. (291.86,19.18) -- cycle ;
\draw  [color={rgb, 255:red, 0; green, 0; blue, 0 }  ,draw opacity=1 ][fill={rgb, 255:red, 0; green, 0; blue, 0 }  ,fill opacity=1 ][line width=0.75]  (291.86,58.89) .. controls (291.85,57.62) and (292.91,56.58) .. (294.23,56.56) .. controls (295.55,56.55) and (296.62,57.58) .. (296.63,58.85) .. controls (296.65,60.13) and (295.59,61.17) .. (294.27,61.18) .. controls (292.95,61.19) and (291.88,60.17) .. (291.86,58.89) -- cycle ;
\draw [color={rgb, 255:red, 0; green, 0; blue, 0 }  ,draw opacity=1 ][line width=0.75]    (305.65,66.2) -- (333.02,66.41)(305.63,69.2) -- (333,69.41) ;
\draw [shift={(341.01,67.97)}, rotate = 180.44] [color={rgb, 255:red, 0; green, 0; blue, 0 }  ,draw opacity=1 ][line width=0.75]    (10.93,-3.29) .. controls (6.95,-1.4) and (3.31,-0.3) .. (0,0) .. controls (3.31,0.3) and (6.95,1.4) .. (10.93,3.29)   ;
\draw [color={rgb, 255:red, 0; green, 0; blue, 0 }  ,draw opacity=1 ][line width=0.75]    (347.15,59.51) -- (347.31,99.32) ;
\draw [color={rgb, 255:red, 74; green, 144; blue, 226 }  ,draw opacity=1 ][line width=0.75]    (347.15,59.51) -- (388.06,59.61) ;
\draw [color={rgb, 255:red, 0; green, 0; blue, 0 }  ,draw opacity=1 ][line width=0.75]    (388.06,59.61) -- (388.06,99.32) ;
\draw [color={rgb, 255:red, 0; green, 0; blue, 0 }  ,draw opacity=1 ][line width=0.75]    (347.31,99.32) -- (388.06,99.32) ;
\draw  [color={rgb, 255:red, 0; green, 0; blue, 0 }  ,draw opacity=1 ][fill={rgb, 255:red, 0; green, 0; blue, 0 }  ,fill opacity=1 ][line width=0.75]  (344.76,59.53) .. controls (344.75,58.26) and (345.81,57.22) .. (347.13,57.21) .. controls (348.45,57.19) and (349.52,58.22) .. (349.53,59.49) .. controls (349.55,60.77) and (348.49,61.81) .. (347.17,61.82) .. controls (345.85,61.83) and (344.78,60.81) .. (344.76,59.53) -- cycle ;
\draw  [color={rgb, 255:red, 0; green, 0; blue, 0 }  ,draw opacity=1 ][fill={rgb, 255:red, 0; green, 0; blue, 0 }  ,fill opacity=1 ][line width=0.75]  (385.68,59.63) .. controls (385.67,58.35) and (386.72,57.31) .. (388.04,57.3) .. controls (389.36,57.29) and (390.44,58.31) .. (390.45,59.59) .. controls (390.46,60.86) and (389.4,61.9) .. (388.08,61.91) .. controls (386.77,61.93) and (385.69,60.9) .. (385.68,59.63) -- cycle ;
\draw  [color={rgb, 255:red, 0; green, 0; blue, 0 }  ,draw opacity=1 ][fill={rgb, 255:red, 0; green, 0; blue, 0 }  ,fill opacity=1 ][line width=0.75]  (344.93,99.34) .. controls (344.92,98.07) and (345.97,97.03) .. (347.29,97.01) .. controls (348.61,97) and (349.69,98.03) .. (349.7,99.3) .. controls (349.71,100.58) and (348.65,101.62) .. (347.34,101.63) .. controls (346.02,101.64) and (344.94,100.62) .. (344.93,99.34) -- cycle ;
\draw  [color={rgb, 255:red, 0; green, 0; blue, 0 }  ,draw opacity=1 ][fill={rgb, 255:red, 0; green, 0; blue, 0 }  ,fill opacity=1 ][line width=0.75]  (385.68,99.34) .. controls (385.67,98.07) and (386.72,97.03) .. (388.04,97.01) .. controls (389.36,97) and (390.44,98.03) .. (390.45,99.3) .. controls (390.46,100.58) and (389.4,101.62) .. (388.08,101.63) .. controls (386.77,101.64) and (385.69,100.62) .. (385.68,99.34) -- cycle ;
\draw [color={rgb, 255:red, 0; green, 0; blue, 0 }  ,draw opacity=1 ][line width=0.75]    (347.15,19.5) -- (347.31,59.31) ;
\draw [color={rgb, 255:red, 0; green, 0; blue, 0 }  ,draw opacity=1 ][line width=0.75]    (347.15,19.5) -- (388.06,19.6) ;
\draw [color={rgb, 255:red, 74; green, 144; blue, 226 }  ,draw opacity=1 ][line width=0.75]    (388.06,19.6) -- (388.06,59.31) ;
\draw  [color={rgb, 255:red, 0; green, 0; blue, 0 }  ,draw opacity=1 ][fill={rgb, 255:red, 0; green, 0; blue, 0 }  ,fill opacity=1 ][line width=0.75]  (344.76,19.53) .. controls (344.75,18.25) and (345.81,17.21) .. (347.13,17.2) .. controls (348.45,17.19) and (349.52,18.21) .. (349.53,19.48) .. controls (349.55,20.76) and (348.49,21.8) .. (347.17,21.81) .. controls (345.85,21.82) and (344.78,20.8) .. (344.76,19.53) -- cycle ;
\draw  [color={rgb, 255:red, 0; green, 0; blue, 0 }  ,draw opacity=1 ][fill={rgb, 255:red, 0; green, 0; blue, 0 }  ,fill opacity=1 ][line width=0.75]  (385.68,19.62) .. controls (385.67,18.35) and (386.72,17.3) .. (388.04,17.29) .. controls (389.36,17.28) and (390.44,18.3) .. (390.45,19.58) .. controls (390.46,20.85) and (389.4,21.89) .. (388.08,21.91) .. controls (386.77,21.92) and (385.69,20.89) .. (385.68,19.62) -- cycle ;
\draw [color={rgb, 255:red, 0; green, 0; blue, 0 }  ,draw opacity=1 ][line width=0.75]    (387.29,19.8) -- (428.2,19.89) ;
\draw [color={rgb, 255:red, 0; green, 0; blue, 0 }  ,draw opacity=1 ][line width=0.75]    (428.2,19.89) -- (428.2,59.61) ;
\draw [color={rgb, 255:red, 74; green, 144; blue, 226 }  ,draw opacity=1 ][line width=0.75]    (387.45,59.61) -- (428.2,59.61) ;
\draw  [color={rgb, 255:red, 0; green, 0; blue, 0 }  ,draw opacity=1 ][fill={rgb, 255:red, 0; green, 0; blue, 0 }  ,fill opacity=1 ][line width=0.75]  (425.82,19.91) .. controls (425.81,18.64) and (426.86,17.6) .. (428.18,17.59) .. controls (429.5,17.57) and (430.58,18.6) .. (430.59,19.87) .. controls (430.6,21.15) and (429.54,22.19) .. (428.22,22.2) .. controls (426.91,22.21) and (425.83,21.19) .. (425.82,19.91) -- cycle ;
\draw  [color={rgb, 255:red, 0; green, 0; blue, 0 }  ,draw opacity=1 ][fill={rgb, 255:red, 0; green, 0; blue, 0 }  ,fill opacity=1 ][line width=0.75]  (425.82,59.63) .. controls (425.81,58.35) and (426.86,57.31) .. (428.18,57.3) .. controls (429.5,57.29) and (430.58,58.31) .. (430.59,59.59) .. controls (430.6,60.86) and (429.54,61.9) .. (428.22,61.91) .. controls (426.91,61.93) and (425.83,60.9) .. (425.82,59.63) -- cycle ;
\draw [color={rgb, 255:red, 0; green, 0; blue, 0 }  ,draw opacity=1 ][line width=0.75]    (444.93,67.08) -- (472.3,67.29)(444.9,70.08) -- (472.28,70.29) ;
\draw [shift={(480.29,68.85)}, rotate = 180.44] [color={rgb, 255:red, 0; green, 0; blue, 0 }  ,draw opacity=1 ][line width=0.75]    (10.93,-3.29) .. controls (6.95,-1.4) and (3.31,-0.3) .. (0,0) .. controls (3.31,0.3) and (6.95,1.4) .. (10.93,3.29)   ;
\draw [color={rgb, 255:red, 74; green, 144; blue, 226 }  ,draw opacity=1 ][line width=0.75]    (528.27,60.72) -- (528.43,100.53) ;
\draw [color={rgb, 255:red, 74; green, 144; blue, 226 }  ,draw opacity=1 ][line width=0.75]    (528.27,60.72) -- (569.18,60.81) ;
\draw [color={rgb, 255:red, 0; green, 0; blue, 0 }  ,draw opacity=1 ][line width=0.75]    (569.18,60.81) -- (569.18,100.53) ;
\draw [color={rgb, 255:red, 0; green, 0; blue, 0 }  ,draw opacity=1 ][line width=0.75]    (528.43,100.53) -- (569.18,100.53) ;
\draw  [color={rgb, 255:red, 0; green, 0; blue, 0 }  ,draw opacity=1 ][fill={rgb, 255:red, 0; green, 0; blue, 0 }  ,fill opacity=1 ][line width=0.75]  (525.88,60.74) .. controls (525.87,59.47) and (526.93,58.42) .. (528.24,58.41) .. controls (529.56,58.4) and (530.64,59.42) .. (530.65,60.7) .. controls (530.66,61.97) and (529.61,63.01) .. (528.29,63.03) .. controls (526.97,63.04) and (525.89,62.01) .. (525.88,60.74) -- cycle ;
\draw  [color={rgb, 255:red, 0; green, 0; blue, 0 }  ,draw opacity=1 ][fill={rgb, 255:red, 0; green, 0; blue, 0 }  ,fill opacity=1 ][line width=0.75]  (566.79,60.83) .. controls (566.78,59.56) and (567.84,58.52) .. (569.16,58.51) .. controls (570.47,58.49) and (571.55,59.52) .. (571.56,60.79) .. controls (571.58,62.07) and (570.52,63.11) .. (569.2,63.12) .. controls (567.88,63.13) and (566.81,62.11) .. (566.79,60.83) -- cycle ;
\draw  [color={rgb, 255:red, 0; green, 0; blue, 0 }  ,draw opacity=1 ][fill={rgb, 255:red, 0; green, 0; blue, 0 }  ,fill opacity=1 ][line width=0.75]  (526.05,100.55) .. controls (526.03,99.27) and (527.09,98.23) .. (528.41,98.22) .. controls (529.73,98.21) and (530.8,99.23) .. (530.82,100.51) .. controls (530.83,101.78) and (529.77,102.82) .. (528.45,102.83) .. controls (527.13,102.85) and (526.06,101.82) .. (526.05,100.55) -- cycle ;
\draw  [color={rgb, 255:red, 0; green, 0; blue, 0 }  ,draw opacity=1 ][fill={rgb, 255:red, 0; green, 0; blue, 0 }  ,fill opacity=1 ][line width=0.75]  (566.79,100.55) .. controls (566.78,99.27) and (567.84,98.23) .. (569.16,98.22) .. controls (570.47,98.21) and (571.55,99.23) .. (571.56,100.51) .. controls (571.58,101.78) and (570.52,102.82) .. (569.2,102.83) .. controls (567.88,102.85) and (566.81,101.82) .. (566.79,100.55) -- cycle ;
\draw [color={rgb, 255:red, 0; green, 0; blue, 0 }  ,draw opacity=1 ][line width=0.75]    (528.27,20.71) -- (528.43,60.52) ;
\draw [color={rgb, 255:red, 0; green, 0; blue, 0 }  ,draw opacity=1 ][line width=0.75]    (528.27,20.71) -- (569.18,20.8) ;
\draw [color={rgb, 255:red, 74; green, 144; blue, 226 }  ,draw opacity=1 ][line width=0.75]    (569.18,20.8) -- (569.18,60.52) ;
\draw  [color={rgb, 255:red, 0; green, 0; blue, 0 }  ,draw opacity=1 ][fill={rgb, 255:red, 0; green, 0; blue, 0 }  ,fill opacity=1 ][line width=0.75]  (525.88,20.73) .. controls (525.87,19.46) and (526.93,18.41) .. (528.24,18.4) .. controls (529.56,18.39) and (530.64,19.42) .. (530.65,20.69) .. controls (530.66,21.96) and (529.61,23.01) .. (528.29,23.02) .. controls (526.97,23.03) and (525.89,22.01) .. (525.88,20.73) -- cycle ;
\draw  [color={rgb, 255:red, 0; green, 0; blue, 0 }  ,draw opacity=1 ][fill={rgb, 255:red, 0; green, 0; blue, 0 }  ,fill opacity=1 ][line width=0.75]  (566.79,20.83) .. controls (566.78,19.55) and (567.84,18.51) .. (569.16,18.5) .. controls (570.47,18.49) and (571.55,19.51) .. (571.56,20.78) .. controls (571.58,22.06) and (570.52,23.1) .. (569.2,23.11) .. controls (567.88,23.12) and (566.81,22.1) .. (566.79,20.83) -- cycle ;
\draw [color={rgb, 255:red, 0; green, 0; blue, 0 }  ,draw opacity=1 ][line width=0.75]    (568.41,21) -- (609.32,21.1) ;
\draw [color={rgb, 255:red, 0; green, 0; blue, 0 }  ,draw opacity=1 ][line width=0.75]    (609.32,21.1) -- (609.32,60.81) ;
\draw [color={rgb, 255:red, 74; green, 144; blue, 226 }  ,draw opacity=1 ][line width=0.75]    (568.57,60.81) -- (609.32,60.81) ;
\draw  [color={rgb, 255:red, 0; green, 0; blue, 0 }  ,draw opacity=1 ][fill={rgb, 255:red, 0; green, 0; blue, 0 }  ,fill opacity=1 ][line width=0.75]  (606.93,21.12) .. controls (606.92,19.85) and (607.98,18.8) .. (609.3,18.79) .. controls (610.61,18.78) and (611.69,19.8) .. (611.7,21.08) .. controls (611.72,22.35) and (610.66,23.39) .. (609.34,23.41) .. controls (608.02,23.42) and (606.95,22.39) .. (606.93,21.12) -- cycle ;
\draw  [color={rgb, 255:red, 0; green, 0; blue, 0 }  ,draw opacity=1 ][fill={rgb, 255:red, 0; green, 0; blue, 0 }  ,fill opacity=1 ][line width=0.75]  (606.93,60.83) .. controls (606.92,59.56) and (607.98,58.52) .. (609.3,58.51) .. controls (610.61,58.49) and (611.69,59.52) .. (611.7,60.79) .. controls (611.72,62.07) and (610.66,63.11) .. (609.34,63.12) .. controls (608.02,63.13) and (606.95,62.11) .. (606.93,60.83) -- cycle ;
\draw [color={rgb, 255:red, 0; green, 0; blue, 0 }  ,draw opacity=1 ][line width=0.75]    (569.18,60.52) -- (584.85,100.27) ;
\draw [color={rgb, 255:red, 0; green, 0; blue, 0 }  ,draw opacity=1 ][line width=0.75]    (584.85,100.56) -- (626.42,100.36) ;
\draw [color={rgb, 255:red, 0; green, 0; blue, 0 }  ,draw opacity=1 ][line width=0.75]    (609.32,60.81) -- (626.42,100.36) ;
\draw  [color={rgb, 255:red, 0; green, 0; blue, 0 }  ,draw opacity=1 ][fill={rgb, 255:red, 0; green, 0; blue, 0 }  ,fill opacity=1 ][line width=0.75]  (624.03,100.39) .. controls (624.02,99.11) and (625.08,98.07) .. (626.39,98.06) .. controls (627.71,98.05) and (628.79,99.07) .. (628.8,100.34) .. controls (628.81,101.62) and (627.75,102.66) .. (626.44,102.67) .. controls (625.12,102.68) and (624.04,101.66) .. (624.03,100.39) -- cycle ;
\draw  [color={rgb, 255:red, 0; green, 0; blue, 0 }  ,draw opacity=1 ][fill={rgb, 255:red, 0; green, 0; blue, 0 }  ,fill opacity=1 ][line width=0.75]  (582.46,100.58) .. controls (582.45,99.31) and (583.51,98.26) .. (584.82,98.25) .. controls (586.14,98.24) and (587.22,99.27) .. (587.23,100.54) .. controls (587.24,101.81) and (586.18,102.86) .. (584.87,102.87) .. controls (583.55,102.88) and (582.47,101.86) .. (582.46,100.58) -- cycle ;
\draw [color={rgb, 255:red, 0; green, 0; blue, 0 }  ,draw opacity=1 ][line width=0.75]    (388.43,59.46) -- (404.1,99.21) ;
\draw [color={rgb, 255:red, 0; green, 0; blue, 0 }  ,draw opacity=1 ][line width=0.75]    (404.1,99.5) -- (445.67,99.31) ;
\draw [color={rgb, 255:red, 0; green, 0; blue, 0 }  ,draw opacity=1 ][line width=0.75]    (428.57,59.75) -- (445.67,99.31) ;
\draw  [color={rgb, 255:red, 0; green, 0; blue, 0 }  ,draw opacity=1 ][fill={rgb, 255:red, 0; green, 0; blue, 0 }  ,fill opacity=1 ][line width=0.75]  (443.28,99.33) .. controls (443.27,98.05) and (444.33,97.01) .. (445.65,97) .. controls (446.96,96.99) and (448.04,98.01) .. (448.05,99.28) .. controls (448.07,100.56) and (447.01,101.6) .. (445.69,101.61) .. controls (444.37,101.62) and (443.3,100.6) .. (443.28,99.33) -- cycle ;
\draw  [color={rgb, 255:red, 0; green, 0; blue, 0 }  ,draw opacity=1 ][fill={rgb, 255:red, 0; green, 0; blue, 0 }  ,fill opacity=1 ][line width=0.75]  (401.71,99.52) .. controls (401.7,98.25) and (402.76,97.21) .. (404.08,97.19) .. controls (405.39,97.18) and (406.47,98.21) .. (406.48,99.48) .. controls (406.49,100.76) and (405.44,101.8) .. (404.12,101.81) .. controls (402.8,101.82) and (401.72,100.8) .. (401.71,99.52) -- cycle ;
\draw [color={rgb, 255:red, 0; green, 0; blue, 0 }  ,draw opacity=1 ][line width=0.75]    (487.1,59.54) -- (487.27,99.35) ;
\draw [color={rgb, 255:red, 0; green, 0; blue, 0 }  ,draw opacity=1 ][line width=0.75]    (487.1,60.54) -- (528.02,60.64) ;
\draw [color={rgb, 255:red, 0; green, 0; blue, 0 }  ,draw opacity=1 ][line width=0.75]    (487.27,100.35) -- (528.02,100.35) ;
\draw  [color={rgb, 255:red, 0; green, 0; blue, 0 }  ,draw opacity=1 ][fill={rgb, 255:red, 0; green, 0; blue, 0 }  ,fill opacity=1 ][line width=0.75]  (484.72,60.56) .. controls (484.71,59.29) and (485.77,58.25) .. (487.08,58.23) .. controls (488.4,58.22) and (489.48,59.25) .. (489.49,60.52) .. controls (489.5,61.8) and (488.44,62.84) .. (487.13,62.85) .. controls (485.81,62.86) and (484.73,61.84) .. (484.72,60.56) -- cycle ;
\draw  [color={rgb, 255:red, 0; green, 0; blue, 0 }  ,draw opacity=1 ][fill={rgb, 255:red, 0; green, 0; blue, 0 }  ,fill opacity=1 ][line width=0.75]  (484.88,100.37) .. controls (484.87,99.1) and (485.93,98.06) .. (487.25,98.04) .. controls (488.56,98.03) and (489.64,99.06) .. (489.65,100.33) .. controls (489.67,101.6) and (488.61,102.65) .. (487.29,102.66) .. controls (485.97,102.67) and (484.89,101.65) .. (484.88,100.37) -- cycle ;

\draw (23.89,110) node [anchor=north west][inner sep=0.75pt]   [align=left] {$T_2$};
\draw (121.35,110.89) node [anchor=north west][inner sep=0.75pt]   [align=left] {$T_3$};
\draw (228.09,110) node [anchor=north west][inner sep=0.75pt]   [align=left] {$T_4$};
\draw (380.9,110) node [anchor=north west][inner sep=0.75pt]   [align=left] {$T_5$};
\draw (547.05,110.95) node [anchor=north west][inner sep=0.75pt]   [align=left] {$T_6$};
\draw (96.87,45.94) node [anchor=north west][inner sep=0.75pt]  [font=\scriptsize] [align=left] {$a_2$};
\draw (156.74,45.94) node [anchor=north west][inner sep=0.75pt]  [font=\scriptsize] [align=left] {$b_2$};
\draw (96.87,8.42) node [anchor=north west][inner sep=0.75pt]  [font=\scriptsize] [align=left] {$x_2$};
\draw (157.15,9.75) node [anchor=north west][inner sep=0.75pt]  [font=\scriptsize] [align=left] {$y_2$};
\draw (250.97,4.01) node [anchor=north west][inner sep=0.75pt]  [font=\scriptsize] [align=left] {$a_3$};
\draw (255.52,62.84) node [anchor=north west][inner sep=0.75pt]  [font=\scriptsize] [align=left] {$b_3$};
\draw (288.25,4.6) node [anchor=north west][inner sep=0.75pt]  [font=\scriptsize] [align=left] {$x_3$};
\draw (288.25,65.84) node [anchor=north west][inner sep=0.75pt]  [font=\scriptsize] [align=left] {$y_3$};
\draw (387.68,49.02) node [anchor=north west][inner sep=0.75pt]  [font=\scriptsize] [align=left] {$a_4$};
\draw (429.12,46.02) node [anchor=north west][inner sep=0.75pt]  [font=\scriptsize] [align=left] {$b_4$};
\draw (528.05,103.55) node [anchor=north west][inner sep=0.75pt]  [font=\scriptsize] [align=left] {$a_5$};
\draw (489.27,103.35) node [anchor=north west][inner sep=0.75pt]  [font=\scriptsize] [align=left] {$x_5$};
\draw (491.08,47.57) node [anchor=north west][inner sep=0.75pt]  [font=\scriptsize] [align=left] {$y_5$};
\draw (530.78,47.57) node [anchor=north west][inner sep=0.75pt]  [font=\scriptsize] [align=left] {$b_5$};
\draw (401.56,102.27) node [anchor=north west][inner sep=0.75pt]  [font=\scriptsize] [align=left] {$x_4$};
\draw (437,101.68) node [anchor=north west][inner sep=0.75pt]  [font=\scriptsize] [align=left] {$y_4$};

\end{tikzpicture}
    \caption{Tetragonal graphs}
    \label{fig:tetragonal}
\end{figure}

By applying Proposition~\ref{prop:Q-graph} to $uv\in E(\partial T)$ in a tetragonal graph $T$ of order $2m$, we see that $T$ contains cycles (each includes $uv$) of all lengths in $\{4,6,\dots,2m\}$.

Before defining the specific core subgraph in bipartite graphs, we motivate the need for a more robust structure than a simple tetragonal graph. 
The analysis of the bipartite case is inherently more challenging in our approach because tetragonal graphs are significantly less efficient at generating admissible paths compared to trigonal graphs. 
To illustrate this disparity, let $T_1$ be a maximum trigonal subgraph of a non-bipartite graph $G_1$, and let $T_2$ be a maximum tetragonal subgraph of a bipartite graph $G_2$, assuming $|T_1| = |T_2| = 2m$ and that both $G_1$ and $G_2$ have sufficiently large minimum degrees.
Then the maximality of $|T_1|$ and the bipartiteness of $G_2$ implies that $\delta(G_i-T_i)\geq \delta(G_i)-m$.
However, a critical difference arises in the number of generated path lengths. Comparing Proposition~\ref{prop: T-graph} and Proposition~\ref{prop:Q-graph}, for two vertices at distance $d$ on the boundary, $T_1$ provides a set of path lengths $[d, 2m-d]$ with cardinality $2m-2d+1$, whereas $T_2$ yields only $\{d, d+2, \dots, 2m-d\}$ with cardinality $m-d+1$. This substantial reduction in available paths necessitates an extension of the tetragonal subgraph to include vertices with high degrees, thereby forming a stronger core.

We then introduce a specific class of tetragonal subgraphs for the subsequent proofs.
\begin{dfn}\label{dfn:best Q}
	We say that $T$ is an {\bf optimal} tetragonal subgraph of a bipartite graph $G$ if the following conditions hold.
	\begin{itemize}
		\item[(1)] $T$ is a tetragonal subgraph of $G$ with maximum order;
		\item[(2)] Subject to condition (1), the number of edges in $G[V(T)]$ is maximized. 
	\end{itemize}
\end{dfn}

The following lemma summarizes the key properties of optimal tetragonal subgraphs. In Section~6 (bipartite case), we employ the subgraph induced by $V(T)\cup R$ as the core subgraph.
\begin{lem}\label{lem:operation}
	Let $T$ be an optimal tetragonal subgraph of a bipartite graph $G$.
    If $|T| = 2m \ge 6$, then the following hold:
\begin{itemize}
    \item[(1)] For every $v \in V(G - T)$, $\deg_T(v) \leq m$. The equality holds only if $G[V(T)] \simeq K_{m,m}$.
    \item[(2)] $R := \{v \in V(G - T) : \deg_T(v) \ge m - 1\}$ is an independent set.
    \item[(3)] If $m\geq 4$, then $R$ is contained in one of the two partite sets of $G$.
    \item[(4)] For every $v \in V(G - T - R)$, exactly one of the following holds:
    \begin{itemize}
		\item[(4.1)] $\deg_R(v)=0$ and $\deg_T(v)\leq m-2$;
		\item[(4.2)] $\deg_R(v)=1$ and $\deg_T(v)=0$.
	\end{itemize}
\end{itemize}
\end{lem}

\begin{proof}
    For (1), since $G$ is bipartite, $\deg_T(v) \le m$ trivially holds. Suppose $\deg_T(v) = m$ for some $v \in V(G-T)$. Let $(A, B)$ be the partite sets of $G$ with $v\in B$.
    Then $v$ is adjacent to every vertex in $V(T)\cap A$.
    If $G[T] \not\simeq K_{m,m}$, then there exists $u \in B$ with $\deg_T(u) < m$. The set $(V(T) \setminus \{u\}) \cup \{v\}$ spans a tetragonal graph that induces strictly more edges than $T$, contradicting the maximality of $e(G[V(T)])$. Hence $G[T] \simeq K_{m,m}$, which proves (1).

    For (2), suppose to the contrary that there are two adjacent vertices $u_1, u_2 \in R$. By Lemma~\ref{lem:short dist}~(1), there exist $v_1, v_2 \in V(T)$ such that $u_1v_1, u_2v_2 \in E(G)$ and $v_1v_2 \in E(\partial T)$. It is routine to verify that $V(T)\cup \{u_1,u_2\}$ spans a larger tetragonal graph, whose boundary cycle is obtained from $\partial T$ by replacing the edge $v_1v_2$ with the path $v_1u_1u_2v_2$, a contradiction. Hence $R$ is an independent set, proving (2).

    For (3), let $(A, B)$ be the partite sets of $G$. Suppose to the contrary that both $R\cap A$ and $R\cap B$ are non-empty; then there exist $u_1\in R\cap A$ and $u_2\in R\cap B$. We claim that $V(T)\cup \{u_1,u_2\}$ would span a larger tetragonal graph.
    Let $\partial T=v_{0}\dots v_{2m-1}v_{0}$, where indices are taken modulo $2m$.
    Since $\deg_{T}(u_1), \deg_{T}(u_{2}) \ge m-1$, we may select $v_i\in V(T)\cap B$ and $v_j\in V(T)\cap A$ such that $u_1$ is adjacent to every vertex in $(V(T)\cap B)\setminus\{v_i\}$, and $u_2$ is adjacent to every vertex in $(V(T)\cap A)\setminus\{v_j\}$.

    If $v_{i}v_{j}\in E(\partial T)$, we may assume that $j=i-1$.
    Note that the edge $v_{i}v_{j}$ is contained in a unique 4-cycle $F=v_{i-1}v_iv_qv_p$ in $T$.
    Since $G$ is bipartite, $i$ and $p$ have the same parity, while $q$ has the opposite parity.
    See Figure~\ref{fig:lem3a} for an illustration, where $\partial T$ is the outer cycle on $2m=22$ vertices.
    We label the 4-cycles $u_{1}v_{s-1}v_{s}v_{s+1}$ (for $s=i-3,i-5,\dots,q+2$), $u_2v_qv_pv_{p+1}$, $u_{2}v_{t-1}v_{t}v_{t+1}$ (for $t=q-1,q-3,\dots,i+2$), and the 4-cycle $F$ by $1,2,\dots,m$.
    Recall the iterative process in Definition~\ref{dfn:Q graph}.
    As visualized in Figure~\ref{fig:lem3a}, let the 4-cycle labeled 1 be the initial tetragonal graph.
    By sequentially adding the 4-cycles labeled $2,3,\dots,m$, we obtain a tetragonal graph with vertex set $V(T)\cup \{u_1,u_2\}$, contradicting the maximality of $m$.

    Now consider the case where $j-i\notin\{-1,1\}$. It follows that $v_{i-1},v_{i+1}\in N_{T}(u_{2})$. 
    We have $|j-i|\geq 2$ since $i$ and $j$ have different parities.
	Since $m\ge 4$, $v_{i+3}$ is distinct from $v_{i-3}$. 
    We may assume without loss of generality that $j\neq i-3$, so $u_2v_{i-3}\in E(G)$.
    See Figure~\ref{fig:lem3b} for an example with $|T|=22$. The 4-cycles $u_{2}v_{i-1}v_iv_{i+1}$, $u_{2}v_{i-3}v_{i-2}v_{i-1}$, and $u_{1}v_{t-1}v_{t}v_{t+1}$ (for $t=i-3,i-5,\dots,i+3$) are labeled $1,2,\dots, m$.
    It is straightforward to verify that, starting from the 4-cycle labeled 1 and sequentially adding the cycles labeled $2,3,\dots,m$, we obtain a tetragonal graph with vertex set $V(T)\cup \{u_1,u_2\}$, contradicting the maximality of $m$.
    This completes the proof of (3).

    \begin{figure}[htbp]
    \centering

    \tikzset{
        lbl/.style={circle, fill=white, inner sep=1.5pt, font=\bfseries\small, opacity=1, text opacity=1, anchor=center},
        txt/.style={circle, fill=white, inner sep=0.5pt, font=\scriptsize, opacity=0.85, text opacity=1, anchor=center}
    }
    \begin{subfigure}[t]{0.35\textwidth}
        \centering
        \begin{tikzpicture}[scale=0.9, baseline=0] 
            \useasboundingbox (-3.05,-3.3) rectangle (3.05,3.3);
            \def\R{2.8} \def\N{22} \def\startAng{100}

            \foreach \k in {-4,...,25} {
                \coordinate (V\k) at ({\R*cos(\startAng - \k*360/\N)}, {\R*sin(\startAng - \k*360/\N)});
            }
            \coordinate (U2) at ({\R*0.2*cos(\startAng - 3.5*360/\N)}, {\R*0.2*sin(\startAng - 3.5*360/\N)});
            \coordinate (U1) at ({\R*0.5*cos(\startAng - 12*360/\N)}, {\R*0.5*sin(\startAng - 12*360/\N)});

            \path[pattern={Lines[angle=90, distance=10pt]}, pattern color=red!60] 
                (U2) -- (V7) -- (V14) -- (V15) -- cycle;
            \draw[thin, red!80!black] (U2) -- (V15);
            \node[lbl, text=red!80!black] at ($ 0.35*(U2) + 0.35*(V7) + 0.15*(V14) + 0.15*(V15) $) {7};

            \foreach \k/\num [count=\c] in {5/8, 3/9, 1/10} {
                \pgfmathsetmacro{\nextK}{int(\k+2)}
                \ifodd\c \def\ang{45} \else \def\ang{-45} \fi
                \path[pattern={Lines[angle=\ang, distance=10pt]}, pattern color=red!60] 
                    (U2) -- (V\k) arc [start angle=\startAng-\k*360/\N, end angle=\startAng-\nextK*360/\N, radius=\R] -- cycle;
                \draw[thin, red!80!black] (U2) -- (V\k); \draw[thin, red!80!black] (U2) -- (V\nextK);
                \pgfmathsetmacro{\midAng}{\startAng - (\k+1)*360/\N}
                \node[lbl, text=red!80!black] at ($ (U2)!0.6!({\R*cos(\midAng)}, {\R*sin(\midAng)}) $) {\num};
            }

            \foreach \k/\num [count=\c] in {18/1, 16/2, 14/3, 12/4, 10/5, 8/6} {
                \pgfmathsetmacro{\nextK}{int(\k+2)}
                \ifodd\c \def\ang{0} \else \def\ang{90} \fi
                \path[pattern={Lines[angle=\ang, distance=10pt]}, pattern color=blue!60] 
                    (U1) -- (V\k) arc [start angle=\startAng-\k*360/\N, end angle=\startAng-\nextK*360/\N, radius=\R] -- cycle;
                \draw[thin, blue!80!black] (U1) -- (V\k); \draw[thin, blue!80!black] (U1) -- (V\nextK);
                \pgfmathsetmacro{\midAng}{\startAng - (\k+1)*360/\N}
                \node[lbl, text=blue!80!black] at ($ (U1)!0.65!({\R*cos(\midAng)}, {\R*sin(\midAng)}) $) {\num};
            }

            \draw[thick] (0,0) circle (\R);
            \draw[dashed, very thick, black!80] (V0) -- (V7) -- (V14) -- (V-1); 
            \node[lbl, text=black] at ($ (V0)!0.3!(V12) $) {11};

            \foreach \k in {0,...,21} \fill[black] (V\k) circle (1.2pt);
            \fill[black] (U1) circle (1.2pt) node[txt, label={[label distance=0.1pt]70:$u_1$}] {};
            \fill[black] (U2) circle (1.2pt) node[txt, label={[label distance=0.1pt]180:$u_2$}] {};

            \node[inner sep=0pt] at (V-1) [txt, label={[label distance=0.1pt]90:$v_{i-1}$}] {};
            \node[inner sep=0pt] at (V0)  [txt, label={[label distance=0.1pt]90:$v_i$}] {};
            \node[inner sep=0pt] at (V7)  [txt, label={[label distance=0.1pt]-10:$v_q$}] {};
            \node[inner sep=0pt] at (V14) [txt, label={[label distance=0.1pt]230:$v_p$}] {};
            \node[inner sep=0pt] at (V15) [txt, label={[label distance=0.1pt]200:$v_{p+1}$}] {};

        \end{tikzpicture}
        \caption{$j=i-1$}
        \label{fig:lem3a}
    \end{subfigure}
    \hspace{1em}
    \begin{subfigure}[t]{0.39\textwidth}
        \centering
        \begin{tikzpicture}[scale=0.9, baseline=0] 
            \useasboundingbox (-3.05,-3.3) rectangle (3.05,3.3);
            \def\R{2.8} \def\N{22} \def\startAng{90}

            \foreach \k in {-6,...,25} {
                \coordinate (V\k) at ({\R*cos(\startAng - \k*360/\N)}, {\R*sin(\startAng - \k*360/\N)});
            }
            \coordinate (U2) at ({\R*0.6*cos(\startAng - 0*360/\N)}, {\R*0.6*sin(\startAng - 0*360/\N)});
            \coordinate (U1) at ({\R*0.3*cos(\startAng - 11*360/\N)}, {\R*0.3*sin(\startAng - 11*360/\N)});

            \foreach \k/\num [count=\c] in {21/1, 19/2} {
                \pgfmathsetmacro{\nextK}{int(\k+2)}
                \ifodd\c \def\ang{-45} \else \def\ang{45} \fi
                \path[pattern={Lines[angle=\ang, distance=10pt]}, pattern color=red!60] 
                    (U2) -- (V\k) arc [start angle=\startAng-\k*360/\N, end angle=\startAng-\nextK*360/\N, radius=\R] -- cycle;
                \draw[thin, red!80!black] (U2) -- (V\k); \draw[thin, red!80!black] (U2) -- (V\nextK);
                \pgfmathsetmacro{\midAng}{\startAng - (\k+1)*360/\N}
                \node[lbl, text=red!80!black] at ($ (U2)!0.6!({\R*cos(\midAng)}, {\R*sin(\midAng)}) $) {\num};
            }

            \foreach \k/\num [count=\c] in {18/3, 16/4, 14/5, 12/6, 10/7, 8/8, 6/9, 4/10, 2/11} {
                \pgfmathsetmacro{\nextK}{int(\k+2)}
                \ifodd\c \def\ang{0} \else \def\ang{90} \fi
                \path[pattern={Lines[angle=\ang, distance=10pt]}, pattern color=blue!60] 
                    (U1) -- (V\k) arc [start angle=\startAng-\k*360/\N, end angle=\startAng-\nextK*360/\N, radius=\R] -- cycle;
                \draw[thin, blue!80!black] (U1) -- (V\k); \draw[thin, blue!80!black] (U1) -- (V\nextK);
                \pgfmathsetmacro{\midAng}{\startAng - (\k+1)*360/\N}
                \node[lbl, text=blue!80!black] at ($ (U1)!0.7!({\R*cos(\midAng)}, {\R*sin(\midAng)}) $) {\num};
            }

            \draw[thick] (0,0) circle (\R);
            \draw[dashed, very thick, black!80] (V17) -- (V0) -- (V7) -- (V12) -- cycle;

            \foreach \k in {0,...,21} \fill[black] (V\k) circle (1.2pt);
            \fill[black] (U1) circle (1.2pt) node[txt, label={[label distance=0.1pt]260:$u_1$}] {};
            \fill[black] (U2) circle (1.2pt) node[txt, label={[label distance=0.1pt]200:$u_2$}] {};

            \node[inner sep=0pt] at (V0)  [txt, label={[label distance=0.1pt]90:$v_i$}] {};
            \node[inner sep=0pt] at (V7)  [txt, label={[label distance=0.1pt]-10:$v_q$}] {};
            \node[inner sep=0pt] at (V12) [txt, label={[label distance=0.1pt]270:$v_p$}] {};
            \node[inner sep=0pt] at (V17) [txt, label={[label distance=0.1pt]180:$v_j$}] {};
            
            \node[inner sep=0pt] at (V20) [txt, label={[label distance=0.1pt]110:$v_{i-2}$}] {};
            \node[inner sep=0pt] at (V21) [txt, label={[label distance=0.05pt]90:$v_{i-1}$}] {};
            \node[inner sep=0pt] at (V1)  [txt, label={[label distance=0.03pt]90:$v_{i+1}$}] {};
            \node[inner sep=0pt] at (V2)  [txt, label={[label distance=0.1pt]80:$v_{i+2}$}] {};
            \node[inner sep=0pt] at (V19)  [txt, label={[label distance=0.1pt]130:$v_{i-3}$}] {};
        \end{tikzpicture}
        \caption{$|j-i|>1$}
        \label{fig:lem3b}
    \end{subfigure}
    \par\vspace{0em}
    \begin{subfigure}[t]{0.39\textwidth}
        \centering
        \begin{tikzpicture}[scale=0.46, baseline=0] 
            \useasboundingbox (-5.8,0) rectangle (5.8,10);
            \def\R{6} \def\AngStart{160} \def\AngEnd{20}
            
            \coordinate (Vi-3) at (\AngStart:\R);
            \coordinate (Vi+3) at (\AngEnd:\R);
            \coordinate (Vi-2) at (145:\R); 
            \coordinate (Vi-1) at (120:\R); 
            \coordinate (Vi)   at (85:\R);   
            \coordinate (Vi+1) at (60:\R);  
            \coordinate (Vi+2) at (35:\R);  

            \coordinate (MidChord) at ($(Vi-2)!0.5!(Vi)$);
            \coordinate (TopCenter) at ($(MidChord)!2.5cm!90:(Vi)$);
            \coordinate (ChordVec) at ($(Vi)-(Vi-2)$);
            \coordinate (Pp) at ($(TopCenter) - 0.2*(ChordVec)$); 
            \coordinate (P)  at ($(TopCenter) + 0.2*(ChordVec)$); 
            \coordinate (MidPp) at ($(Pp)!0.5!(P)$);
            \path let \p1 = ($(P)-(MidPp)$), \n1 = {veclen(\x1,\y1)} in 
                  coordinate (W) at ($(MidPp)!{\n1/tan(60)}!90:(P)$);

            \fill[black!15] (Vi-1) -- (Vi+2) arc (35:\AngEnd:\R) -- (Vi-3) arc (\AngStart:120:\R) -- cycle;
            \node[lbl, text=black] at ($ 0.25*(Vi-1) + 0.25*(Vi+2) + 0.25*(Vi+3) + 0.25*(Vi-3) $) {1};

            \fill[blue!15] (W) -- (Pp) -- (Vi-2) -- (P) -- cycle;
            \node[lbl, text=blue] at ($ 0.25*(W) + 0.25*(Pp) + 0.25*(Vi-2) + 0.25*(P) $) {3};

            \fill[red!15] (P) -- (Vi-2) arc (145:85:\R) -- cycle;
            \node[lbl, text=red] at ($ 0.45*(P) + 0.15*(Vi-2) + 0.4*(Vi) $) {2};

            \draw[thick] (Vi-3) arc (\AngStart:\AngEnd:\R); 
            \draw[thick] (Vi-3) -- (Vi+3);                  
            \draw[dashed, thick] (Vi-1) -- (Vi+2);
            \draw[thin] (W) -- (Pp); \draw[thin] (W) -- (P);
            \draw[thin] (Pp) -- (Vi-2); \draw[thin] (P) -- (Vi);
            \draw[thin] (P) -- (Vi-2);

            \foreach \pt/\lab/\pos in {Vi-2/v_{i-2}/160, Vi-1/v_{i-1}/100, Vi/v_{i}/90, Vi+1/v_{i+1}/80, Vi+2/v_{i+2}/30} {
                \fill[black] (\pt) circle (1.5pt);
                \node[inner sep=0pt] at (\pt) [label={[label distance=0.1pt, inner sep=1pt]\pos:$\lab$}] {};
            }
            \fill[black] (Pp) circle (1.5pt) node[left, xshift=-1pt] {$p'$};
            \fill[black] (P)  circle (1.5pt) node[right, xshift=1pt] {$p$};
            \fill[black] (W)  circle (1.5pt) node[above, yshift=1pt] {$w$};
        \end{tikzpicture}
        \caption{$\big|E(\{v_{i-2},v_i\},\{p,p'\})\big|\geq 3$}
        \label{fig:lem4a}
    \end{subfigure}
    \hspace{1em}
    \begin{subfigure}[t]{0.39\textwidth}
        \centering
        \begin{tikzpicture}[scale=0.46, baseline=0] 
            \useasboundingbox (-5.8,0) rectangle (5.8,10);
            
            \def\R{6} \def\AngStart{160} \def\AngEnd{20}
            
            \coordinate (Vi-3) at (\AngStart:\R);
            \coordinate (Vi+3) at (\AngEnd:\R);
            \coordinate (Vi-2) at (145:\R); 
            \coordinate (Vi-1) at (120:\R); 
            \coordinate (Vi)   at (85:\R);   
            \coordinate (Vi+1) at (60:\R);  
            \coordinate (Vi+2) at (35:\R);  

            \coordinate (MidChord) at ($(Vi)!0.5!(Vi+2)$);
            \coordinate (TopCenter) at ($(MidChord)!2.5cm!90:(Vi+2)$);
            \coordinate (ChordVec) at ($(Vi+2)-(Vi)$);
            \coordinate (Pp) at ($(TopCenter) - 0.2*(ChordVec)$); 
            \coordinate (P)  at ($(TopCenter) + 0.2*(ChordVec)$); 
            \coordinate (MidPp) at ($(Pp)!0.5!(P)$);
            \path let \p1 = ($(P)-(MidPp)$), \n1 = {veclen(\x1,\y1)} in 
                  coordinate (W) at ($(MidPp)!{\n1/tan(60)}!90:(P)$);

            \fill[black!15] (Vi-1) -- (Vi+2) arc (35:\AngEnd:\R) -- (Vi-3) arc (\AngStart:120:\R) -- cycle;
            \node[lbl, text=black] at ($ 0.25*(Vi-1) + 0.25*(Vi+2) + 0.25*(Vi+3) + 0.25*(Vi-3) $) {1};

            \fill[blue!15] (W) -- (Pp) -- (Vi) -- (P) -- cycle;
            \node[lbl, text=blue] at ($ 0.25*(W) + 0.25*(Pp) + 0.25*(Vi) + 0.25*(P) $) {3};

            \fill[red!15] (P) -- (Vi+2) -- (Vi-1) arc (120:85:\R) -- cycle;
            \node[lbl, text=red] at ($ 0.25*(P) + 0.25*(Vi+2) + 0.25*(Vi-1) + 0.25*(Vi) $) {2};

            \draw[thick] (Vi-3) arc (\AngStart:\AngEnd:\R); 
            \draw[thick] (Vi-3) -- (Vi+3);                  
            \draw[dashed, thick] (Vi-1) -- (Vi+2);
            \draw[thin] (W) -- (Pp); \draw[thin] (W) -- (P);
            \draw[thin] (Pp) -- (Vi); \draw[thin] (P) -- (Vi);
            \draw[thin] (P) -- (Vi+2);

            \foreach \pt/\lab/\pos in {Vi-2/v_{i-2}/160, Vi-1/v_{i-1}/100, Vi/v_{i}/90, Vi+1/v_{i+1}/80, Vi+2/v_{i+2}/30} {
                \fill[black] (\pt) circle (1.5pt);
                \node[inner sep=0pt] at (\pt) [label={[label distance=0.1pt, inner sep=1pt]\pos:$\lab$}] {};
            }
            \fill[black] (Pp) circle (1.5pt) node[left, xshift=-1pt] {$p'$};
            \fill[black] (P)  circle (1.5pt) node[right, xshift=1pt] {$p$};
            \fill[black] (W)  circle (1.5pt) node[above, yshift=1pt] {$w$};
        \end{tikzpicture}
        \caption{$\big|E(\{v_i,v_{i+2}\},\{p,p'\})\big|\geq 3$}
        \label{fig:lem4b}
    \end{subfigure}
    \caption{Forming larger tetragonal graphs in Lemma~\ref{lem:operation}}
    \label{fig:lem4}
\end{figure}
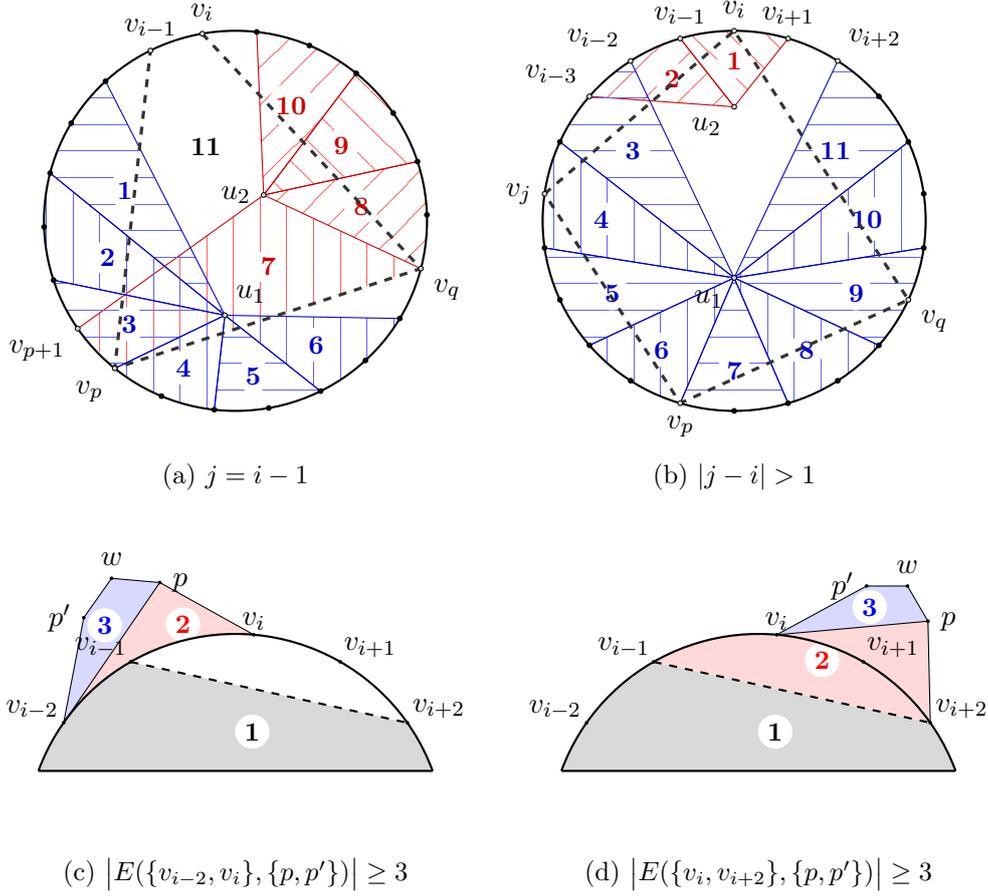

    For (4), select an arbitrary vertex $w \in V(G - T - R)$. If $\deg_R(w)=0$, then $\deg_T(w) \le m - 2$ by the definition of $R$.
    Now assume $\deg_R(w)>0$. First, we show $\deg_T(w) = 0$. Let $p \in N_R(w)$. If there exists $v \in N_T(w)$, since $\deg_T(p) \ge m - 1$, there exists $v'\in N_T(p)$ such that $vv'\in E(\partial T)$.
	Then $V(T) \cup \{p, w\}$ spans a larger tetragonal graph (with the boundary cycle obtained by replacing the edge $vv'$ on $\partial T$ with the path $vwpv'$).
	This contradiction implies that $\deg_T(w)=0$.

    It remains to show that $\deg_R(w) = 1$. Suppose for the sake of contradiction that there is a vertex $p' \in N_R(w)$ distinct from $p$.
    Assume that $\partial T=v_{0}\dots v_{2m-1}v_{0}$, with indices taken modulo $2m$.
    We say that an ordered pair of indices $(i,j)\in [0,2m-1]^2$ is \textbf{bad} if:
    \begin{itemize}
        \item[(\romannumeral1)] $j\in \{i-1,i+1\}$;
        \item[(\romannumeral2)] $T-\{v_i,v_j\}$ is a tetragonal graph;
        \item[(\romannumeral3)] $\max\big\{\big|E(\{v_{i-2},v_i\},\{p,p'\})\big|, \big|E(\{v_{i+2},v_i\},\{p,p'\})\big|\big\} \ge 3$.
    \end{itemize}
    We claim that if $(i,j)$ is bad, then $V(T)\setminus\{v_{j}\}\cup \{p,p',w\}$ spans a larger tetragonal graph.
    Assume without loss of generality that $j=i+1$; then the condition (\romannumeral2) that $T-\{v_i,v_j\}$ is a tetragonal graph is equivalent to $v_{i-1}v_{i+2}\in E(T)$.
    The claim is illustrated in Figures~\ref{fig:lem4a} and~\ref{fig:lem4b} for the cases $|E(\{v_{i-2},v_i\},\{p,p'\})|\geq 3$ and $|E(\{v_{i+2},v_i\},\{p,p'\})|\geq 3$.
    The gray region represents the tetragonal subgraph with vertex set $V(T) \setminus \{v_i, v_{i+1}\}$, whose boundary cycle is derived from $\partial T$ by replacing the path $v_{i-1}v_iv_{i+1}v_{i+2}$ with the edge $v_{i-1}v_{i+2}$. The `red' and `blue' regions correspond to the 4-cycles $pv_{i-2}v_{i-1}v_{i}$ and $wp'v_{i-2}p$ in Figure~\ref{fig:lem4a}, and to $pv_{i}v_{i-1}v_{i+2}$ and $wp'v_{i}p$ in Figure~\ref{fig:lem4b}, respectively.
    By adding the red 4-cycle and then the blue 4-cycle to the gray tetragonal graph, we obtain a tetragonal graph with vertex set $V(T)\setminus\{v_{j}\}\cup \{p,p',w\}$, proving the claim.
    We remark that these figures depict a representative edge configuration; for another scenario satisfying condition (\romannumeral3) (for example, when $p'$ is adjacent to $v_i$ rather than $v_{i-2}$), the construction is analogous.
    Thus, the maximality of $m$ implies that no bad pair $(i,j)$ exists.

    We now derive a contradiction by identifying such a bad pair.
    Note that $T$ is an outer-planar graph where every inner face is a 4-cycle. Since $|T| \ge 6$, $T$ must contain a chord of $\partial T$, implying that there are at least two faces that share three edges with $\partial T$.
    It follows that there are two distinct indices $i$ and $i'$ such that $v_{i-1}v_{i+2}, v_{i'-1}v_{i'+2}\in E(T)$.
    We will show that one of the pairs $(i,i+1), (i+1,i), (i',i'+1)$, or $(i'+1,i')$ is bad.
    Indeed, conditions (\romannumeral1) and (\romannumeral2) hold for all four pairs.
    Suppose neither $(i,i+1)$ nor $(i+1,i)$ is bad. By symmetry, we may assume that $p$ and $v_i$ belong to different partite sets. Then $|E(\{v_{i-2},v_i\},\{p,p'\})| \le 2$ and $|E(\{v_{i+2},v_i\},\{p,p'\})|\leq 2$.
    Recall that each of $p'$ and $p$ is non-adjacent to at most one vertex in $\{v_{i-2},v_i,v_{i+2}\}$. It follows that $v_i$ must be the common non-neighbor for both $p$ and $p'$.
    This implies that the vertex in $\{v_{i'}, v_{i'+1}\}$ belonging to the partite set distinct from that of $p$ cannot serve as a common non-neighbor for both $p$ and $p'$.
    Consequently, either $(i',i'+1)$ or $(i'+1,i')$ must be bad.
    Thus, a bad pair always exists.
    This contradiction proves that if $\deg_T(w)=0$, then $\deg_R(w)=1$, completing the proof of (4).
\end{proof}

\section{The $k$-weak graphs and proof of Theorem~\ref{main 1}}
\noindent In this section, we introduce a class of graphs called \emph{$k$-weak graphs}, which structurally approximate 3-connected graphs with minimum degree at least $k$, and state the main technical result on cycle lengths in this class (Theorem~\ref{thm:weak graph}).
The section is then divided into two parts: in Section~\ref{subsec:reduction}, we prove Theorem~\ref{main 1} via a reduction argument assuming Theorem~\ref{thm:weak graph}; in Section~\ref{subsec:structural lemmas}, we establish key structural lemmas for $k$-weak graphs, preparing for the proof of Theorem~\ref{thm:weak graph} in Sections~5 and~6.

We first provide the formal definition. 

\begin{dfn}\label{weak}
Let $k\geq 3$ be an integer. A graph $G$ is {\bf $k$-weak} if one of the following holds:
\begin{description}
\item[~~\textup{Type I.}] $G$ is $3$-connected with $\delta_{2}(G)\ge k$.
\item[~~\textup{Type II.}] There exists exactly one vertex $\theta\in V(G)$ with $\deg_G(\theta)=2$, whose neighbors are $\theta_1$ and $\theta_2$, such that the graph
\( G-\theta+\theta_1\theta_2 \)
is $3$-connected with minimum degree at least $k$.
\end{description}
In either case, we denote by $\theta$ a vertex of minimum degree in $G$. Note that $\theta$ is the unique vertex with $\deg_G(\theta) < k$ whenever $\delta(G) < k$.
\end{dfn}

It follows from the definition that a $k$-weak graph is always 2-connected with $\delta_2(G) \ge k$. While Type I covers the 3-connected case, Type II represents the minimal structural deviation from 3-connectivity: the graph is 2-connected, but the separating set isolates exactly one vertex.

We now state Theorem~\ref{thm:weak graph}, which constitutes the main technical result of this paper.

\begin{thm}\label{thm:weak graph}
Let $k\geq 6$ be an integer. If $G$ is a $k$-weak graph not isomorphic to any graph in $\{K_{k+1}, K_{k,k}\}\cup \mathcal{H}_{k}$, then $G$ contains a cycle of length $r\pmod k$ for every even integer $r$.
\end{thm}

\subsection{Proof reduction of Theorem~\ref{main 1}}\label{subsec:reduction}

\noindent In this subsection, we prove Theorem~\ref{main 1}, under the validity of Theorem~\ref{thm:weak graph}.

\begin{proof}[\bf{Proof of Theorem~\ref{main 1} (assuming Theorem~\ref{thm:weak graph}).}]
It suffices to establish the following stability result for 2-connected graphs: for $k\geq 6$, every 2-connected graph $G$ with $\delta_2(G) \ge k$ contains a cycle of length $r \pmod k$ for every even integer $r$, unless $G$ is isomorphic to a graph in $\{K_{k+1}, K_{k,k}\}\cup \mathcal{H}_{k}$.
Observe that every graph in $\{K_{k+1}, K_{k,k}\}\cup \mathcal{H}_{k}$ contains cycles of all even lengths modulo $k$ with the only exception of $2 \pmod k$. 
Thus, Theorem~\ref{main 1} follows directly by applying this result to each end-block of the graph (since every end-block $B$ of a graph with $\delta(G) \ge k$ satisfies $\delta_2(B) \ge k$).

We now proceed to prove this statement.
Suppose for the sake of contradiction that $G$ is a 2-connected graph with $\delta_2(G)\geq k$ that is not isomorphic to any graph in $\{K_{k+1},K_{k,k}\}\cup \mathcal{H}_k$, yet there is no cycle of length $r \pmod k$ in $G$ for some even integer $r$. 
Then $G$ cannot be 3-connected; otherwise, $G$ would be a $k$-weak graph of Type I, and Theorem~\ref{thm:weak graph} would imply the existence of such a cycle, a contradiction.

We first claim that there exists a 2-cut $S$ of $G$ such that at least two components of $G - S$ have order at least 2.
Suppose to the contrary that for every 2-cut $S = \{x, y\}$ of $G$, there is a component of $G-S$ consisting of a single vertex $z$, i.e., $N_G(z) = \{x, y\}$.
Then $(G - z - xy, x, y)$ is a 2-connected rooted graph with $\delta_2(G - z - xy, x, y)\geq \delta_2(G)\geq k$. 
By Lemma~\ref{lem:admis path}, there are $k - 1$ admissible $(x, y)$-paths in $G - z - xy$. 
If $xy \in E(G)$, combining these paths with the edge $xy$ or the path $xzy$ yields at least $(k - 1) + 2 - 1 = k$ consecutive cycles in $G$, a contradiction. 
Hence, we must have $xy \notin E(G)$. 
Since $\delta_2(G) \ge k$, $z$ is the unique vertex with degree less than $k$ in $G$.
Hence, $S$ is the unique $2$-cut of $G$, and $G-z+xy$ is 3-connected.
This implies that $G$ is a $k$-weak graph of Type II, so Theorem~\ref{thm:weak graph} guarantees the existence of the desired cycle, a contradiction.
This proves the claim.

Let $S = \{x, y\}$ be a 2-cut of $G$ satisfying the claim, and let $M$ and $N$ be the vertex sets of two components of $G - S$ with $|M|, |N| \ge 2$. 
Define $G_M := G[M \cup S]$ and $G_N := G[N \cup S]$.
Then $(G_M, x, y)$ and $(G_N, x, y)$ are 2-connected rooted graph on at least four vertices with $\delta_2(G_M, x, y) \ge k$ and $\delta_2(G_N, x, y) \ge k$. 
By Lemma~\ref{lem:admis path}, there are $k - 1$ admissible $(x, y)$-paths in both $G_M$ and $G_N$. 
Concatenating these paths produces at least $(k - 1) + (k - 1) - 1 = 2k - 3$ admissible cycles in $G$. 
Since none of these cycles has length $r \pmod k$, we deduce that $k$ must be even, all these cycles must have odd lengths, and the lengths of admissible paths in $G_M$ or $G_N$ form a $2$-AP. 

If one of $G_M$ and $G_N$ is non-bipartite, say $G_M$, then there exist two $(x, y)$-paths $L_1, L_2$ in $G_M$ whose lengths have different parities. 
Consequently, combining one of $\{L_1, L_2\}$ with the $k - 1$ admissible $(x, y)$-paths in $G_N$ would produce at least $k - 1 \ge k/2$ even cycles, whose lengths form a $2$-$\AP$. 
This collection must contain a cycle of length $r \pmod k$ (since $k$ is even, any set of $k/2$ even lengths forming a $2$-AP covers all even residues modulo $k$), a contradiction.

Therefore, we assume that both $G_M$ and $G_N$ are bipartite. 
We may assume without loss of generality that $\theta\notin M$, so that $\deg_{G_M}(v) \ge k$ for every vertex $v \in M$.
We claim that $G_M$ contains a block $B$ of order at least 4 such that $|V(B)\cap\big(\mathrm{Cut}(B) \cup S\big)|\leq 2$.
If such a block $B$ exists, then $B$ is a 2-connected bipartite graph where $\deg_B(v) \ge k$ holds for all but at most two vertices $v \in V(B)$.
According to Theorem~\ref{Thm:32}, $B$ contains $k - 1 \ge k/2$ even cycles with lengths forming a 2-AP, which guarantees a cycle of length $r \pmod k$, a contradiction.

It remains to verify the existence of such a block $B$. 
If $G_M$ is 2-connected, then $B := G_M$ suffices. 
Suppose $G_M$ is not 2-connected. 
If $G_M$ has an end-block $B$ disjoint from $S$, then any non-cut-vertex $v \in V(B)$ satisfies $\deg_B(v) = \deg_G(v) \ge k$, implying $|B| \ge k+1 > 4$, which suffices.
Thus, we may assume that $G_M$ has exactly two end-blocks, say $B_x$ and $B_y$, containing $x$ and $y$ respectively. 
If $|B_x| \ge 3$, let $B := B_x$. 
Then, similar to the previous case, any non-cut-vertex $v \in V(B) \setminus \{x\}$ satisfies $\deg_B(v) \ge k$, implying $|B| \ge k+1 > 4$.
Finally, assume $|B_x| = |B_y| = 2$, with $B_x = \{x, z\}$. 
Let $B$ be the unique block of $G_M$ distinct from $B_x$ that contains $z$.
Since $|M| \ge 2$, the graph $G_M$ is not merely the path $xzy$, which implies $y \notin V(B)$.
Moreover, since $z$ is adjacent only to $x$ outside of $B$, we have $|B|\geq \deg_B(z)+1= \deg_G(z)\geq k>4$. 
Thus, $B$ satisfies the required conditions.
This completes the proof.
\end{proof}

\subsection{Notation and lemmas for $k$-weak graphs}\label{subsec:structural lemmas}
\noindent We now establish the necessary definitions and structural lemmas for $k$-weak graphs.
Henceforth, we use $G^\star$ to represent a $k$-weak graph. 

For technical reasons, we define a specific subgraph $G \subseteq G^\star$ as follows:

\begin{dfn}\label{dfn:Ghat}
Let $G^\star$ be a $k$-weak graph and $\theta$ be the vertex defined in Definition~\ref{weak}. We define the subgraph $G\subseteq G^\star$ as follows:
 \begin{equation*}\label{equ:gcd=1}
G :=
\begin{cases}
G^\star, & \text{if } G^\star \text{ is of Type I,} \\
G^\star-\theta, & \text{if } G^\star \text{ is of Type II.}
\end{cases}
\end{equation*}   
\end{dfn}

Our general strategy is to find two specific path families in a core subgraph $H\subseteq G$ and in $G - H$, respectively, and then concatenate these paths yields the required cycles.
Let us summarize the degree and connectivity constraints for $G$ that will be used later.

\begin{prop}
The graph $G$ satisfies $\delta(G)\geq 3$, $\delta_2(G)\geq k-1$, and $\delta_3(G)\geq k$. In particular:
\begin{itemize}  
    \item If $G^\star$ is of Type I, then $G$ is 3-connected with $\delta(G)\ge 3$ and $\delta_2(G)\geq k$.
    \item If $G^\star$ is of Type II, then $G$ is 2-connected with $\delta(G)\geq k-1$ and $\delta_3(G)\geq k$. 
    Moreover, if $\theta_{1}\theta_{2}\in E(G)$ (i.e., $\theta_{1}\theta_{2}\in E(G^\star)$), then $\delta(G)\ge k$.
\end{itemize}
\end{prop}

In the subsequent proofs, a recurring task is to find admissible paths within specific components. 
To apply Lemma~\ref{lem:admis path} effectively, we identify pairs of vertices within a 2-connected rooted subgraph satisfying the requisite degree conditions. 
This motivates the following definition.

\begin{dfn}\label{dfn:valid}
    Let $x, y$ be distinct vertices in a connected graph $M$, and let $t \ge 2$ be an integer. The ordered pair $(x, y)$ is \textbf{$t$-valid} if there exists an end-block $B$ of $M$ with $|B| \ge t$ such that one of the following holds:
    \begin{itemize}
        \item[(1)] $x, y \in V(B)$, and $\deg_B(v) \ge t$ for all but at most one vertex $v \in V(B) \setminus \{x, y\}$.
        \item[(2)] $x \in V(B) \setminus \{b\}$ and $y \notin V(B)$, where $\{b\} = \mathrm{Cut}(B)$, and $\deg_B(v) \ge t$ for all but at most one vertex $v \in V(B) \setminus \{x, b\}$.
    \end{itemize}
\end{dfn}

The following observation follows immediately from Lemma~\ref{lem:admis path} and Definition~\ref{dfn:valid}.

\begin{obs}\label{obs:valid}
	Let $x$ and $y$ be distinct vertices in a connected graph $M$, and let $t\geq 2$ be an integer.
	If $(x, y)$ is $t$-valid, then $\mathcal{P}^M_{x,y}$ contains $t-1$ admissible paths. 
\end{obs}

\begin{proof}
    Suppose $(x, y)$ is $t$-valid. 
    Let $B$ be the end-block containing $x$ as specified in Definition~\ref{dfn:valid}.
    If $|B| \leq 3$, then $B \simeq K_{|B|}$ (as $B$ is a block), and $t \leq |B| \leq 3$. The conclusion holds trivially.
    Now assume $|B| \geq 4$. Apply Lemma~\ref{lem:admis path} to the 2-connected rooted graph $(B, x, y)$ if $y \in V(B)$, or to $(B, x, b)$ if $y \notin V(B)$ (where $\{b\}=\mathrm{Cut}(B)$). The conclusion holds in either case.
\end{proof}

We introduce $B_M$ and $u_M$ to identify a specific configuration in $M$ suitable for establishing valid pairs.

\begin{dfn}[$B_M, u_M$]\label{dfn: block,vertex}
	Let $G^\star$ be a $k$-weak graph, and let $T$ be a subgraph of $G$. For any component $M$ of $G-T$ of order at least $3$, we select an end-block $B_{M}$ of $M$ as follows:
	
	\begin{itemize}
		\item[(1)] If $M$ is 2-connected, let $B_{M} := M$.
		
		\item[(2)] If $M$ is not 2-connected and $G^\star$ is of Type I, let $B_{M}$ be an arbitrary end-block such that $\theta \notin V(B_{M}) \setminus \mathrm{Cut}(B_{M})$.
        
        \item[(3)] If $M$ is not 2-connected and $G^\star$ is of Type II, let $B_{M}$ be an end-block of maximum order such that $V(B_{M}) \setminus \mathrm{Cut}(B_{M})$ contains at most one of $\theta_{1}$ and $\theta_{2}$.
	\end{itemize}
	
	Based on this selection, we choose a vertex $u_{M} \in V(B_M) \setminus \mathrm{Cut}(B_M)$ satisfying:
	\begin{itemize}
		\item[(1)] $\deg_{T}(u_{M}) = \max \{ \deg_{T}(v) : v \in V(B_{M}) \setminus \mathrm{Cut}(B_{M}) \}$.
		
		\item[(2)] Subject to (1), $\deg_{G}(u_{M})$ is minimized.
	\end{itemize}
\end{dfn}
We omit the subscript $T$ from $B_M$ and $u_M$, as $M$ being a component of $G - T$ implicitly fixes $T$.
The following lemma establishes that $u_{M}$ forms a valid pair with any other vertex in $M$.
We remark that the proof below does not rely on the maximality of the order of $B_M$ in cases (3); however, this property will be useful in Section 6.

\begin{lem}\label{lem: B}
    Let $k\geq 6$ be an integer and $G^\star$ be a $k$-weak graph.
    For every subgraph $T\subseteq G$ and every component $M$ of $G-T$ with $|M|\geq 3$, if $\deg_{T}(u_{M})\le k-2$, then for every vertex $w\in V(M)\setminus\{u_{M}\}$, the pair $(u_{M},w)$ is $(k-\deg_{T}(u_{M}))$-valid. 
\end{lem}

\begin{proof}
    Let $t := k - \deg_T(u_M)$. By definition, it suffices to show that $|B_M| \ge t$ and that $\deg_{B_M}(v) \ge t$ holds for all but at most one vertex $v \in V(B_M) \setminus \{u_M\}$.

    We first consider the case that $M$ is $2$-connected.   
    Then $B_M = M$ and $|B_M| \ge 3$. 
    Since $\delta_2(G) \ge k-1$, there exists a vertex $u \in V(M)$ with $\deg_{G}(u) \ge k-1$.
    By the maximality of $\deg_T(u_M)$, we have $\deg_T(u) \le \deg_T(u_M)$, and thus
    \[
    |M| \ge \deg_M(u) + 1 = \deg_{G}(u) - \deg_T(u) + 1 \ge (k-1) - \deg_T(u_M) + 1 = t.
    \]
    
    It remains to verify the degree condition.
    If $G^\star$ is of Type I, $\theta$ is the only vertex that may have degree less than $k$ in $G$.
    The maximality of $\deg_T(u_M)$ implies that for every $v \in V(M) \setminus \{u_M, \theta\}$, we have $\deg_M(v) = \deg_{G}(v) - \deg_T(v) \ge k - \deg_T(u_M) = t$. Thus, at most one vertex (namely, $\theta$) may fail the degree condition.

    If $G^\star$ is of Type II, all vertices other than $\theta_1$ and $\theta_2$ have degree at least $k$ in $G$.
    For every $v \in V(M) \setminus \{u_M, \theta_1, \theta_2\}$, we have 
    \begin{align}\label{eq:typeII}
        \deg_{M}(v) = \deg_G(v) - \deg_T(v) \ge k - \deg_T(u_M) = t.
    \end{align}
    We claim that if both $\theta_1$ and $\theta_2$ fail the degree condition (i.e., $\deg_M(\theta_i) < t$ for $i=1,2$), then $u_M \in \{\theta_1, \theta_2\}$.
    Indeed, by revisiting inequality \eqref{eq:typeII}, the condition $\deg_M(\theta_i) < t$ implies that $\deg_{G}(\theta_i) = k-1$.
    Suppose to the contrary that $u_M \notin \{\theta_1, \theta_2\}$. 
    By the selection criterion (2) for $u_M$, we must have $\deg_T(\theta_1) < \deg_T(u_M)$ (as otherwise $u_M$ would not minimize $\deg_G$).
    However, this yields $\deg_M(\theta_1) = \deg_{G}(\theta_1) - \deg_T(\theta_1) \ge (k-1) - (\deg_T(u_M) - 1) = t$, a contradiction.
    Thus, at most one vertex in $V(M) \setminus \{u_M\}$ has degree less than $t$, completing the verification.

    We then assume that $M$ is not 2-connected.
    Then $B_M$ is an end-block with $|B_M| \ge 2$. 
    For any $u \in V(B_M) \setminus \mathrm{Cut}(B_M)$, we have $\deg_{G}(u) \ge k-1$. (This is ensured by the criterion (2) for Type I since $\theta \notin V(B_M) \setminus \mathrm{Cut}(B_M)$, and by $\delta(G) \ge k-1$ for Type II).
    By the maximality of $\deg_T(u_M)$, it follows that 
    \[
    |B_M| \ge \deg_{B_M}(u) + 1 = \deg_{G}(u) - \deg_T(u) + 1 \ge (k-1) - \deg_T(u_M) + 1 = t.
    \]

    Now we verify the degree condition.
    If $|B_M| = 2$, the condition holds trivially.
    Assume $|B_M| \ge 3$, then $V(B_M) \setminus (\mathrm{Cut}(B_M) \cup \{u_M\})$ is non-empty.
    If $G^\star$ is of Type I, for every $v \in V(B_M) \setminus (\mathrm{Cut}(B_M) \cup \{u_M\})$, we have $v \neq \theta$, so $\deg_{B_{M}}(v) = \deg_{G}(v) - \deg_{T}(v) \ge k - \deg_T(u_M) = t$.
    If $G^\star$ is of Type II, the verification is analogous to the 2-connected case.
    Recall that $V(B_{M}) \setminus \mathrm{Cut}(B_{M})$ contains at most one of $\theta_{1}$ and $\theta_{2}$; we may assume without loss of generality that $\theta_2 \notin V(B_{M}) \setminus \mathrm{Cut}(B_{M})$.
    Consequently, for every $v \in V(B_M) \setminus (\mathrm{Cut}(B_M) \cup \{u_M, \theta_1\})$, we have $\deg_{B_{M}}(v) \ge t$.
    By repeating the argument based on the selection criterion (2), we deduce that if $\theta_1 \in V(B_M) \setminus \mathrm{Cut}(B_M)$ and $\deg_{B_{M}}(\theta_1) < t$, then necessarily $u_M = \theta_1$.
    Thus, only the vertex in $\mathrm{Cut}(B_M)$ may violate the degree condition.
    This completes the proof.
\end{proof}

Finally, we establish a strengthened form of Menger's Theorem for 2-connected graphs.
\begin{lem}\label{lem:menger extend}
    Let $G$ be a 2-connected graph, and let $X, Y$ be a partition of $V(G)$ such that $|X| \ge 2$ and $|Y| \ge 2$. Then for every $x \in N_X(Y)$, there exist two disjoint $(X, Y)$-edges, one of which is incident to $x$.
\end{lem}

\begin{proof}
    Let $y \in N_Y(x)$. If $E(X-x, Y-y) \neq \emptyset$, then any edge in this set, together with $xy$, forms the desired pair of disjoint edges.
    Suppose otherwise. Since $G$ is 2-connected, $E(X-x, Y) \neq \emptyset$. By our assumption, every edge in this set must be incident to $y$, implying $E(X-x, y) \neq \emptyset$. Similarly, $E(x, Y-y) \neq \emptyset$. Selecting an arbitrary edge from each of these two sets yields a pair of disjoint edges, which completes the proof.
\end{proof}

\section{Proof of Theorem~\ref{thm:weak graph}: the non-bipartite case}
\noindent Throughout the rest of the paper, let $G^\star$ be a given $k$-weak graph and let $G\subseteq G^\star$ be defined in Definition~\ref{dfn:Ghat}.
This section is devoted to the proof of Theorem~\ref{thm:weak graph} when $G$ is non-bipartite.

\begin{thm}\label{main}
    Let $k\ge 6$ be an integer, and let $G^\star$ be a $k$-weak graph not isomorphic to $K_{k+1}$. If $G$ is non-bipartite, then $G^\star$ contains a cycle of length $r\pmod k$ for every even integer $r$.
\end{thm}

Throughout this section, we assume the conditions of Theorem~\ref{main}: 
\[
\mbox{ $k\geq 6$,~ $G^\star$ is a $k$-weak graph, ~$G^\star \not\simeq K_{k+1}$, ~ and $G$ is non-bipartite.}
\]
The first lemma treats the $K_3$-free case, with a minimal induced odd cycle as the core graph.

\begin{lem}\label{K3}
    If $G$ is $K_{3}$-free, then $G$ contains a cycle of length $r\pmod k$ for any even $r$.
\end{lem}

\begin{proof}
    Suppose for a contradiction that for some even $r$, $G^\star$ does not contain any cycle of length $r\pmod k$. Let $C=v_{0}v_{1}\ldots v_{2s}v_{0}\subseteq G$ ($s\ge 2$) be an induced odd cycle of minimum order such that $\sum_{i=0}^{2s}\deg_{G}(v_{i})$ is minimized. 
    By the minimality of $|C|$, every vertex $v\in V(G-C)$ satisfies $\deg_{C}(v)\le 2$, with equality holding if and only if $N_{C}(v)=\{v_{i},v_{i+2}\}$ for some index $i$ (taken modulo $2s+1$). 
    Moreover, if $G^\star$ is of Type I, the vertex $\theta\notin V(C)$ and $\deg_{G}(\theta)<k$, then the minimality of the degree sum implies $\deg_{C}(\theta)\le 1$.
	Consequently, for any component $M$ of $G-C$, every $v\in V(M)$ satisfies $\deg_M(v) \ge \deg_G(v) - 2 \ge k-3$, except possibly when $G^\star$ is of Type I and $v=\theta$, in which case $\deg_M(\theta) \ge \deg_G(\theta) - 1 \ge 2$.
    In summary, it always holds that $\delta(M)\geq 2$ and $\delta_{2}(M)\ge k-3$, which implies that every end-block of $M$ has order at least $k-2\geq 4$.
	\medskip
	
	\noindent{\bf Claim 1.} $G-C$ is connected.
	
	\begin{proof}
    Suppose to the contrary that $G-C$ has components $M_{1}, \ldots, M_{t}$ with $t \ge 2$. For each $s\in [t]$, let $u_s := u_{M_s}$ and $B_s := B_{M_s}$ as defined in Definition~\ref{dfn: block,vertex}.

    First, consider the case where $\deg_{C}(u_i) = 1$ for some $i\in [t]$, or $k \ge 7$ (in which case let $i \in [t]$ be arbitrary).
    Pick any index $j\in[t]$ distinct from $i$.
    Since $G$ is 2-connected, Lemma~\ref{lem:menger extend} guarantees the existence of vertices $w_i \in V(M_i) \setminus \{u_i\}$ and $w_j \in V(M_j) \setminus \{u_j\}$, along with two disjoint paths connecting $\{u_i, w_i\}$ and $\{u_j, w_j\}$ whose internal vertices lie in $C$.
    By Observation~\ref{obs:valid} and Lemma~\ref{lem: B}, $\mathcal{P}_{u_i,w_i}^{M_{i}}$ contains $k-\deg_{C}(u_i)-1$ admissible paths, and $\mathcal{P}_{u_j,w_j}^{M_{j}}$ contains $k-\deg_{C}(u_j)-1$ admissible paths.
    Concatenating these path collections, together with the two disjoint connecting paths, yields at least $(k-\deg_{C}(u_{i})-1)+(k-\deg_{C}(u_{j})-1)-1 = 2k - 3 - \deg_C(u_i) - \deg_C(u_j) \ge k$ admissible cycles in $G$ (using the fact that $k \ge 7$ or $\deg_{C}(u_i) = 1$).
    Since $G$ contains no cycle of length $r \pmod k$, we deduce that $k$ must be even, and the lengths of admissible paths in $\mathcal{P}_{u_{i},w_{i}}^{M_{i}}$ must form a 2-AP.
    As $C$ is an odd cycle, $\mathcal{P}_{u_{i},w_i}^{C}$ contains a path $L$ whose length has the same parity as the admissible paths in $\mathcal{P}_{u_{i},w_{i}}^{M_{i}}$.
    Combining $\mathcal{P}_{u_{i},w_i}^{M_{i}}$ with $L$ produces at least $k-\deg_C(u_i)-1 \geq k-3 \ge k/2$ cycles of even lengths forming a 2-AP, one of which must have length $r \pmod k$, a contradiction.

    Thus, we may assume $k=6$ and $\deg_{C}(u_s) = 2$ for each $s \in [t]$. 
    We claim that for any $s \in [t]$, $N_{C}(M_s)$ is contained in a set of three consecutive vertices on $C$.
    Suppose the claim fails for $M_1$. 
    We may assume $N_{C}(u_1) = \{v_{0}, v_{2}\}$. 
    Since the claim fails, there exists $w \in V(M_1) - u_1$ adjacent to some $v_\ell \in V(C)$ with $\ell \notin \{0,1,2\}$.
    Then $\mathcal{P}_{u_{1},w}^{C}$ contains paths of lengths $\{\ell, \ell+2, 2s-\ell+3, 2s-\ell+5\}$ (two odd and two even integers, with differences of 2).
    It is straightforward to verify that combining these paths with $\mathcal{P}_{u_{1},w}^{M}$ (which contains $k-3=3$ admissible paths) generates cycles of all lengths modulo 6, a contradiction.

    We now assert that $N_{C}(u_i) \cap N_{C}(M_i \setminus \{u_i\}) \neq \emptyset$ for every $i \in [t]$.
    Suppose for the sake of contradiction that this intersection is empty for some index, say $i=1$. 
    We may assume $N_C(M_1) \subseteq \{v_0, v_1, v_2\}$ and $N_C(u_1) = \{v_0, v_2\}$, then $N_C(M_1 \setminus \{u_1\}) \subseteq \{v_1\}$.
    Consequently, $\{u_1, v_1\}$ is a 2-cut of $G$ separating $M_1 \setminus \{u_1\}$ (which is non-empty since $|M_1| \ge 3$) from the rest of the graph.
    This is impossible if $G^\star$ is of Type I (as $G$ is 3-connected). 
    Thus, $G^\star$ must be of Type II, and $M_1$ must contain exactly one of the vertices $\theta_1, \theta_2$ (as $G+\theta_1\theta_2$ is 3-connected).
    For every $v\in V(M_1) \setminus \{u_1\}$, since $N_C(v) \subseteq \{v_1\}$, we have $\deg_C(v) \le 1$.
    Consequently, $\deg_{M_1}(v) =\deg_G(v)-\deg_C(v)\ge k-1$ for all $v \in V(M_1)\setminus\{u_1\}$, with the possible exception for the single vertex in $V(M_1) \cap \{\theta_1, \theta_2\}$. 
    We then select $q$ as follows: if $M_1$ is 2-connected, let $q$ be any neighbor of $v_1$ in $M_1 \setminus \{u_1\}$.
    If $M_1$ is not 2-connected, let $q$ be any neighbor of $v_1$ in $M_1 - B_1$ (such a vertex exists because $\mathrm{Cut}(B_1)$ is not a cut-vertex in $G$).
    In either case, a routine check confirms that $(u_1, q)$ is $(k-1)$-valid, as the required degree condition holds for all vertices in $M_1 \setminus \{u_1\}$ except possibly for the single vertex in $V(M_1) \cap \{\theta_1, \theta_2\}$.
    By Observation~\ref{obs:valid}, $\mathcal{P}^{M_1}_{u_1,q}$ contains $k-2$ admissible paths. Recall that the pair $(u_2, w_2)$ in $M_2$ yields $k-3$ admissible paths in $\mathcal{P}_{u_2,w_2}^{M_2}$. Since $G$ is 2-connected, Lemma~\ref{lem:menger extend} provides two disjoint paths between $\{u_1, q\}$ and $\{u_2, w_2\}$ with all internal vertices in $C$. Concatenating the paths from $\mathcal{P}^{M_1}_{u_1,q}$ and $\mathcal{P}_{u_2,w_2}^{M_{2}}$ via these connecting paths produces $(k-2)+(k-3)-1=k$ admissible cycles (given $k=6$). Following the previous parity argument, if these cycles fail to cover some residue $r \pmod k$, the lengths of admissible paths in $\mathcal{P}^{M_1}_{u_1,q}$ must form a 2-AP. 
    Combining these with a path in $\mathcal{P}^{C}_{u_1,q}$ of the appropriate parity yields $k-2 > k/2$ even admissible cycles, a contradiction.
    This proves the assertion.
    
    According to the assertion, for each $i\in [t]$, there exists $p_i \in V(M_i) \setminus \{u_i\}$ such that $N_C(p_i) \cap N_C(u_i) \neq \emptyset$.
    Then $\Pset_{u_{i},p_i}^{C}\supseteq\{2, 4, 2s+1\}$.
    The proof of Claim 1 is then partitioned into the cases $s \in \{2, 3\}$ and $s \ge 4$.
    If $s \in \{2, 3\}$, the set $\mathcal{P}_{u_1,p_1}^{C}$ contains two paths whose lengths differ by 3. Concatenating these with the $k-3=3$ admissible paths in $\mathcal{P}_{u_1,p_1}^{M_1}$ yields cycles of all possible lengths modulo 6, a contradiction.
    If $s \ge 4$, then $|C| = 2s+1 \ge 9$.
    Recall that each $N_C(M_i)$ is some set of three consecutive vertices on $C$. Since $\delta(G) \ge 3$, every vertex in $C$ has a neighbor outside $C$, so these sets $N_C(M_i)$'s cover $V(C)$.
    A routine calculation shows that there must be three sets that are pairwise disjoint or intersect in at most one vertex.
    In other words, there exist distinct $i, j, \ell \in [t]$ and indices $a, b, c$ such that $N_C(u_i) = \{v_{a-1}, v_a, v_{a+1}\}$, $N_C(u_j) = \{v_{b-1}, v_b, v_{b+1}\}$, and $N_C(u_\ell) = \{v_{c-1}, v_c, v_{c+1}\}$, with the distance between any pair of $\{v_a, v_b, v_c\}$ being at least 2.
    Since each of $\mathcal{P}_{u_{i},p_i}^{M_{i}}$, $\mathcal{P}_{u_{j},p_j}^{M_{j}}$, and $\mathcal{P}_{u_{\ell},p_\ell}^{M_{\ell}}$ contains $k-3=3$ admissible paths, combining them via three disjoint subpaths of $C$ connecting these endpoints, yields at least $3+3+3-2 > 6$ admissible cycles.
    Again, if these fail to cover some residue $r \pmod k$, we derive a contradiction by combining $\mathcal{P}^{M_1}_{u_1,p_1}$ with a path in $\mathcal{P}^{C}_{u_1,p_1}$ with a suitable length parity.
    This completes the proof of Claim 1.
	\end{proof}

    By Claim 1, $M := G-C$ is connected. 
    Let $u := u_M$.
    Without loss of generality, assume $N_C(u)$ is either $\{v_{0}\}$ or $\{v_{0}, v_2\}$.
    Since $\delta(G) \ge 3$, every vertex on $C$ has a neighbor in $M$. In particular, some $w \in V(M)$ is adjacent to $v_{s+2}$. Note $u \neq w$ as $s \ge 2$.
    If $N_C(u)=\{v_0\}$, then by Observation~\ref{obs:valid} and Lemma~\ref{lem: B}, $\mathcal{P}^M_{u,w}$ contains $k-2$ admissible paths.
    Observe that $\Pset^C_{u,w}\supseteq \{s+1,s+4\}$.
    Since $k \ge 6$, combining paths in $\mathcal{P}^C_{u,w}$ and $\mathcal{P}^M_{u,w}$ yields $k$ consecutive cycles, a contradiction.
    If $N_C(u)=\{v_0, v_2\}$, then $\mathcal{P}^M_{u,w}$ contains $k-3$ admissible paths, and one can verify that $\Pset^C_{u,w}\supseteq \{s+1, s+2, s+3, s+4\}$.
    Combining the paths in $\mathcal{P}^M_{u,w}$ and $\mathcal{P}^C_{u,w}$ yields $(k-3)+4-1=k$ consecutive cycles, a contradiction.
    This completes the proof.
\end{proof}

In the remainder of this section, following Lemma~\ref{K3}, we may assume $K_3\subseteq G$ and 
\[
\mbox{let $T \subseteq G$ be a trigonal subgraph of maximum order.}
\]
In what follows, the core graph always refers to this maximum trigonal subgraph $T$, and
the proofs are divided into Lemmas~\ref{T4}, \ref{T5}, and \ref{T6}, according to the order of $T$.

Denote by $K_{n}^-$ the graph obtained from $K_n$ by removing an edge. 

\begin{lem}\label{T4}
If $|T|=3$, then $G^\star$ contains $k$ consecutive cycles.
\end{lem}

\begin{proof}
    Define a subgraph $H \subseteq G^\star$ as follows: if $G^\star$ is of Type II and $\theta_{1}\theta_{2} \notin E(G^\star)$, let $H := G^\star$; otherwise, let $H := G$. By assumption, $H$ contains a triangle, say with vertex set $\{a, b, c\}$, but no $K_{4}^{-}$ subgraph. It follows that every $v \in V(H) - \{a, b, c\}$ has at most one neighbor in $\{a, b, c\}$.
    
    Let $H'$ be the graph obtained from $H$ by contracting the edge $bc$ into a new vertex $a'$. We claim that $H'$ is $2$-connected. If $G^\star$ is of Type II and $\theta_{1}\theta_{2} \notin E(G^\star)$, then $\{\theta_1, \theta_2\}$ is the unique $2$-cut of $H=G^\star$ (since $G+\theta_1\theta_2$ is $3$-connected), which implies that $\{b, c\}$ is not a $2$-cut in $H$ (as $bc\in E(G^\star))$; thus, $a'$ is not a cut-vertex in $H'$. Clearly, no other vertex can be a cut-vertex in $H'$. Otherwise, we have $G^\star$ is of Type II and $\theta_{1}\theta_{2} \in E(G^\star)$, or $G^\star$ is of Type I. 
    It follows from the definition that $H=G$ is $3$-connected, which immediately ensures that $H'$ is $2$-connected.
    
    In either case, $H'$ is $2$-connected, implying that the rooted graph $(H' - aa', a, a')$ is $2$-connected.
    Since $H$ contains no $K_4^-$, every $v\in V(H)\setminus\{a,b,c\}$ satisfies $\deg_{\{a,b,c\}}(v)\leq 1$.
    This implies that $\delta_{2}(H' - aa', a, a') \ge \delta_2(H)\geq k$. 
    Note that $|H'|\geq |H|-1\geq |G|-1\geq \delta_2(G)>4$.
    By Lemma~\ref{lem:admis path}, $\mathcal{P}_{a,a'}^{H' - aa'}$ contains $k-1$ admissible paths. Equivalently, $\mathcal{P}_{a,b}^{H-\{a,b,c\}} \cup \mathcal{P}_{a,c}^{H-\{a,b,c\}}$ contains $k-1$ admissible paths. By combining these with the edges $ab, ac$ or the paths $acb, abc$ (which have length 1 or 2), we obtain $(k-1)+2-1=k$ consecutive cycles in $H$, and thus in $G^\star$. This completes the proof of Lemma~\ref{T4}.
\end{proof}

\begin{lem}\label{T5}
    If $|T|=4$, then $G^\star$ contains $k$ consecutive cycles.
\end{lem}

\begin{proof}
    Let $T$ be a trigonal subgraph with $|T|=4$. Since $\delta(G)\geq 3$ and $\delta_2(G) \ge k-1 \ge 5$, $G-T \neq \emptyset$.
    By the maximality of $T$, every $v \in V(G-T)$ has at most two neighbors in $T$.
    Thus, for any component $M$ of $G-T$, $\delta(M) \ge 1$ and $\delta_2(M) \ge k-3$, implying that every end-block of $M$ is an edge or has order at least $k-2 \ge 4$.
    
    If $G[V(T)] \simeq K_{4}$, then the maximality of $T$ implies that every $v \in V(G-T)$ has $\deg_T(v) \le 1$.
	Let $M$ be a component of $G-T$ and let $u := u_M$. 
	By Lemma~\ref{lem:menger extend}, there exists $w \in V(M) \setminus \{u\}$ such that $u$ and $w$ are adjacent to distinct vertices in $T$. Thus $\mathcal{L}_{u,w}^{T}=\{3, 4, 5\}$.
    By Lemma~\ref{lem: B} and Observation~\ref{obs:valid}, $\mathcal{P}_{u,w}^{M}$ contains at least $k-\deg_T(u)-1\geq k-2$ admissible paths.
    The union of paths in $\mathcal{P}_{u,w}^{M}$ and $\mathcal{P}_{u,w}^{T}$ yields $(k-2)+3-1=k$ consecutive cycles.
	
	If $G[V(T)] \simeq K_{4}^{-}$, define $H$ as in Lemma~\ref{T4}: if $G^\star$ is of Type II and $\theta_{1}\theta_{2} \notin E(G^\star)$, let $H := G^\star$; otherwise, let $H := G$.
    Then $T$ remains a trigonal subgraph of maximum order in $H$.
	Let $V(T)=\{a,b,c,d\}$ with $bd \notin E(H)$.
    The maximality of $T$ implies that every $v \in V(H-T)$ satisfies $|N_H(v)\cap\{b,c\}|\leq 1$.
    Let $H'$ be obtained from $H$ by contracting $bc$ into a vertex $a'$. 
    Similar to Lemma~\ref{T4}, one can verify that $(H'-aa', a, a')$ is a 2-connected rooted graph with $\delta_2(H'-aa', a, a') \ge k$. 
    According to Lemma \ref{lem:admis path}, $\mathcal{P}_{a,a'}^{H'-aa'}$ contains $k-1$ admissible paths, and equivalently, $\mathcal{P}_{a,b}^{H-\{a,b,c\}}\cup \mathcal{P}_{a,c}^{H-\{a,b,c\}}$ contains $k-1$ admissible paths.
    Concatenating these with the subpaths in the triangle $H[\{a,b,c\}]$ produces $k$ consecutive cycles in $H$, and thus in $G^\star$.
	This completes the proof of Lemma~\ref{T5}.
\end{proof}

Finally, we consider the remaining case where $|T| \ge 5$. To proceed, we need the following classical pancyclicity criterion due to Bondy~\cite{bondy1971pancyclic}.

\begin{lem}[\cite{bondy1971pancyclic}]\label{lem:Ore}
    Let $G$ be a graph of order $n$. If $d_G(u) + d_G(v) \ge n$ for every pair of non-adjacent vertices $u, v \in V(G)$, then $G$ contains cycles of all lengths in $[3, n]$, unless $G\simeq K_{n/2,n/2}$.
\end{lem}

\begin{lem}\label{T6}
    If $|T|\ge 5$, then $G^\star$ contains a cycle of length $r\pmod k$ for any even $r$.
\end{lem}

\begin{proof}
    Suppose $G^\star$ is a counterexample. 
    Let $|T|=t\geq 5$. 
    Recall Proposition~\ref{prop: T-graph} implies that every trigonal graph on $k+2$ vertices contains cycles of all lengths in $[3,k+2]$, so we must have $5\le t\le k+1$.
    Let $\partial T=v_{0}v_{1}\dots v_{t-1}v_{0}$.
    By the maximality of $T$, no vertex $v\in V(G-T)$ can be adjacent to two consecutive vertices on $\partial T$; consequently, $\deg_{T}(v)\le \lfloor t/2\rfloor$.

    We first consider the case where $|V(G)| \le k+3$. For any pair of non-adjacent vertices in $G$, the sum of their degrees is at least $\delta(G)+\delta_2(G) \geq \min\{3+k, 2(k-1)\} = k+3 \ge |G|$. 
    Since $K_3\subseteq T\subseteq G$, $G$ cannot be bipartite.
    It then follows from Lemma~\ref{lem:Ore} that $G$ contains cycles of all lengths in $[3, |V(G)|]$.
    Since $G^\star$ is a counterexample and thus misses a cycle of length in $[3,k+2]$, we must have $|G| \le k+1$.
    On the other hand, $|G| \ge \delta_3(G)+1 \ge k+1$.
    This forces $|G| = k+1$, and that all but at most two vertices in $G$ have degree exactly $k$.
    Consequently, $G$ is isomorphic to $K_{k+1}$ or $K_{k+1}^-$.
    Since $G^\star$ is not isomorphic to $K_{k+1}$ or $K_{k+1}^-$, $G^\star$ must be obtained from $G$ by adding the vertex $\theta$ adjacent to exactly two vertices in $G$. It is clear that any such $G^\star$ contains cycles of all lengths in $[3, k+2]$, a contradiction.
    
    Henceforth, we assume $|G|\geq k+4$.
    Thus $|G-T|\geq (k+4)-t\geq 3$.
    Since $\delta_{2}(G-T)\ge (k-1)-\lfloor t/2\rfloor\ge 2$, there exists a component $M$ of $G-T$ with order at least 3.
	Let $u = u_{M}$.
    By Lemma~\ref{lem: B}, for every $v\in V(M)\setminus \{u\}$, the pair $(u,v)$ is $(k-\deg_{T}(u))$-valid.
	
	\medskip
	
	\noindent{\bf Claim 1.} $\deg_{T}(u)\le \max\{1,\lfloor t/2 \rfloor-2\}$.
	
	\begin{proof}
        Suppose to the contrary that $\deg_{T}(u)=\lfloor t/2 \rfloor-r$ where $r\in \{0,1\}$ (with $r=0$ if $t=5$). 
        Since $G$ is 2-connected, there exists a vertex $w\in V(M)\setminus \{u\}$ with $N_{T}(w)\neq \emptyset$.

        We first consider the case where $N_{T}(u)$ and $N_{T}(w)$ are disjoint.
        By Lemma~\ref{lem:short dist}, there exist $v_{i}\in N_{T}(u)$ and $v_{j}\in N_{T}(w)$ such that $\mathrm{dist}_{\partial T}(v_{i},v_{j})\le \max\{1,\lfloor t/2 \rfloor+2-(\lfloor t/2 \rfloor-r)-1\}=r+1$.
        It follows from Proposition~\ref{prop: T-graph} that $\Pset_{u,w}^{T} \supseteq 2+\Pset_{v_i,v_j}^T\supseteq[r+3, t-r+1]$.
        In addition, Observation~\ref{obs:valid} ensures that $\mathcal{P}_{u,w}^{M}$ contains $k-\deg_{T}(u)-1= k-\lfloor t/2 \rfloor+r-1$ admissible paths. 
        If $t=5$ or $t\geq 7$, concatenating the paths in $\mathcal{P}_{u,w}^{T}$ and $\mathcal{P}_{u,w}^{M}$ produces at least $(t-2r-1)+(k-\lfloor t/2 \rfloor+r-1)-1 =k+(t-\lfloor t/2 \rfloor-3)-r\ge k$ consecutive cycles, a contradiction.
        If $t=6$, we observe that $\deg_T(u) = 3-r \ge 2$. 
        By choosing a neighbor in $N_T(u)$ that avoids Case III in Observation 3.3 (possible since $\deg_T(u) \ge 2$), we deduce that $\mathcal{P}^T_{u,w}$ contains at least $t-2\mathrm{dist}_{\partial T}(v_{i},v_{j})+2\geq 6-2(r+1)+2=6-2r$ consecutive paths.
        Combining this with $\mathcal{P}_{u,w}^{M}$ yields at least $(6-2r)+(k+r-4)-1 = k+1-r \ge k$ consecutive cycles, a contradiction.

        It remains to consider the case where $N_{T}(u)$ and $N_{T}(w)$ intersect.
        Let $v_i \in N_{T}(u)\cap N_{T}(w)$.
        Since no two vertices in $N_T(u)$ are consecutive on $\partial T$, it is straightforward to find a vertex $v_{j}\in N_T(u)\setminus\{v_i\}$ such that $\mathrm{dist}_{\partial T}(v_{i},v_{j}) \le r+2$.
		By Proposition~\ref{prop: T-graph}, $\Pset_{u,w}^{T}\supseteq \{2\}\cup (2+\Pset_{v_i,v_j}^{T})\supseteq\{2\} \cup [4+r, t-r]$. 
		Moreover, $\Pset_{u,w}^{M}$ contains an admissible subset of size $k-\lfloor t/2 \rfloor+r-1$; we denote this subset by $\mathcal{L}$.
        Combining these two path collections, it follows that $G$ contains cycles of all lengths in $\mathcal{L}+2$ and $\mathcal{L}+[4+r,t-r]$.

        If the set $\mathcal{L}$ is consecutive, a routine calculation shows that the union of consecutive sets $\mathcal{L} + 2$ and $\mathcal{L}+[4+r, t-r]$ is a larger consecutive set (this is equivalent to $k\geq \lfloor t/2 \rfloor+3$).
        This union is exactly the set sum $\mathcal{L} + [2, t-r]$, which has size at least $(k-\lfloor t/2 \rfloor+r-1) + (t-r-1) - 1 = k+t-\lfloor t/2 \rfloor-3 \ge k.$
        This implies that $G$ contains $k$ consecutive cycles, a contradiction.

        Thus, we may assume that $\mathcal{L}$ is a 2-AP. 
        If $k\geq 7$ or $t$ is odd, it follows from Observation~\ref{obs:addition}~(3) that the sum $\mathcal{L}+[4+r, t-r]$ yields a consecutive set of size at least $2\left(k-\lfloor t/2 \rfloor+r-1\right) + (t-2r-3) - 2 = (2k-7) + \left(t-2\lfloor t/2 \rfloor\right) \ge k$, a contradiction.
        Otherwise, we must have $k=6$ and $t \in [5, k+1]$ is even, which implies $t=k=6$.
        Since $\deg_T(u)=3-r \ge 2$ and Observation~\ref{obs: small T-graph} implies at most one neighbor in $N_T(u)$ leads to the configuration in Case III described there, we can select a neighbor in $N_T(u)$ to avoid Case III.
        Consequently, $\mathcal{P}_{u,w}^{T}$ contains $t-2\mathrm{dist}_{\partial T}(v_{i},v_{j})+2 \ge 6 - 2(r+2) + 2 = 4-2r$ consecutive paths.
        Hence, by Observation~\ref{obs:addition}~(3) again, the sum of $\mathcal{L}$ and the consecutive subset in $\Pset_{u,w}^{T}$ yields a consecutive set of size at least $2\left(k-\lfloor t/2 \rfloor+r-1\right) + (4-2r) - 2 = k$, implying that $G$ contains $k$ consecutive cycles, a contradiction.
		This completes the proof of Claim 1.
	\end{proof}
	
	\medskip
	
	\noindent{\bf Claim 2.} $M$ is the unique component of $G-T$ containing at least $3$ vertices.
	
	\begin{proof}
    Suppose to the contrary that $G-T$ contains another component $M'$ of order at least $3$. 
    Let $u' = u_{M'}$. By Lemma~\ref{lem: B}, for any vertex $v' \in V(M') \setminus \{u'\}$, the pair $(u', v')$ is $(k-\deg_{T}(u'))$-valid. 
    By Claim 1, both $\deg_{T}(u)$ and $\deg_{T}(u')$ are at most $\max\{1, \lfloor t/2 \rfloor - 2\}$.
    Since $G$ is 2-connected, Lemma~\ref{lem:menger extend} guarantees the existence of vertices $w \in V(M) \setminus \{u\}$ and $w' \in V(M') \setminus \{u'\}$, along with two disjoint paths connecting $\{u, w\}$ and $\{u', w'\}$ whose internal vertices lie in $T$.
    By Observation~\ref{obs:valid}, both $\mathcal{P}_{u,w}^{M}$ and $\mathcal{P}_{u',w'}^{M'}$ contain at least $k - \max\{1, \lfloor t/2 \rfloor - 2\} - 1 \geq \lceil (k+1)/2 \rceil$ admissible paths.
    The union of these two path collections, together with the two connecting paths, yields at least $\lceil (k+1)/2 \rceil + \lceil (k+1)/2 \rceil - 1 \ge k$ admissible cycles.
    If these cycles do not cover $r \pmod k$, the path lengths in $\mathcal{P}_{u,w}^{M}$ must form a 2-AP. 
    Since $T$ is non-bipartite, we can combine $\mathcal{P}_{u,w}^{M}$ with a path in $\mathcal{P}_{u,w}^{T}$ of a suitable parity to obtain $\lceil (k+1)/2 \rceil$ admissible even cycles, which must contain one of length $r \pmod k$, a contradiction.
    This proves Claim~2.
	\end{proof}
    
    \medskip
	
	\noindent{\bf Claim 3.} No two consecutive vertices on $\partial T$ have degrees that are both at most $k-1$ in $V(T)\cup V(M)$.

    \begin{proof}
    By Claim 2, each component of $G - T - M$ has order at most two.
    Hence, every $v\in V(G-T-M)$ satisfies $\deg_{G}(v) \leq 1+\deg_T(v) \leq 1+\lfloor t/2 \rfloor < k-1$.
    Thus, $|G-T-M| \le 1$, with equality only if $G^\star$ is of Type I and $V(G-T-M)=\{\theta\}$.

    Recall $\partial T=v_{0}v_{1}\dots v_{t-1}v_{0}$.
    Suppose for the sake of contradiction that there exist consecutive vertices, say $v_0$ and $v_1$, such that their degrees in $V(T)\cup V(M)$ are at most $k-1$.
    
    If $V(G-T-M) \neq \emptyset$, then as noted above, $V(G-T-M)=\{\theta\}$. 
    Since $\deg_{G}(v) \ge k$ for all $v \neq \theta$, the vertices $v_0$ and $v_1$ must be adjacent to $\theta$.
    It follows that $\theta$ is adjacent to consecutive vertices on $\partial T$, contradicting the maximality of $T$.
    Otherwise, $V(G-T-M) = \emptyset$, so $\deg_{G}(v_0) \le k-1$ and $\deg_{G}(v_1) \le k-1$.
    This implies $G^\star$ is of Type II and $\{v_0, v_1\} = \{\theta_1, \theta_2\}$.
    Consequently, $\theta_1\theta_2 \in E(G)$, which implies $\delta(G) \ge k$, a contradiction.
    This proves Claim~3.
    \end{proof}

    Returning to the main proof, let $u = u_M$.
    Choose distinct vertices $v_{i} \in N_{T}(u)$ and $v_{j} \in N_{T}(M-u)$ to minimize the distance $\mathrm{dist}_{\partial T}(v_i, v_j)$. 
    Without loss of generality, we assume the indices satisfy $0 \le i < j$ and the distance along the boundary is $j-i \le t/2$.
    By the minimality of $j-i$, we have $N_M(v_\ell) = \emptyset$ for all $i < \ell < j$.
    Since $v_j \in N_T(M-u)$, there exists a vertex $w \in V(M) \setminus \{u\}$ adjacent to $v_j$.
    By Claim~1 and Observation~\ref{obs:valid}, $\mathcal{P}_{u,w}^{M}$ contains $k-\max\{1,\lfloor t/2 \rfloor-2\}-1=\min\{k-2, k-\lfloor t/2 \rfloor + 1\}$ admissible paths.
    We proceed by distinguishing cases based on the order of $T$.

    If $t=5$, Observation~\ref{obs: small T-graph} implies that $\mathcal{P}_{u,w}^{T}$ always contains 3 consecutive paths. 
    Consequently, the union of paths in $\mathcal{P}_{u,w}^{M}$ and $\mathcal{P}_{u,w}^{T}$ yields at least $(k-2)+3-1=k$ consecutive cycles, a contradiction. 

    If $6 \le t \le k$, we must have $j-i \in \{1, 2\}$. 
    Indeed, if $j-i \ge 3$, then $N_M(v_{i+1}) = N_M(v_{i+2}) = \emptyset$. This implies that both $v_{i+1}$ and $v_{i+2}$ have degree at most $t-1 \le k-1$ in $V(T) \cup V(M)$, which contradicts Claim~3.
    By Proposition~\ref{prop: T-graph}, $\mathcal{P}_{u,w}^{T}$ contains $t-3$ consecutive paths. 
    The union of paths in $\mathcal{P}_{u,w}^{M}$ and $\mathcal{P}_{u,w}^{T}$ thus produces at least $(k-\lfloor t/2 \rfloor + 1) + (t-3) - 1 = k + (t - \lfloor t/2 \rfloor - 3) \ge k$ consecutive cycles, a contradiction.

    Finally, suppose $t=k+1$. The case $j-i \in \{1, 2\}$ follows the same reasoning as above.
    If $j-i \ge 3$, it follows from Claim~3 that at least one of $v_{i+1}, v_{i+2}$ must have degree $k=t-1$ in $T$. 
    If $\deg_{T}(v_{i+1})=k$, consider paths of the form $v_i v_{i-1} \dots v_{\alpha} v_{i+1} v_{\beta} v_{\beta+1} \dots v_j$, where $\alpha \in [j+1, i]$ and $\beta \in [i+2, j]$ are indices such that $\alpha, \beta$ have the same parity as $i$. This construction yields $\Pset_{v_{i},v_{j}}^{T} \supseteq [2, t-1]$. 
    If $\deg_{T}(v_{i+2})=k$, the paths of the form $v_i v_{i-1} \dots v_{\alpha} v_{i+2} v_{\beta} v_{\beta+1} \dots v_j$ (which avoids $v_{i+1}$) establish $\Pset_{v_{i},v_{j}}^{T} \supseteq [2, t-2]$. 
    In either case, $\mathcal{P}_{u,w}^{T}$ contains at least $t-3$ consecutive paths.
    The union of paths in $\mathcal{P}_{u,w}^{M}$ and $\mathcal{P}_{u,w}^{T}$ thus produces at least $(k-\lfloor t/2 \rfloor + 1) + (t-3) - 1 \ge k$ consecutive cycles, a contradiction.
    This finishes the proof of Lemma~\ref{T6}.
\end{proof}

Now we are ready to complete the proof of Theorem~\ref{main}.

\begin{proof}[\bf Proof of Theorem~\ref{main}.]
Theorem~\ref{main} follows from Lemma~\ref{K3}, which handles the triangle-free case, and Lemmas~\ref{T4},~\ref{T5},~\ref{T6}, which collectively cover all possible cases of the maximum trigonal subgraph $T$. This concludes the proof of Theorem~\ref{thm:weak graph} for the non-bipartite case.
\end{proof}

We remark that for $k \ge 7$, our proof in fact yields the existence of $k$ admissible cycles; see the following corollary. 
This case can be verified directly from the arguments, and is distinct from the case $k=6$, mainly due to differences in the proofs of Claim 1 in Lemmas~\ref{K3}~and~\ref{T6}.

\begin{corollary}\label{cor: non-bipartite admis}
    Let $k\ge 7$ be an integer, and let $G^\star$ be a $k$-weak graph not isomorphic to $K_{k+1}$. If $G$ is non-bipartite, then $G^\star$ contains $k$ admissible cycles.
\end{corollary}

\section{Proof of Theorem~\ref{thm:weak graph}: the bipartite case}
\noindent In this section, we establish Theorem~\ref{thm:weak graph} for the case where $G$ is bipartite via the following.

\begin{thm}\label{thm:main bipartite}
    Let $k\ge 6$ be an integer, and let $G^\star$ be a $k$-weak graph not isomorphic to $K_{k,k}$ or $H_{k,n;t}$ for any $2\le t\le k< n$. If $G$ is bipartite, then $G^\star$ contains a cycle of length $r \pmod k$ for every even integer $r$.
\end{thm}

Note that in bipartite graphs, a collection of $k$ admissible cycles contains one of length $r \pmod k$ for every even $r$.
For $k\geq 7$, we obtain a stronger result for admissible cycles as follows.

\begin{thm}\label{thm:main bipartite admis}
    Let $k\ge 7$ be an integer, and let $G^\star$ be a $k$-weak graph not isomorphic to $K_{k,k}$ or $H_{k,n;t}$ for any $2\le t\le k< n$. If $G$ is bipartite, then $G^\star$ contains $k$ admissible cycles.
\end{thm}

\noindent Throughout this section, we assume the conditions of Theorem~\ref{thm:main bipartite admis}: 
\[
\mbox{ $k\geq 7$, ~$G^\star$ is not isomorphic to $K_{k,k}$ or $H_{k,n;t}$ for any $2\le t\le k< n$, ~ and $G$ is bipartite.}
\]
The main bulk of this section is used to prove Theorem~\ref{thm:main bipartite admis}.
We conclude this section by establishing Theorem~\ref{thm:main bipartite} via Theorem~\ref{thm:main bipartite admis} and outlining the proof of Theorem~\ref{thm:exist admis}.

\subsection{No tetragonal subgraph on six vertices} 
\noindent Our first lemma treats the $C_{4}$-free case, using a minimal induced cycle as the core subgraph. 

\begin{lem}\label{lem: C4-free}
    If $G$ is $C_{4}$-free, then $G^\star$ contains $k$ admissible cycles. 
\end{lem}

\begin{proof}
    Let $C=v_0v_1\cdots v_{2s-1}v_0$ ($s\geq 3$) be an induced cycle in $G$ with $|C|=2s$ minimized.

    We first show that for any vertex $v\in V(G-C)$, $\deg_{C}(v)\le 1$.
    Suppose to the contrary that $\deg_C(v) \ge 2$. Without loss of generality, assume that $v_0, v_i \in N_C(v)$ for some even index $i \in [s]$, and $v_1, \dots, v_{i-1} \notin N_C(v)$.
    Then $v v_0 v_1 \dots v_i v$ is an induced cycle of length $i+2 \le s+2 < 2s$, which contradicts the minimality of $|C|$.
    Since $\delta(G)\geq 3$, every vertex in $C$ must have a neighbor in $G-C$, which implies $G-C\neq \emptyset$.
    Moreover, for any component $M$ of $G-C$, we have $\delta(M)\ge \delta(G)-1\ge 2$, and thus $|M|\geq 3$.

    Suppose first that $G-C$ is not connected. 
    Let $M_{1}$ and $M_{2}$ be two distinct components of $G-C$.
    For $i \in \{1, 2\}$, let $u_i = u_{M_i}$ (see Definition~\ref{dfn: block,vertex}). 
    By Lemma~\ref{lem: B}, for any vertex $v \in V(M_i) \setminus \{u_i\}$, the pair $(u_i, v)$ is $(k-1)$-valid.
    By Lemma~\ref{lem:menger extend}, there exist vertices $w_i \in V(M_i) \setminus \{u_i\}$ and two disjoint paths between $\{u_1, w_1\}$ and $\{u_2, w_2\}$ in $G$ whose internal vertices lie in $C$.
    By Observation~\ref{obs:valid}, $\mathcal{P}_{u_{i},w_{i}}^{M_{i}}$ contains $k-2$ admissible paths.
    The union of paths in $\mathcal{P}_{u_{1},w_{1}}^{M_{1}}$ and $\mathcal{P}_{u_{2},w_{2}}^{M_{2}}$, together with the two connecting paths, produces at least $(k-2)+(k-2)-1 > k$ admissible cycles.

    Now assume that $M := G-C$ is connected, so $|M| \ge 3$. 
    Let $u = u_M$. 
    Without loss of generality, we assume $N_{C}(u)=\{v_{0}\}$.
    Since every vertex in $C$ has a neighbor in $M$, there exists $w \in V(M)$ adjacent to $v_{s-2}$. Since $s \ge 3$, we have $v_{s-2} \neq v_0$, which implies $w \neq u$.
    By Observation~\ref{obs:valid} and Lemma~\ref{lem: B}, $\mathcal{P}_{u,w}^{M}$ contains $k-2$ admissible paths.
    One can verify that $\Pset_{u,w}^{C}=\{s,s+4\}$.
    Since $k \ge 6$, the union of paths in $\mathcal{P}_{u,w}^{M}$ and $\mathcal{P}_{u,w}^{C}$ produces $k$ admissible cycles.
    This completes the proof.
\end{proof}

Throughout the rest of this section, we assume that $C_4\subseteq G$, and 
\[
\mbox{let $T$ be an optimal tetragonal subgraph of $G$ (recall Definition~\ref{dfn:best Q}).}
\]
The remainder of the proof is divided according to the order of \(T\): Lemmas~\ref{lem: m=2 I} and~\ref{lem: m=2 II} consider the case \(|T|=4\), while Lemma~\ref{Q1:m>2} handles the case \(|T|>4\).

In the following two lemmas, we utilize a subgraph $K_{2,t}$ with maximum $t$ as the core subgraph.

\begin{lem}\label{lem: m=2 I}
If $G^\star$ is of Type I and $|T|=4$, then $G$ contains $k$ admissible cycles.
\end{lem}

\begin{proof}
    Recall that $G=G^\star$ since $G^\star$ is of Type I.
    Let $K \simeq K_{2,t}$ (with $t \ge 2$) be a subgraph of $G$ that maximizes $t$.
    Let $(X, Y)$ denote the partite sets of $K$ such that $|X|=2$ and $|Y|=t$.
    The maximality of $K$ and $T$ implies that $\deg_K(v) \le 1$ for every vertex $v \in V(G-K)$.

    Suppose first that there exists a component $M$ of $G-Y$ such that $V(M) \cap X = \emptyset$. Then $Y$ separates $X$ from $M$ in $G$.
    Since $G$ is 3-connected, we must have $|N_Y(M)| \ge 3$.
    This implies $|Y| \ge 3$. Moreover, since $\deg_Y(v) = \deg_K(v) \le 1$ for every $v \in V(M)$, we also have $|M| \ge |N_Y(M)| \ge 3$.
    Fix a vertex $u \in N_Y(M)$ and let $W := Y \setminus \{u\}$.
    Let $H$ be the graph obtained from $G[V(M) \cup Y]$ by contracting $W$ into a single vertex $w$.
    Since both $u$ and $W$ have neighbors in $M$, $(H, u, w)$ is a 2-connected rooted graph with $\delta_2(H, u, w) \ge \delta_2(G) \ge k$ and $|H|\geq |M|+2>4$.
    By Lemma~\ref{lem:admis path}, there exist $k-1$ admissible $(u, w)$-paths in $H$.
    Lifting this back to $G$, this implies that the path collection $\bigcup_{v \in W} \mathcal{P}_{u,v}^M$ contains $k-1$ admissible paths.
    Since $|Y| \ge 3$, for any $v \in W$, we have $\Pset_{u,v}^K = \{2, 4\}$.
    Consequently, combining the paths in $\bigcup_{v \in W} \mathcal{P}_{u,v}^M$ with those in $\mathcal{P}_{u,v}^K$ yields $(k-1) + 2 - 1 = k$ admissible cycles, which suffices.

    Consequently, we may assume that every component of $G-Y$ intersects $X$. 
    Let $G'$ be the graph obtained from $G$ by contracting $X$ into a vertex $x$ and $Y$ into a vertex $y$.
    Since $\delta_2(G) \ge k$, $G-K$ is non-empty.
    Let $L$ be an arbitrary component of $G-K$.
    Similar to the previous paragraph, we have $|L|\geq |N_K(L)|\geq 3$.
    By our assumption, $N_X(L) \neq \emptyset$.
    Since $|X|=2$ and $G$ is 3-connected, $X$ cannot be a 2-cut; thus $N_Y(L) \neq \emptyset$.
    It follows that $(G', x, y)$ is a 2-connected rooted graph with $|G'| \ge |L| + 2 > 4$ and $\delta_2(G', x, y) \ge k$.
    By Lemma~\ref{lem:admis path}, there are $k-1$ admissible $(x, y)$-paths in $G'$.
    Equivalently, in terms of $G$, the union $\bigcup_{u \in X, v \in Y} \mathcal{P}_{u,v}^{G-K}$ contains $k-1$ admissible paths.
    Observe that for any $u \in X$ and $v \in Y$, we have $\Pset_{u,v}^K = \{1, 3\}$.
    Therefore, combining the paths in $\mathcal{P}_{u,v}^{G-K}$ (where $u \in X, v \in Y$) with the paths in $\mathcal{P}_{u,v}^K$ produces $(k-1) + 2 - 1 = k$ admissible cycles.
    This proves Lemma~\ref{lem: m=2 I}.
\end{proof}

\begin{lem}\label{lem: m=2 II}
    If $G^\star$ is of Type II and $|T|=4$, then $G^\star$ contains $k$ admissible cycles.
\end{lem}

\begin{proof}
    Let $K \simeq K_{2,t}$ be a subgraph of $G$ that maximizes $t$, and let $(X, Y)$ be its partite sets with $|X|=2$ and $|Y|=t$.
    The maximality of $K$ and $T$ implies that $\deg_K(v) \leq 1$ for any $v \in V(G-K)$.

    Suppose first that there exists a component $M$ of $G-Y$ such that $V(M) \cap X = \emptyset$. Then $Y$ separates $X$ from $M$ in $G$.
    For any $v \in V(M)$, we have $\deg_M(v) = \deg_{G}(v) - \deg_K(v) \geq (k-1) - 1 = k-2$, which implies $|M| \geq k-1 \geq 6$.
    Fix a vertex $u \in N_Y(M)$ and let $W := Y \setminus \{u\}$.
    Let $H$ be the graph obtained from $G[V(M) \cup Y]$ by contracting $W$ into a vertex $w$.
    Since $u$ is not a cut-vertex in $G$ (as $G$ is 2-connected), both $u$ and $W$ have neighbors in $M$, so $(H, u, w)$ is a 2-connected rooted graph.

    We distinguish two cases based on the location of $\theta_1$ and $\theta_2$.
    If $|V(M) \cap \{\theta_1, \theta_2\}| \le 1$, then clearly $\delta_2(H, u, w) \ge k$. 
    If $\{\theta_1, \theta_2\} \subseteq V(M)$, let $H'$ be the graph obtained from $G^\star[V(M) \cup Y \cup \{\theta\}]$ by contracting $W$ into a vertex $w$.
    Observe that $H'$ is obtained from $H$ by adding the vertex $\theta$ and the edges $\theta\theta_1, \theta\theta_2$.
    Since $(H, u, w)$ is a 2-connected rooted graph, it follows that $(H', u, w)$ is also 2-connected, and it is clear that $\delta_2(H', u, w) \ge k$.
    Applying Lemma~\ref{lem:admis path} to $H$ (in the first case) or $H'$ (in the second case), we always find $k-1$ admissible $(u, w)$-paths in $H'$.
    In either case, the path collection $\bigcup_{v \in W} \mathcal{P}_{u,v}^{G^\star[V(M) \cup \{\theta\}]}$ contains $k-1$ admissible paths.

    Now consider paths whose internal vertices lie in $K$, based on the size of $Y$.
    If $t \geq 3$, then for any $v \in W$, we have $\Pset_{u,v}^K = \{2, 4\}$.
    Combining these with the $k-1$ admissible paths in $\bigcup_{v \in W} \mathcal{P}_{u,v}^{G^\star[V(M) \cup \{\theta\}]}$ yields $(k-1) + 2 - 1 = k$ admissible cycles in $G^\star$.
    If $t=2$, since $G + \theta_1\theta_2$ is 3-connected, the graph $(G + \theta_1\theta_2) - Y$ must be connected.
    Consequently, $G-Y$ contains exactly one component $N$ distinct from $M$, and each component contains exactly one of $\theta_1, \theta_2$. $N$ (distinct from $M$) such that one of $\theta_1, \theta_2$ lies in $M$ and the other in $N$.
    Let $Y=\{y_1, y_2\}$. The 2-connectivity of $G$ implies that both $(G[V(M) \cup Y], y_1, y_2)$ and $(G[V(N) \cup Y], y_1, y_2)$ are 2-connected rooted graphs, and one can readily verify that their second minimum degrees are at least $k$.
    By Lemma~\ref{lem:admis path}, both $\mathcal{P}^M_{y_1,y_2}$ and $\mathcal{P}^N_{y_1,y_2}$ contain $k-1$ admissible paths.
    The union of these paths yields $(k-1) + (k-1) - 1 > k$ admissible cycles in $G$ (and thus in $G^\star$).
    This settles the case where a component $M$ of $G-Y$ is disjoint from $X$.

    It remains to assume that every component of $G-Y$ intersects $X$. 
    Let $G'$ be the graph obtained from $G^\star$ by contracting $X$ into $x$ and $Y$ into $y$.
    Recall that $\{\theta_1, \theta_2\}$ is the unique 2-cut in $G^\star$.
    We proceed by discussing the position of $\{\theta_1, \theta_2\}$.
    If $\{\theta_1, \theta_2\} \neq X$, then $X$ is not a 2-cut in $G^\star$. Thus, every component of $G^\star - K$ has a neighbor in $Y$ (in the host graph $G^\star$); by assumption, this component also has a neighbor in $X$.
    It follows that $(G', x, y)$ is a 2-connected rooted graph with $\delta_2(G', x, y) \ge k$.
    If $\{\theta_1, \theta_2\} = X$, then $X$ is not a cut in $G$ (as $G + \theta_1\theta_2$ is 3-connected).
    Thus, every component of $G-K$ has neighbors both in $X$ and $Y$.
    Note that $G = G^\star - \theta$. We deduce that $(G' - \theta, x, y)$ is a 2-connected rooted graph with $\delta_2(G' - \theta, x, y) \ge k$.
    Applying Lemma~\ref{lem:admis path} to $(G', x, y)$ or $(G' - \theta, x, y)$, we obtain $k-1$ admissible $(x, y)$-paths in $G'$.
    This implies that $\bigcup_{u \in X, v \in Y} \mathcal{P}_{u,v}^{G^\star-K}$ contains $k-1$ admissible paths.
    Since $\Pset_{u,v}^{K} = \{1, 3\}$ for any $u \in X, v \in Y$, combining the corresponding path families yields $(k-1) + 2 - 1 = k$ admissible cycles in $G^\star$.
    This completes the proof of Lemma~\ref{lem: m=2 II}.
\end{proof}

\subsection{Tetragonal subgraphs of order at least six}
\noindent This subsection is devoted to the case $|T| \ge 6$, constituting the bulk of the proof of Theorem~\ref{thm:main bipartite admis}.

\begin{lem}\label{Q1:m>2}
	If $|T|\ge 6$, then $G^\star$ contains $k$ admissible cycles.
\end{lem} 

By Proposition~\ref{prop:Q-graph}, a tetragonal graph on at least $2k+2$ vertices contains cycles of all lengths in $\{4,6,\dots,2k+2\}$.
Thus, for the proof of Lemma~\ref{Q1:m>2}, we may assume that $|T|=2m$ with $m\in [3,k]$. 
Let $R:=\{v\in V(G-T):\deg_{T}(v)\ge m-1\}$.
In the following proof of this subsection,
\[
\mbox{the subgraph $T^*:=G[V(T)\cup R]$ will be used as the core subgraph of $G$.}
\]
We refer to Lemma~\ref{lem:operation} for properties of $T$ and $R$.
Our subsequent analysis mainly relies on examining the structural properties of the components of $G - T^*$. 
Specifically, Lemmas~\ref{lem: N} and~\ref{lem:RN extra} investigate components of order at most two, whereas Lemmas~\ref{lem:size of BM} through~\ref{lem: edges between M and R} (except for Lemma~\ref{lem:Q not complete} dealing with an extreme case) focus on components of order at least three.

The following lemma characterizes the set of path lengths $\Pset_{u,v}^{T^*}$.
\begin{lem}\label{lem:Q-admis}
    Let $u_1, u_2$ be two distinct vertices in $V(G-T^*)$ with positive degrees $d_1:=\deg_{T^*}(u_1)$ and $d_2:=\deg_{T^*}(u_2)$. Then one of the following statements holds:
    \begin{itemize}
        \item[(1)] $\mathcal{P}_{u_1,u_2}^{T^*}$ contains at least $\min\{d_1+d_2-1,m\}$ admissible paths.
        \item[(2)] $\Pset_{u_1,u_2}^{T^*} \supseteq \{2\} \cup \{2+d, 4+d, \dots, 2m+2-d\}$ for some even $d \leq \max\{2, 2m/\max\{d_1,d_2\}, \\m-d_1-d_2+3\}$.
    \end{itemize}
\end{lem}

\begin{proof}
    By Lemma~\ref{lem:operation}~(4), for each $i \in \{1, 2\}$, we have either $\deg_{T}(u_i)=0$ or $\deg_R(u_i)=0$.
    We proceed by discussing the degrees of $u_1$ and $u_2$ in $R$.
    
    First, suppose that both $\deg_R(u_1)$ and $\deg_R(u_2)$ are positive.
    It follows from Lemma~\ref{lem:operation}~(4) that $d_1=d_2=1$. 
    Thus, item (1) holds, as $\min\{d_1+d_2-1,m\}=1$.
    
    Next, consider the case where exactly one of $\deg_R(u_1)$ and $\deg_R(u_2)$ is positive.
    Without loss of generality, assume $\deg_R(u_1)>0$. Then Lemma~\ref{lem:operation}~(4) implies $d_1=1$ and $d_2\leq m-2$.
    Let $N_R(u_1)=\{p\}$. Recall that $\deg_T(p) \ge m-1$.
    By Lemma~\ref{lem:short dist}~(2), there exist distinct vertices $v_1\in N_T(p)$ and $v_2\in N_T(u_2)$ such that $\mathrm{dist}_{\partial T}(v_1,v_2)\leq 3$.
    By Proposition~\ref{prop:Q-graph}, $\mathcal{P}_{p,u_2}^{T}$ contains at least $m-\mathrm{dist}_{\partial T}(v_1,v_2)+1 \geq m-2 \geq d_1+d_2-1$ admissible paths.
    Consequently, $\mathcal{P}_{u_1,u_2}^{T^*}$ also contains $d_1+d_2-1$ admissible paths (extended by the edge $u_1p$), so item (1) follows.

    Finally, assume that $\deg_R(u_1)=\deg_R(u_2)=0$.
    If $d_1=d_2=1$, then $d_1+d_2-1=1$, and item~(1) holds trivially.
    Suppose now that $\max\{d_1,d_2\}\geq 2$. 
    If $N_T(u_1)\cap N_T(u_2)=\emptyset$, then Lemma~\ref{lem:short dist}~(1) implies that $u_1$ and $u_2$ have neighbors in $T$ with distance along $\partial T$ at most $\max\{1, m+2-d_1-d_2\}$.
    It follows from Proposition~\ref{prop:Q-graph} that $\mathcal{P}^T_{u_1,u_2}$ contains $\min\{d_1+d_2-1, m\}$ admissible paths, satisfying~(1).
    Otherwise, if $N_T(u_1)\cap N_T(u_2)\neq\emptyset$, Lemma~\ref{lem:short dist}~(2) ensures that they have distinct neighbors in $T$ whose distance along $\partial T$ is bounded by $\max\{3, 2m/\max\{d_1,d_2\}, m-d_1-d_2+3\}$. 
    In this scenario, Proposition~\ref{prop:Q-graph} implies that $\Pset^T_{u_1,u_2}$ (and thus $\Pset^{T^*}_{u_1,u_2}$) includes $\{2\} \cup \{2+d, 4+d, \dots, 2m+2-d\}$ for some even integer $d \leq \max\{2, 2m/\max\{d_1,d_2\}, m-d_1-d_2+3\}$, satisfying~(2).
\end{proof}

Towards the proof of Lemma~\ref{Q1:m>2}, we then present several lemmas that analyze the structures of $G-T^*$.
We first describe the components of $G-T^*$ of order at most two.
Let $N$ be the graph obtained from $G - T^*$ by deleting all components of order at least three.

\begin{lem}\label{lem: N}
The graph $N$ satisfies the following properties:
\begin{itemize}
    \item[(1)] $|V(N)| \leq 1$, and $|V(N)| = 1$ implies that $G^\star$ is of Type I and $V(N) = \{\theta\}$.
    \item[(2)] $R \cup V(N)$ is an independent set.
\end{itemize}
\end{lem}

\begin{proof}
    Since every component in $N$ has order at most two, it follows that for every $v\in V(N)$, $\deg_{G}(v)=\deg_{T^*}(v)+\deg_N(v)\leq (m-2)+1\leq k-1$, which implies $N$ has order at most two.

    If $|V(N)|=2$, then the inequality above must be an equality. This forces $G^\star$ to be of Type II, $N$ to be the edge $\theta_1\theta_2$, and $\deg_{G}(\theta_1)=\deg_{G}(\theta_2)=k-1$.
    However, the existence of the edge $\theta_1\theta_2$ in $G$ implies $\delta(G) \geq k$, a contradiction.
    
    Thus, we must have $|V(N)|\leq 1$. When $V(N)=\{v\}$, we have $\deg_{G}(v) = \deg_{T^*}(v) \leq m-2\leq k-2$, which implies that $G^\star$ is of Type I and $v = \theta$. This completes the proof of (1).

    For (2), suppose to the contrary that $R \cup V(N)$ is not independent. 
    Recall Lemma~\ref{lem:operation}~(2) that $R$ is independent, so $V(N) \neq \emptyset$.
    By (1), $G^\star$ is of Type I and $V(N) = \{\theta\}$.
    Recall that $\deg_{G}(\theta) \geq \delta(G) \geq 3$. 
    By Lemma~\ref{lem:operation}~(4), we have $\deg_R(\theta) = 0$, implying that $R \cup V(N)$ remains an independent set, a contradiction. This proves (2).
\end{proof}

The following lemma shows that for every component $M$ of $G-T^*$ with at least three vertices, it suffices to consider that the corresponding end-block $B_M$ (See Definition~\ref{dfn: block,vertex}) contains at least four vertices.
\begin{lem}\label{lem:size of BM}
	Let $M$ be a component of $G-T^*$ of order at least $3$. If $|B_{M}| \leq 3$, then $G$ contains $k$ admissible cycles.
\end{lem}

\begin{proof}
    We note that $|B_M| \neq 3$, as otherwise $B_M \simeq K_3$, contradicting the fact that $G$ is bipartite.
    Suppose $|B_M| = 2$. Then $u_M$ is the unique vertex in $V(B_M) \setminus \mathrm{Cut}(B_M)$. 
    By Lemma~\ref{lem:operation}~(4), $\deg_{T^*}(u_M)\leq m-2$, and thus $\deg_{G}(u_M)=1+\deg_{T^*}(u_M)\leq 1+(m-2)=m-1\leq k-1$.
    Hence, $G^\star$ must be of Type II, for otherwise $G^\star$ is of Type I and $u_M$ must be $\theta$, which contradicts the condition that $\theta \notin V(B_M) \setminus \mathrm{Cut}(B_M)$.
    It follows that the inequality for $\deg_{G}(u_M)$ must hold with equality.
    Specifically, we derive that $u_M \in \{\theta_1, \theta_2\}$, $m=k \geq 6$, and $\deg_{T^*}(u_M) = k-2$.

    Recall from Definition~\ref{dfn: block,vertex} that $B_M$ is chosen to be an end-block of maximum order satisfying the stated conditions.
    Thus, any other end-block $B$ of $M$ (which already satisfies that $V(B) \setminus \mathrm{Cut}(B)$ contains at most one of $\theta_1$ and $\theta_2$) must also satisfy $|B| \leq 2$. This further implies that the unique vertex in $V(B) \setminus \mathrm{Cut}(B)$ must be in $\{\theta_1, \theta_2\}$.
    It follows that $M$ has exactly two end-blocks, say $B_M$ and $B$, with $V(B_M) \setminus \mathrm{Cut}(B_M) = \{\theta_1\}$, $V(B) \setminus \mathrm{Cut}(B) = \{\theta_2\}$, and $\deg_{T^*}(\theta_1) = \deg_{T^*}(\theta_2) = k-2$.
    A routine calculation verifies that each case of Lemma~\ref{lem:Q-admis} yields $k$ admissible paths in $\mathcal{P}_{\theta_{1},\theta_{2}}^{T^*}$.
    Therefore, the union of an arbitrary path in $\mathcal{P}_{\theta_{1},\theta_{2}}^{M}$ and paths in $\mathcal{P}_{\theta_{1},\theta_{2}}^{T^*}$ produces $k$ admissible cycles in $G$.
	This completes the proof.
\end{proof}

Now we consider the special case $G[V(T)]\simeq K_{k,k}$.
\begin{lem}\label{lem:Q not complete}
	If $m=k$ and $G[V(T)]\simeq K_{k,k}$, then $G^\star$ contains $k$ admissible cycles.
\end{lem}

\begin{proof}
    Suppose for the sake of contradiction that $G[V(T)]\simeq K_{k,k}$ but $G^\star$ does not contain $k$ admissible cycles. We proceed by analyzing the components of $G-T^*$.
	
	First, assume that $G-T^*$ has a component $M$ of order at least $3$.
	Let $u = u_{M}$, and let $v \in V(M)\setminus\{u\}$ be a vertex maximizing $\deg_{T^*}(v)$. Additionally, if $\deg_{T^*}(u)=1$, we may assume $N_{T^*}(v) \neq N_{T^*}(u)$, since the unique neighbor of $u$ in $T^*$ cannot be a cut-vertex in $G$.
	By Lemma~\ref{lem:operation}~(4), we have $1\le \deg_{T^*}(u), \deg_{T^*}(v)\le k-2$.
	Consider the sum $\deg_{T^*}(u)+\deg_{T^*}(v)$. 

    If $\deg_{T^*}(u)+\deg_{T^*}(v)\ge k+1$, then $\deg_{T^*}(u) ,\deg_{T^*}(v) \ge 3$.
    It follows from Lemma~\ref{lem:operation}~(4) that $\deg_R(u)=\deg_R(v)=0$.
    Since $G[V(T)] \simeq K_{k,k}$, the set $\Pset_{u,v}^{T}$ contains $k$ admissible lengths. Specifically, $\Pset_{u,v}^{T}=\{3,5,\dots,2k+1\}$ if $u$ and $v$ belong to different partite sets, and $\Pset_{u,v}^{T}=\{2,4,\dots,2k\}$ otherwise (as $u$ and $v$ share a common neighbor in $T$).
    Hence, the union of an arbitrary path in $\mathcal{P}_{u,v}^M$ with paths in $\mathcal{P}_{u,v}^{T}$ yields $k$ admissible cycles, a contradiction.

    Now assume that $\deg_{T^*}(u)+\deg_{T^*}(v)\le k$.
    By Lemma~\ref{lem:size of BM}, we have $|B_M|\geq 4$.
    By the choice of $u$ and $v$, every vertex $w\in V(B_{M})\setminus\{u,v\}$ satisfies $\deg_{T^*}(w)\leq \min\{\deg_{T^*}(u),\deg_{T^*}(v)\}\leq \lfloor k/2 \rfloor$.
    Observe that $\deg_G(w) \ge k-1$ holds for all $w \in V(B_M)$—except possibly for $\theta$ if $M$ is 2-connected, or the vertex in $\mathrm{Cut}(B_M)$ otherwise.
    It follows that for all but at most one vertex $w \in V(B_{M})\setminus\{u,v\}$,
    $\deg_{B_{M}}(w) = \deg_G(w) - \deg_{T^*}(w) \ge k - 1 - \lfloor k/2 \rfloor \ge 3.$
    Consequently, the pair $(u,v)$ is $3$-valid, and $\mathcal{P}_{u,v}^{M}$ contains 2 admissible paths.
    Since $G[V(T)]\simeq K_{k,k}$, and $u,v$ are adjacent to distinct vertices in $T^*$ (due to the selection of $v$), it follows that $\mathcal{P}_{u,v}^{T^*}$ contains $k-1$ admissible paths.
    The union of paths in $\mathcal{P}_{u,v}^{M}$ and $\mathcal{P}_{u,v}^{T^*}$ thus yields $2+(k-1)-1=k$ admissible cycles.

    It remains to consider the case where every component of $G-T^*$ has order at most 2, i.e., $G-T^* = N$.
    Let $(A,B)$ be the partite sets of $G$. 
    We claim that $V(G-T) = R \cup V(N)$ is contained entirely in $A$ or $B$.
    If $V(N) = \emptyset$, then the claim follows from Lemma~\ref{lem:operation}~(3).
    Assume $V(N) \neq \emptyset$. 
    It follows from Lemma~\ref{lem: N} that $G^\star$ is of Type I, $V(N)=\{\theta\}$, and $R\cup \{\theta\}$ is an independent set.
    Suppose the claim is false; then there exist $a \in A \cap (R \cup \{\theta\})$ and $b \in B \cap (R \cup \{\theta\})$.
    Every vertex in $R$ has degree at least $k-1 \ge 6$ in $T$, and $\deg_T(\theta) \ge \delta(G) \ge 3$; thus $\deg_{T}(a), \deg_{T}(b) \ge 3$.
    Since $G[V(T)] \simeq K_{k,k}$, and $a, b$ have at least three neighbors in distinct partite sets of $T$, a routine verification confirms that $G[V(T) \cup \{a,b\}]$ contains cycles of all lengths in $\{4, 6, \dots, 2k+2\}$, a contradiction. This proves the claim; hence, we may assume $R \cup V(N) \subseteq A$.

    We conclude by discussing the type of $G^\star$.
    If $G^\star$ is of Type I, then every vertex in $R \cup V(N)$ (with the exception of $\theta$) has degree $k$ in $T$, hence must be fully connected to $V(T) \cap B$.
    Consequently, $G^\star$ is obtained from $G[V(T)] \simeq K_{k,k}$ by adding vertices to $A$ that are adjacent to all vertices in $B$, with the possible exception of one vertex (namely $\theta$) that is adjacent to at least three vertices in $B$. 
    Thus, $G^\star$ is isomorphic to $K_{k,k}$ or $H_{k,n;t}$ for some $3 \le t \le k < n$, a contradiction.
    If $G^\star$ is of Type II, then Lemma~\ref{lem: N} implies $V(N)=\emptyset$. 
    Moreover, every $p\in R$ has at least $k-1$ neighbors in $V(T)\cap B$.
    One can verify that for any distinct $a, a' \in A$ and $b \in B$, we have $\Pset_{a,a'}^{G} \supseteq \Pset_{a,a'}^{V(T)\cup \{p\}}\supseteq \{2,4,\dots,2k\}$ and $\Pset_{a,b}^{G} \supseteq \Pset_{a,b}^{T} \supseteq\{1,3,\dots,2k-1\}$.
    If $\theta_1, \theta_2 \in A$ or if they belong to different partite sets, then $\mathcal{P}_{\theta_1, \theta_2}^{G}$ contains $k$ admissible paths. Combined with the path $\theta_1\theta\theta_2$, we obtain $k$ admissible cycles in $G^\star$.
    If $\theta_1, \theta_2 \in B$, then $R \cap \{\theta_1, \theta_2\} = \emptyset$ (as $R \subseteq A$).
    Thus, every $p \in R$ has $\deg_G(p) \ge k$, forcing $\deg_T(p)=k$ (i.e., $p$ is adjacent to all vertices in $B$).
    This implies that $G \simeq K_{k,n-1}$ for some $n>k$, and thus $G^\star \simeq H_{k,n;2}$, a contradiction.
    This completes the proof.
\end{proof}

Recall that $K_{n,n}^{-}$ is the graph obtained from $K_{n,n}$ by removing an edge.
The following lemma considers the case that $G-T^*$ contains no component of order at least three. 

\begin{lem}\label{lem:non empty}
	If $G-T^*$ does not contain any component of order at least three, then $G^\star$ contains $k$ admissible cycles.
\end{lem}

\begin{proof}
    The assumption implies $G-T^* = N$.
    By Lemma~\ref{lem:Q not complete}, we may assume $G[V(T)] \not\simeq K_{k,k}$.
    Then Lemma~\ref{lem:operation}~(1) implies that every $p \in R$ satisfies $\deg_T(p) \le k-1$.
    Let $(A,B)$ be the partite sets of $G$, and select vertices $a \in A \cap V(T)$, $b \in B \cap V(T)$ satisfying that $\deg_T(a), \deg_T(b) \leq k-1$.

    Suppose first that $V(N) \neq \emptyset$. 
    By Lemma~\ref{lem: N}, $G^\star$ is of Type I, $V(N)=\{\theta\}$, and $R \cup V(N)$ is independent.
    Thus, every $p \in R$ has $\deg_{G}(p)\le \deg_T(p)\leq k-1$, implying that $R=\emptyset$.
    However, one of $a, b$ is not adjacent to $\theta$. This vertex would then have degree less than $k$ in $G$, a contradiction.
    
    Thus, we assume $V(N)= \emptyset$.
    We claim that the set $\{a\} \cup N_R(a)$ contains a vertex with degree at most $k-1$ in $G$.
    Indeed, if $N_R(a) = \emptyset$, then $\deg_G(a) = \deg_T(a) \le k-1$. If $N_R(a) \neq \emptyset$, then every $p \in N_R(a)$ has $\deg_G(p) = \deg_T(p) \le k-1$, proving the claim.
    The same holds for $\{b\} \cup N_R(b)$.
    Hence, $G^\star$ must be of Type II, and each of $\{a\} \cup N_R(a)$ and $\{b\} \cup N_R(b)$ contains exactly one of $\theta_1, \theta_2$. Also, we have $\theta_1\theta_2 \notin E(G)$ (otherwise $\delta(G) \ge k$).

    We distinguish three cases regarding the locations of $\theta_1$ and $\theta_2$.
    First, assume that $\{\theta_1, \theta_2\} = \{a, b\}$. In this case, we must have $R = \emptyset$, and every vertex in $V(T) \setminus \{a, b\}$ has degree $k$ in $T$. Consequently, $G[V(T)] \simeq K_{k,k}^-$. It follows that $G^\star = G + \{\theta\theta_1, \theta\theta_2\}$ contains cycles of all lengths in $[4,2k+1]$, as desired.
    
    Next, suppose that $\theta_1 = a$ and $\theta_2 \in N_R(b)$ (the symmetric case is analogous). Then $R = \{\theta_2\}$. Since $\theta_2$ is adjacent to all but at most one vertex in $V(T) \cap B$, there exists a neighbor $c \in N_T(\theta_2) \subseteq B$ such that $c$ is adjacent to $\theta_1$ on $\partial T$.
    Thus, $G^\star = G + \{\theta\theta_1, \theta\theta_2\}$ contains a tetragonal graph on $2k+2$ vertices, whose boundary cycle is obtained from $\partial T$ by replacing the edge $\theta_1 c$ with the path $\theta_1 \theta \theta_2 c$. By Proposition~\ref{prop:Q-graph}, $G^\star$ contains cycles of all lengths in $\{4,6,\dots,2k+2\}$, which suffices.

    Finally, consider the case where $\theta_1 \in N_R(a)$ and $\theta_2 \in N_R(b)$. Then $R = \{\theta_1, \theta_2\}$. 
    Observe that $k \le \deg_G(a) \le 1 + m$, which implies $m \ge 5$.
    By Lemma~\ref{lem:operation}~(3), $R$ is entirely contained in $A$ or $B$. 
    However, since $a \in A$ and $b \in B$, we necessarily have $\theta_1 \in B$ and $\theta_2 \in A$.
    Thus, $R$ intersects both partite sets, a contradiction.
    This completes the proof.
\end{proof}

According to Lemma~\ref{lem:non empty}, we may assume that $G-T^*$ has a component of order at least three.
Let $M$ be such a component.
The following lemma reduces the proof to the scenario where the degrees of vertices in $M$ are highly constrained.
\begin{lem}\label{lem:one comp deg}
    If there exists $w\in V(M)$ with $\deg_{T^*}(w)\geq 3$, then $G$ has $k$ admissible cycles.
\end{lem}

\begin{proof}
    Let $u$ denote $u_{M}$. 
	By Lemma~\ref{lem:size of BM}, we may assume that $|B_{M}|\geq 4$.
	Select a vertex $v\in V(M)\setminus\{u\}$ as follows.
	\begin{itemize}
		\item[(1)] $\deg_{T^*}(v)=\max\{\deg_{T^*}(w):w\in V(M)\setminus \{u\}\}$. 
		
		\item[(2)] Subject to (1), $\deg_{G}(v)$ is minimum.
	\end{itemize}

    Let $d_{\min} = \min\{\deg_{T^*}(u),\deg_{T^*}(v)\}$ and $d_{\max} = \max\{\deg_{T^*}(u),\deg_{T^*}(v)\}$.
    Then $d_{\max} \ge3$ by the hypothesis, and Lemma~\ref{lem:operation}~(4) implies that both $d_{\min}$ and $d_{\max}$ are at most $m-2$.
    Moreover, we have $d_{\min} \ge 1$, since $\deg_{T^*}(u) \ge 1$ holds by definition, while $\deg_{T^*}(v) \ge 1$ follows from the fact that $u$ is not a cut-vertex of $G$.
    Combining these bounds yields $1\le d_{\min}\le m-2$ and $3\le d_{\max}\le m-2$, which forces $m\ge 5$.

    Using an argument analogous to Lemma~\ref{lem: B}, we deduce that $(u,v)$ is $(k-d_{\min})$-valid.
    In fact, the proof of Lemma~\ref{lem: B} relied on the fact that $\deg_{T^*}(w) \le \deg_{T^*}(u)$ for every $w\in V(B_M)\setminus\big(\{u\}\cup \mathrm{Cut}(B_M)\big)$. 
    Here, by the maximality of $\deg_{T^*}(v)$, every $w\in V(B_M)\setminus\big(\{u\}\cup \mathrm{Cut}(B_M)\big)$ satisfies $\deg_{T^*}(w) \le \deg_{T^*}(v)$, and thus $\deg_{T^*}(w) \le d_{\min}$.
    Substituting this stronger inequality into the proof of Lemma~\ref{lem: B} confirms that $(u, v)$ is $(k-d_{\min})$-valid.
    By Observation~\ref{obs:valid}, $\mathcal{P}^M_{u,v}$ contains $k-d_{\min}-1$ admissible paths.
    Let $\Pset_1$ denote the set of lengths of these admissible
    paths.
    According to Lemma~\ref{lem:Q-admis}, $\Pset^{T^*}_{u,v}$ contains a subset $\Pset_2$ with one of the following forms:
\begin{itemize}
    \item An admissible set of size at least $\min\{d_{\min}+d_{\max}-1,m\}$;
    \item $\{2\}\cup\{2+d,4+d,\cdots,2m+2-d\}$ for some even $d\leq\max\{2,2m/d_{\max},m-d_{\min}-d_{\max}+3\}$.
\end{itemize}
We verify that $\Pset_1+\Pset_2$ is an admissible set of size at least $k$.
If $\Pset_2$ is of the former case, since $3\le d_{\max}\le m-2$, the sum $\Pset_1+\Pset_2$ is an admissible set of length at least $(k-d_{\min}-1)+\min\{d_{\min}+d_{\max}-1,m\}-1\geq k$.
In the latter case, the condition for $\Pset_1+\Pset_2$ to form an admissible set is equivalent to:
\begin{align*}
    d\leq \max\{\Pset_{1}\}-\min\{\Pset_{1}\}+2\nonumber
    \iff d\leq 2(k-d_{\min}-1)\nonumber
\end{align*}
This is guaranteed by the following strict inequality, since $d$ is even. 
\begin{align*}
   &d\leq \max\{2,2m/d_{\max},m-d_{\min}-d_{\max}+3\}
   \overset{\text{($\ast$)}}{<}2(m-d_{\min})\le 2(k-d_{\min}),
\end{align*}
where $(\ast)$ is implied by the following three inequalities (as $m\geq 5$):
\begin{align*}
    2(m-d_{\min})&\geq 4>2,\\ 
    2(m-d_{\min})&\geq 2(m-d_{\max})>2m/d_{\max}, ~~~\mbox{and}\\
    2(m-d_{\min})&>m-2d_{\min}+3\geq m-d_{\min}-d_{\max}+3.
\end{align*}
Therefore, $\Pset_1+\Pset_2$ is an admissible set with minimum element $\min\{\Pset_1\}+2$ and maximum element $\max\{\Pset_1\}+\max\{\Pset_2\}$, thus has size $$\frac{\max\{\Pset_1\}-\min\{\Pset_1\}+\max\{\Pset_2\}}{2}=k+m-d_{\min}-d/2-1.$$
Note that $k+m-d_{\min}-d/2-1\geq k$ is equivalent to $d\leq 2(m-d_{\min}-1)$, which we have already established via inequality $(\ast)$.
Hence, $\Pset_1+\Pset_2$ is an admissible set of size at least $k$, and the union of paths in $\mathcal{P}^M_{u,v}$ and $\mathcal{P}^{T^*}_{u,v}$ produces $k$ admissible cycles in $G$.
This completes the proof of Lemma~\ref{lem:one comp deg}.
\end{proof}

The following lemma addresses the case where $G-T^*$ has two components of order at least~3.

\begin{lem}\label{lem: G-Q connected 1}
	If $G-T^*$ has two components of order $\geq 3$, then $G$ has $k$ admissible cycles.
\end{lem}

\begin{proof}
    Suppose for a contradiction that $G-T^*$ contains two distinct components $M_{1}$ and $M_{2}$, both of order at least $3$, yet $G$ does not contain $k$ admissible cycles.
	By Lemma~\ref{lem:one comp deg}, every vertex $v \in V(M_1) \cup V(M_2)$ satisfies $\deg_{T^*}(v) \le 2$.

    For $i \in \{1, 2\}$, let $u_i = u_{M_i}$.
    Since $G$ is 2-connected, Lemma~\ref{lem:menger extend} guarantees the existence of vertices $w_i \in V(M_i) \setminus \{u_i\}$ for each $i$, along with two disjoint $(\{u_1,w_1\},\{u_2,w_2\})$-paths whose internal vertices lie in $T^*$.
    By Observation~\ref{obs:valid} and Lemma~\ref{lem: B}, the set $\mathcal{P}_{u_i, w_i}^{M_i}$ contains at least $k - \deg_{T^*}(u_i) - 1 \geq k-3$ admissible paths.
	Consequently, the union of the paths in $\mathcal{P}_{u_1, w_1}^{M_1}$ and $\mathcal{P}_{u_2, w_2}^{M_2}$, together with the two connecting paths, yields at least $(k-3) + (k-3) - 1 = 2k - 7 \ge k$ admissible cycles in $G$, a contradiction. 
	This completes the proof.
\end{proof}

By Lemma~\ref{lem: G-Q connected 1}, we may henceforth assume that $M$ is the only component of $G-T^*$ of order at least three. 
The following lemma further restricts the neighborhood of vertices in $M$.

\begin{lem}\label{lem: edges between M and R}
	If $E(M,R)\neq \emptyset$, then $G$ contains $k$ admissible cycles.
\end{lem}

\begin{proof}
    Suppose for the sake of contradiction that $E(M,R)\neq \emptyset$, yet $G^\star$ does not contain $k$ admissible cycles.
    Let $(A,B)$ be the partite sets of $G$, and let $u$ denote $u_{M}$.
    By Lemmas~\ref{lem:operation}~(4) and~\ref{lem:one comp deg}, we have $\deg_{T^*}(u) \le \min\{2, m-2\}$.

    We first rule out a specific configuration by showing that its existence yields $k$ admissible cycles. 
    We say that a triple $(v, x, y)$ is a \textbf{forbidden triple} if $v \in V(M) \setminus \{u\}$, $xy \in E(\partial T)$, and there exist two disjoint $(M,T)$-paths connecting $\{u, v\}$ to $\{x, y\}$.
    If such a triple exists, by Observation~\ref{obs:valid} and Lemma~\ref{lem: B}, $\mathcal{P}_{u,v}^{M}$ contains $k - \deg_{T^*}(u) - 1 \ge k-m+1$ admissible paths.
    By Proposition~\ref{prop:Q-graph}, $\mathcal{P}_{x,y}^{T}$ contains $m$ admissible paths.
    Consequently, the union of paths in $\mathcal{P}_{u,v}^{M}$ and $\mathcal{P}_{x,y}^{T}$, together with the two connecting paths, yields at least $(k-m+1) + m - 1 = k$ admissible cycles.
    Consequently, the existence of a forbidden triple yields a contradiction.

    We then claim that $R$ is contained in either $A$ or $B$.
    Recall from Lemma~\ref{lem:operation}~(3) that this property already holds when $m \ge 4$; thus, we may assume $m=3$ (i.e., $T$ consists of two 4-cycles sharing one edge).
    In this case, $\deg_{T^*}(u)=1$.
    Suppose the claim is false. Let $a \in R \cap A$ and $b \in R \cap B$ be arbitrary vertices.
    A routine verification confirms that if $\deg_{T}(a)=3$ (or $\deg_{T}(b)=3$), then $V(T) \cup \{a,b\}$ would span a larger tetragonal subgraph, a contradiction.
	Hence, we must have $\deg_{T}(a) = \deg_{T}(b) = 2$. 
    Recall from Lemma~\ref{lem: N} that $a,b$ have no neighbor in $N$. 
    Given that $\delta(G) \ge 3$, both $a$ and $b$ must have at least one neighbor in $M$.
    We now identify a forbidden triple as follows.
    Fix an arbitrary $(M,T)$-path $L$ (which must have length $\leq 2$) that connects $u$ to some $x\in V(T)$.
    Without loss of generality, assume $x \in A$. 
    Then there exists $y \in N_T(a)$ such that $xy\in E(\partial T)$.
    Let $v$ be an arbitrary vertex in the non-empty set $N_M(a)$.
    Note that $v \neq u$; otherwise, $u$ would be adjacent to both $x \in V(T)$ and $a \in R$, contradicting Lemma~\ref{lem:operation}~(4).
    Thus, we obtain the disjoint paths $L$ and $vay$, confirming that $(v, x, y)$ is a forbidden triple, a contradiction.
	This proves the claim.

    In view of the claim, we may assume without loss of generality that $R \subseteq A$.
	We proceed by discussing the order of $T$.
    
    For the case $m \le 4$, since $\delta_{2}(G) > k-1 \ge 5 \ge m+1$, it follows that at most one vertex in $V(T)\cap A$ has no neighbor in $M$.
    Recall from Lemma~\ref{lem:operation}~(4) that $N_{T^*}(u)$ is contained in either $V(T)$ or $R$. We consider the specific location of $N_{T^*}(u)$.
    
    If $N_{T^*}(u)$ is contained in $V(T) \cap B$ or $R$, then there exists an $(M,T)$-path $L$ connecting $u$ to some vertex $x \in V(T) \cap B$.
    Let $y \in V(T) \cap A$ be adjacent to $x$ on $\partial T$ such that $N_{M}(y) \neq \emptyset$, and select an arbitrary neighbor $v \in N_{M}(y)$.
    As established before, we must have $v \neq u$.
    Thus, the disjoint paths $L$ and $vy$ force $(v, x, y)$ to be a forbidden triple, a contradiction.

    Suppose instead that $N_{T^*}(u) \subseteq V(T) \cap A$. 
    Take any $x\in N_{T^*}(u)$.
    Since $E(M,R) \neq \emptyset$, there exist vertices $p \in R\subseteq A$ and $v \in N_M(p)\subseteq B$.
    Since $\deg_T(p) \ge m-1$, $p$ must be adjacent to some $y \in V(T) \cap B$ such that $xy\in E(\partial T)$.
    As established before, we have $v \neq u$, which derives disjoint $(M,T)$-paths $ux$ and $vpy$.
    Hence, $(v, x, y)$ is a forbidden triple, a contradiction.

    Finally, consider the case where $m \ge 5$.
    Since $E(M,R) \neq \emptyset$, we may select $v \in V(M) \setminus \{u\}$ such that either $u$ or $v$ is adjacent to some $r \in R$.
    By Observation~\ref{obs:valid} and Lemma~\ref{lem: B}, $\mathcal{P}^M_{u,v}$ contains $k - \deg_{T^*}(u) - 1 \ge k-3$ admissible paths.
    Since $\deg_T(r)\geq m-1$, there are disjoint $(M,T)$-paths that connects $\{u,v\}$ to vertices with distance on $\partial T$ at most two.
    By Lemma~\ref{lem:short dist}, $\mathcal{P}_{u,v}^{T}$ contains $m-1$ admissible paths. 
    The union of these paths produces $(k-3) + (m-1) - 1 \ge k$ admissible cycles, a contradiction.
	This completes the proof.
\end{proof}
Based on the previous analysis in the proof of Lemma~\ref{Q1:m>2}, we may assume that $G-T^*$ contains exactly one component, denoted by $M$, of order at least $3$. Furthermore, every vertex in $M$ has degree at most $\min\{2,m-2\}$ in $T^*$, and $E(M,R)=\emptyset$.
These facts allow us to impose further constraints on $N=G-T^*-M$, as stated in the following lemma.

\begin{lem}\label{lem:RN extra}
    If $G^\star$ does not contain $k$ admissible cycles, then one of the following holds:
    \begin{itemize}
        \item $G^\star$ is of Type I, and $R\cup V(N)\subseteq\{\theta\}$.
        \item $G^\star$ is of Type II, $m=k-1$, and $V(N)=R=\emptyset$.
        \item $G^\star$ is of Type II, $m=k$, $V(N)=\emptyset$, and $R\subseteq\{\theta_1,\theta_2\}$.
    \end{itemize}
    In particular, $R\cup V(N)$ is contained in one of the partite sets of $G$.
\end{lem}

\begin{proof}
    We first show that $R\cup V(N)$ is a subset of $\{\theta\}$ (for Type I) or $\{\theta_1,\theta_2\}$ (for Type II).
    It suffices to show that every $p\in R$ satisfies $\deg_{G}(p)\leq k-1$.
    By Lemma~\ref{lem: N}~(2) and Lemma~\ref{lem: edges between M and R}, for every $p\in R$, we have $N_{G}(p)\subseteq V(T)$, and thus $\deg_{G}(p)=\deg_T(p)\in \{m-1,m\}$.
    Consequently, the conclusion follows immediately when $m<k$.
    Now assume that $m=k$.
    By Lemma~\ref{lem:Q not complete}, $G[V(T)]$ is not isomorphic to $K_{k,k}$.
    It then follows from Lemma~\ref{lem:operation}~(1) that every $p\in R$ has $\deg_T(p)= k-1$.
    This confirms that $R\cup V(N)$ is included in $\{\theta\}$ or $\{\theta_1,\theta_2\}$.

    Next, consider the case where $G^\star$ is of Type II.
    By Lemma~\ref{lem: N}~(1), $N$ must be empty.
    Observe that every vertex $p\in R$ (if any) satisfies $k-1\leq \deg_{G}(p)=\deg_T(p)\leq m$.
    Consequently, if $R \neq \emptyset$, we must have $m\in \{k-1,k\}$.
    It remains to show that if $m=k-1$, then $R=\emptyset$.
    Suppose for a contradiction that $R\neq\emptyset$, then the inequality above becomes an equality, implying that every $p\in R$ has $\deg_T(p)=m=k-1$.
    By Lemma~\ref{lem:operation}~(1), we have $T\simeq K_{k-1,k-1}$.
    Let $u = u_M$. According to Observation~\ref{obs:valid} and Lemma~\ref{lem: B}, for every $v\in N_M(T)\setminus\{u\}$, $\mathcal{P}_{u,v}^M$ contains $k-\deg_{T^*}(u)-1\geq k-3$ admissible paths.
    Since $T\simeq K_{k-1,k-1}$, $\mathcal{P}_{u,v}^T$ contains $k-2$ admissible paths.
    The union of paths in $\mathcal{P}_{u,v}^M$ and $\mathcal{P}_{u,v}^T$ thus yields at least $(k-3)+(k-2)-1 = 2k-6 > k$ admissible cycles, a contradiction.

    The final claim, that $R\cup V(N)$ lies in a single partite set, is trivial when $G^\star$ is of Type I.
    When $G^\star$ is of Type II, we have $m\geq k-1 > 4$, and the claim follows directly from Lemma~\ref{lem:operation}~(3).
\end{proof}

Now we are ready to prove Lemma~\ref{Q1:m>2}.

\begin{proof}[\bf{Proof of Lemma~\ref{Q1:m>2}}]
    Suppose $G^\star$ is a counterexample.
    Let $u = u_{M}$, and let $\partial T=v_{0}v_{1}\ldots v_{2k-1}v_{0}$.
	By Observation~\ref{obs:valid} and Lemma~\ref{lem: B}, for any $v\in V(M)\setminus\{u\}$, the set $\mathcal{P}_{u,v}^{M}$ contains at least $k-\deg_{T^*}(u)-1\geq \max\{k-m+1,k-3\}$ admissible paths.
    Let $A$ and $B$ be the partite sets of $G$.
    By Lemma~\ref{lem:RN extra}, we may assume that $R\cup V(N)\subseteq A$.
    We distinguish cases based on the order of $T$.

    We begin with the case $m \le 4$. Since $k \ge 7$, we have $m \le k-3$. It follows from Lemma~\ref{lem:RN extra} that $|R \cup V(N)| \le 1$.
    Since $\delta_{2}(G) \ge k-1 \ge m+2$, it follows that all but at most one vertex in $T$ has a neighbor in $M$.
    Consequently, there exists a vertex $v \in V(M) \setminus \{u\}$ such that $u$ and $v$ are adjacent to consecutive vertices on $\partial T$.
    In this scenario, $\mathcal{P}_{u,v}^{M}$ contains $\max\{k-m+1, k-3\} = k-m+1$ admissible paths. 
    By Proposition~\ref{prop:Q-graph}, $\mathcal{P}_{u,v}^{T}$ contains $m$ admissible paths.
    Combining these path families yields $(k-m+1) + m - 1 = k$ admissible cycles, a contradiction.

    Next, assume that $5\leq m\leq k-1$.
    Fix an arbitrary neighbor $v_i \in N_T(u)$.
    We claim that there exists a vertex $v \in V(M) \setminus \{u\}$ adjacent to some $v_j \in V(T)$ such that $\mathrm{dist}_{\partial T}(v_i, v_j) \le 2$.
    
    To see this, first assume $G^\star$ is of Type I. Since $\delta_{2}(G) \ge k \ge m+1$ and $R\cup V(N)\subseteq A$, all vertices in $V(T) \cap A$, except possibly $\theta$, have neighbors in $M$.
    If $u \in A$, then $\{v_{i-1}, v_{i+1}\} \subseteq A$. Thus, at least one of them must have a neighbor $v \in V(M) \setminus \{u\}$, proving the claim.
    If $u \in B$, consider $\{v_{i-2}, v_{i+2}\} \subseteq A$. If $\theta \notin \{v_{i-2}, v_{i+2}\}$, then both vertices have neighbors in $M$. Since $\deg_T(u) \le 2$, at least one of them connects to a vertex $v \in V(M) \setminus \{u\}$, which suffices.
    Otherwise, assume without loss of generality that $v_{i-2} = \theta$. Then $v_{i+2}$ must have a neighbor in $M$.
    If this neighbor is distinct from $u$, then the claim follows. 
    Thus, we may assume that $v_{i+2}$ is adjacent to $u$. Repeating the argument for $v_{i+2}$, we observe that $v_{i+4} \neq \theta$ (since $m \ge 5$). Consequently, $v_{i+4}$ must have a neighbor $v \in V(M) \setminus \{u\}$, as desired.
    Now assume $G^\star$ is of Type II. By Lemma~\ref{lem:RN extra}~(2), $R = V(N) = \emptyset$.
    Since $\delta_{3}(G) = k \ge m+1$, at most two vertices in $T$ have no neighbors in $M$.
    If either $v_{i-1}$ or $v_{i+1}$ has a neighbor in $M$, the claim follows immediately.
    Otherwise, $v_{i-1}$ and $v_{i+1}$ are the only two vertices in $T$ with no neighbors in $M$.
    Consequently, both $v_{i-2}$ and $v_{i+2}$ must have neighbors in $M$.
    Since $\deg_T(u) \le 2$, $u$ cannot be adjacent to both $v_{i-2}$ and $v_{i+2}$. 
    Thus, one of them has a neighbor in $V(M) \setminus \{u\}$, proving the claim.

    With such a choice of $v$, $\mathcal{P}_{u,v}^{M}$ contains $k-3$ admissible paths. By Proposition~\ref{prop:Q-graph}, $\mathcal{P}_{u,v}^{T}$ contains $m-1$ admissible paths.
    Their union yields $(k-3) + (m-1) - 1 \ge k$ admissible cycles, a contradiction.

    Finally, assume that $m=k\geq 7$.
    Let $w$ be a vertex in $V(M)\setminus\{u\}$ such that $u,w$ connect to distinct vertices $v_i, v_j \in V(T)$, minimizing the distance $\mathrm{dist}_{\partial T}(v_i,v_j)$.
    Without loss of generality, assume $0\leq i < j \leq m$, so that $\mathrm{dist}_{\partial T}(v_i,v_j)=j-i$.
    By Observation~\ref{obs:valid} and Lemma~\ref{lem: B}, $\mathcal{P}_{u,w}^{M}$ contains at least $\max\{k-m+1, k-3\} = k-3$ admissible paths.
    If $j-i\leq 4$, Proposition~\ref{prop:Q-graph} implies that $\mathcal{P}_{u,w}^{T}$ contains $k-\mathrm{dist}_{\partial T}(v_i,v_j)+1\geq k-3$ admissible paths.
    The union of paths in $\mathcal{P}_{u,w}^{M}$ and $\mathcal{P}_{u,w}^{T}$ would then yield $(k-3)+(k-3)-1 = 2k-7 \geq k$ admissible cycles, a contradiction.

    Thus, we may assume $j-i \in [5, k]$. 
    By the minimality of $j-i$, each of $v_{i+1},v_{i+2},v_{i+3},v_{i+4}$ has no neighbor in $M$.
    We claim that at least one of these vertices has degree $k$ in $T$.
    Recall that $R \cup V(N) \subseteq A$, which implies that any $v \in \{v_{i+1},v_{i+2},v_{i+3},v_{i+4}\} \cap A$ satisfies $\deg_T(v) = \deg_G(v)$.
    Suppose the claim is false. Then both vertices in $\{v_{i+1}, \dots, v_{i+4}\} \cap A$ must have degree less than $k$ in $G$, which forces $G^\star$ to be of Type II, and $\{v_{i+1}, \dots, v_{i+4}\} \cap A = \{\theta_1, \theta_2\}$.
    It follows from Lemma~\ref{lem:RN extra} that $R=V(N)=\emptyset$, so the vertices in $\{v_{i+1}, \dots, v_{i+4}\} \cap B$ have degree $k$ in $T$.
    This contradiction proves the claim.

    According to the claim, assume that $\deg_{T}(v_{i+\ell})=k$ for some $\ell\in [4]$.
    Consider $(v_i,v_j)$-paths of the form $v_i v_{i-1} \dots v_{\alpha} v_{i+\ell} v_{\beta} v_{\beta+1} \dots v_{j}$, where indices satisfy $\alpha \in [j+1, i]$ and $\beta \in [i+\ell+1, j]$, with $\alpha$ and $\beta$ having different parity from $i+\ell$.
    A routine verification confirms that the lengths of these paths cover all integers of the appropriate parity in $[3, 2k-4]$.
    In particular, $\mathcal{P}_{u,w}^{T}$ contains $k-3$ admissible paths.
    Combining these with $\mathcal{P}_{u,w}^{M}$, we obtain $(k-3)+(k-3)-1 \geq k$ admissible cycles in $G$.
    This contradiction completes the proof.
\end{proof}

\subsection{The completion} 
\noindent We now complete the proofs of Theorems~\ref{thm:main bipartite} and~\ref{thm:main bipartite admis}; combined with Theorem~\ref{main}, the latter yields Theorem~\ref{thm:weak graph}.

\begin{proof}[\bf Proof of Theorem~\ref{thm:main bipartite admis}]
    The theorem follows directly from Lemma~\ref{lem: C4-free}, which handles the $C_4$-free case, and Lemmas~\ref{lem: m=2 I},~\ref{lem: m=2 II}, and~\ref{Q1:m>2}, which exhaust all possibilities for the optimal tetragonal subgraph $T$. This concludes the proof of Theorem~\ref{thm:weak graph} for the bipartite case.
\end{proof}

\begin{proof}[\bf Proof of Theorem~\ref{thm:main bipartite}]
    For $k\geq 7$, the conclusion follows immediately from Theorem~\ref{thm:main bipartite admis}, as a collection of $k$ admissible cycles in a bipartite graph covers all even residues modulo $k$.
    Now assume that $k=6$.
    By Theorem~\ref{Thm:32}, $G$ contains $\delta_3(G)-1\geq k-1=5$ admissible cycles, which cover all even residues modulo $6$.
    This completes the proof.
\end{proof}

Finally, we present the proof of Theorem~\ref{thm:exist admis}, which parallels the proof of Theorem~\ref{main 1}.

\begin{proof}[\bf Proof of Theorem~\ref{thm:exist admis}]
    Let $k \ge 7$. We claim that every 2-connected graph $G$ with $\delta_2(G) \ge k$ contains $k$ admissible cycles, unless $G$ is isomorphic to a graph in $\{K_{k+1}, K_{k,k}\} \cup \mathcal{H}_{k}$. By applying this claim to each end-block of the graph, we obtain Theorem~\ref{thm:exist admis}.

    To prove the claim, we first note that Corollary~\ref{cor: non-bipartite admis} and Theorem~\ref{thm:main bipartite admis} guarantee the existence of $k$ admissible cycles in every $k$-weak graph not isomorphic to a graph in $\{K_{k+1}, K_{k,k}\} \cup \mathcal{H}_{k}$. We then proceed via a reduction argument analogous to Section~4.1. Suppose $G$ does not contain $k$ admissible cycles. Then $G$ cannot be $k$-weak (of Type I) and thus admits a 2-cut $S=\{x,y\}$. 
    
    If $G-S$ contains two components of order at least two, then Lemma~\ref{lem:admis path} yields $(k-1) + (k-1) - 1 > k$ admissible cycles, which suffices. 
    Otherwise, a component of $G-S$ consists of a single vertex $z$. Hence, we may assume that $xy \notin E(G)$, as otherwise Lemma~\ref{lem:admis path} yields $(k-1) + 2 - 1 = k$ admissible cycles. 
    Since $\delta_2(G)\geq k$, $z$ is the unique vertex with degree less than $k$ in $G$.
    Hence, $S$ is the unique $2$-cut of $G$, and $G$ is a $k$-weak graph (of Type II), implying the existence of admissible cycles.
    This proves the claim.
\end{proof}

\section*{Acknowledgments}
\noindent This work is supported by National Key Research and Development Program of China 2023YFA1010201, National Natural Science Foundation of China grant 12125106, and Innovation Program for Quantum Science and Technology 2021ZD0302902.
\bibliographystyle{plain}
\bibliography{ref}

\vspace{1.0cm}

\indent{\it Email address}: \texttt{lyf619311271@mail.ustc.edu.cn}
\vspace{0.3cm}

\indent{\it Email address}: \texttt{jiema@ustc.edu.cn}
\vspace{0.3cm}

\indent{\it Email address}: \texttt{zyzhao2024@mail.ustc.edu.cn}

\end{document}